\documentclass[12pt]{amsart}

\usepackage[left=3cm,top=2.5cm,bottom=2.5cm,right=3cm]{geometry}
\usepackage{times}
\usepackage{amssymb}
\usepackage{mathtools}
\usepackage{wasysym}
\usepackage{graphicx,xspace}
\usepackage{epsfig}
\usepackage{paralist}
\usepackage{enumitem}
\usepackage[usenames,dvipsnames]{xcolor}
\usepackage{tikz}
\usepackage{gensymb}
\usepackage[T1]{fontenc}
\usepackage[utf8]{inputenc}
\usepackage{bm}
\usepackage{enumitem}
\usepackage[export]{adjustbox}
\usepackage[config, labelfont={normalsize}]{caption,subfig}
\usepackage{mathrsfs}
\definecolor{green}{RGB}{0,127,0}
\definecolor{red}{RGB}{191,0,0}
 \usepackage[colorlinks,cite color=red,link color=green,pagebackref=true]{hyperref}
\usepackage{todonotes}
\presetkeys{todonotes}{inline}{}
\usepackage{tabularx}

\usepackage{algorithm,algorithmicx,algpseudocode}
\algrenewcommand{\algorithmicrequire}{\textbf{Input:}}
\algrenewcommand{\algorithmicensure}{\textbf{Output:}}
\algnewcommand\And{\textbf{and} }
\algnewcommand\Or{\textbf{or} }
\usepackage[presets={abc,vec-cev}]{letterswitharrows}

\usetikzlibrary{matrix,arrows,calc}
\usetikzlibrary{decorations.markings}
\usetikzlibrary{patterns}
\usetikzlibrary{snakes}
\usetikzlibrary{shapes}

\usepackage[capitalize]{cleveref}
\usepackage{autonum}
\usepackage{marginnote}
\usepackage{siunitx}
\usepackage[record,abbreviations,symbols,stylemods={tree},style=treegroup]{glossaries-extra}

\theoremstyle{plain}

\newtheorem{lemma}{Lemma}[section]
\newtheorem{theorem}[lemma]{Theorem}
\newtheorem{corollary}[lemma]{Corollary}

\theoremstyle{remark}
\newtheorem{remark}{Remark}

\newtheorem{definition}[lemma]{Definition}

\newcommand{\T}{\mathcal{T}}

\newcommand{\V}{\mathbf{V}}

\newcommand{\N}{\mathbb{N}}

\newcommand{\CCC}{\mathcal{C}}

\newcommand{\h}{\mathfrak{h}}

\DeclareMathOperator{\Motz}{Motz}

\def\a{\alpha}

\def\r{r}

\newcommand{\cc}{\mathbf{c}}
\newcommand{\zz}{\mathbf{z}}

\renewcommand\d[1]{\frac{\partial}{\partial #1}}

\def\aa{\bm{a}}

\author[M.~Dołęga]{Maciej Dołęga}
\address{
Institute of Mathematics, 
Polish Academy of Sciences, 
ul. Śniadeckich 8, 
00-956 Warszawa, Poland}
\email{mdolega@impan.pl}

\author[M.~Lepoutre]{Mathias Lepoutre}
 \address{LIX, \'Ecole Polytechnique, 1 rue Honor\'e D' Estiennes
   D' Orves, 91120 Palaiseau, France}
\email{mathias.lepoutre@polytechnique.edu}

 \thanks{MD is supported by {\it Narodowe Centrum Nauki}, grant UMO-2017/26/D/ST1/00186.}


\title[Blossoming bijection for pointed maps of any surface]{Blossoming bijection for bipartite pointed maps \\ 
 and parametric rationality of general maps \\ 
 of any surface}

\begin{document}

\begin{abstract}
We construct an explicit bijection between bipartite pointed maps of
an arbitrary surface $\mathbb{S}$,
and specific unicellular blossoming maps of the same surface. Our
bijection gives access to the degrees of all the faces, and
distances from the pointed vertex in the initial map. The main construction
generalizes recent work of the second author which covered the case of
an orientable surface.

Our bijection gives rise to a first combinatorial proof of a parametric
rationality result concerning the bivariate generating series of maps
of a given surface with respect to their numbers of faces and
vertices. 
In particular, it provides a combinatorial explanation of
the structural difference between the aforementioned bivariate parametric
generating series in the case of orientable and non-orientable maps.
\end{abstract}

\maketitle

\section{Introduction}

\emph{Maps} are ubiquitous objects which have been studied by both mathematicians and
physicists for decades. Roughly speaking, a map is a polygonal tiling
of a surface $\mathbb{S}$, generalizing an important concept of triangulations. Rich structural properties of maps have multiple connections with various areas of
discrete mathematics, algebraic geometry, number theory,
statistical physics, and
most recently random geometries. This has pushed the study of maps forward in recent years, especially in the enumerative context (see e.g.~\cite{LandoZvonkin2004, Eynard:book} among others).

\subsection{Enumeration -- solving functional equations}

The study of enumerative properties of planar maps was initiated
by Tutte in the sixties \cite{Tutte1963}. He studied structural
properties of their generating function weighted by number of edges by building certain functional
equations satisfied by this function. In the aforementioned paper
Tutte guessed the solution of the equation that he cooked up, however
in the joint paper with Brown \cite{BrownTutte1964} they developed a
more systematic method of solving similar equations containing one \emph{catalytic variable}. This method turned out
to be very powerful and was substantially extended by Bousquet-M\'elou and Jehanne \cite{BousquetJehanne2006} to the much broader
context of enumerating various combinatorial objects. 
The first attempt to enumerate maps of higher genus was presented by
Lehman and Walsh~\cite{WalshLehman1972}, who extended Tutte's
decomposition algorithm and found an explicit recursive formula
for the number of orientable maps of genus $g$ with one face and $n$
edges. Similar decomposition methods were used by Bender and Canfield~\cite{BenderCanfield1986}, who
obtained asymptotic enumeration results in the case of the univariate generating function, that is
the generating function weighted by number of edges, that happened to be very similar in both the orientable and non-orientable cases:
\begin{theorem}[\cite{BenderCanfield1986}]
\label{theo:asymptotics}
For each $g\in\{\frac{1}{2},1,\frac{3}{2},2,\dots\}$, there exist
constants $t_g,p_g$ such that the number of rooted maps with $n$ edges
on the orientable (non-orientable, respectively) surface of genus $g$ satisfies:
\[\vec o_{g}(n) \sim t_g n^{\frac{5}{2}(g-1)} 12^n, \quad \vec{
  no_{g}}(n) \sim p_g n^{\frac{5}{2}(g-1)} 12^n, \text{ respectively.}\]
\end{theorem} 
Later on, Bender and Canfield refined their result in the case of
orientable surfaces to the following one:
\begin{theorem}[\cite{BenderCanfield1991}]
\label{theo:ratBC91}
For each $g \geq 0$, the generating series of orientable
maps of genus $g$ enumerated by the number of edges is a rational
function of $z$ and $\sqrt{1-12z}$.
\end{theorem} 

Next step was to study the
bivariate generating functions of orientable maps
enumerated by their number of vertices
and faces, directly generalizing the univariate enumeration, thanks to
Euler formula. The first result in this direction, obtained by Bender, Canfield and Richmond~\cite{BenderCanfieldRichmond1993}, was refined by Arqu\'es and Giorgetti \cite{ArquesGiorgetti1999}, and shortly afterwards extended to the case of non-orientable surfaces \cite{ArquesGiorgetti2000}.
\begin{theorem}[\cite{ArquesGiorgetti2000}]\label{thm:AG00}
Set $g\geq 1$.
The bivariate series of all maps (orientable or not) having genus $g$ can be written:
\[\frac{P_g(t_\bullet,t_\circ,a)}{a^{10g-6}}, \]
where $P_g$ is a polynomial of degree lower than or equal to $6g-3$,
and $t_\bullet$, $t_\circ$, and $a$ are defined by: 
\begin{equation}
    \begin{dcases}
        t_\bullet&=x+2t_\bullet t_\circ+t_\bullet^2\\
        t_\circ&=y+2t_\bullet t_\circ+t_\circ^2\\
        a&=\sqrt{(1-2(t_\bullet+t_\circ))^2-4t_\bullet t_\circ}
    \end{dcases}.
\end{equation}
Moreover, the bivariate series of orientable maps having genus $g$ can be written:
\[\frac{P'_g(t_\bullet,t_\circ)}{a^{10g-6}}, \]
where $P'_g$ is a polynomial of degree lower than or equal to $6g-3$.
\end{theorem}

Note that \cref{thm:AG00} raises a strong
structural difference between the generating functions of maps on 
orientable and non-orientable surfaces.

\medskip

We finish this section by making the following remark. All these remarkable
properties are obtained by rather
indirect methods. Indeed, the general proof strategy so far was to build
certain functional equations, and to use their particular
form to find a solution with nice
properties. 
On the other hand, these enumerative results give a strong indication that
enumeration of maps should be strictly related with enumeration of some decorated binary trees and Motzkin paths, which are very simple combinatorial objects. Uncovering
and understanding
this hidden connection would lead to a conceptual proof of the
aforementioned results, which naturally call for combinatorial
interpretations. The desire to obtain such a better understanding was
one of the reason that led to the rise of the bijective study of maps
-- nowadays a well-established domain on its own.

\subsection{Enumeration -- understanding the nature of enumerated objects}

The motivation behind the usage of bijective methods of enumerating maps is twofold. 
On the one hand, to explain the remarkable enumerative properties of maps by a direct, combinatorial
argument. On the other hand, to give an access to a better
understanding of their geometric nature. This second motivation turned
out to be a conceptual breakthrough leading to completely new areas
such as random geometry. These methods show that any map can be
constructed from trees with some additional decorations in a
systematic way. We can naturally divide these bijections into 
two different families of constructions: labeled and blossoming ones.

\subsubsection{Labeled objects}
The first bijective method of enumerating maps was due to Cori and
Vauquelin \cite{CoriVauquelin1981} who were able to recover the celebrated
formula of Tutte for the number of planar maps with $n$ edges. 
They constructed a bijection between these maps and certain labeled
trees, giving a bijective proof of Tutte's formula.
Their result also opened a way to study the global geometry of planar maps, initiated in the pioneering work of Chassaing and Schaeffer \cite{ChassaingSchaeffer2004}. 
Cori and Vauquelin's method was improved and extended to the case of higher genus
orientable surfaces by Chapuy, Marcus and Schaeffer
\cite{ChapuyMarcusSchaeffer2009}, hence providing a bijective proof of
the orientable part of \cref{theo:asymptotics}. Finally, Chapuy and
the first author \cite{ChapuyDolega2017} found a way to drop the assumption of orientability
in the Marcus-Schaeffer construction which resulted in a universal
bijection explaining similarities between orientable and non-orientable
enumeration from \cref{theo:asymptotics} (see also a reformulation and
refinement given by Bettinelli~\cite{Bettinelli2015})
.

\subsubsection{Blossoming objects}
Schaeffer also found a different bijection \cite{Schaeffer1997}
between Eulerian planar maps and so-called blossoming trees, that are
planar trees with some additional decorations. It turned out that this
bijection found many extensions, allowing enumeration of various
families of planar maps, such as simple triangulations \cite{PoulalhonSchaeffer2006}, or plane bipolar orientations \cite{FusyPoulalhonSchaeffer2009}. The blossoming bijections for planar maps were later unified into a single general scheme by Albenque and Poulalhon \cite{AlbenquePoulalhon2015}.

The first generalization of Schaeffer's initial construction
\cite{Schaeffer1997} for higher genus orientable maps was provided by the second author in a recent work \cite{Lepoutre2019}, which allowed him to obtain the first combinatorial interpretation of the rationality presented in \cref{theo:ratBC91}. 

In a subsequent work by the second author and Albenque \cite{Lepoutre:thesis}, they extended the previous work so as to track the number of vertices and faces of a map, and obtained a combinatorial interpretation of the bivariate rationality expressed in the orientable case of \cref{thm:AG00}.

The main \emph{bijective} result of the present work extends \cite{Lepoutre2019} in two different directions:
\begin{itemize}
\item maps are not required to be orientable anymore,
\item we are studying pointed maps, so that the blossoming objects bijectively encode metric properties of the initial map.
\end{itemize}

We now briefly introduce the notion of blossoming maps on general
surfaces in order to state our first main theorem (more
detailed definitions are presented in \cref{subsec:Blossoming}; the level
of generality of the objects defined here may slightly differ from definitions
presented in \cref{subsec:Blossoming} in order to keep the introduction as simple as possible).

A \emph{blossoming map} is a map with only one face which has
additional halfedges attached. These halfedges are decorated by
tiny arrows, which can be either ingoing (in this case the halfedge
is called a \emph{leaf}) or outgoing (in this case the halfedge
is called a \emph{bud}) and we require that a blossoming map has the
same number of leaves and buds. Note that a map of the sphere with only one face is necessarily a tree, which explains the origin
of this ``botanic'' terminology.

When we make a tour of the unique face
of a blossoming map starting from the distinguished \emph{root corner} and following its orientation, we successively
visit all the corners of this map. During
this tour we label each corner in the following way: the root
corner is labeled by $0$ and we iteratively increase/decrease by one or
leave unchanged the label of the next visited corner if it is preceded
by a bud/leaf or by an edge, respectively. Since a blossoming map has the same
number of leaves and buds, it implies that when we make the whole tour
of the root face and we reach
the root corner again, its label associated by the iterative
algorithm will agree with the initially chosen label (which was
$0$). Note also that each edge of a blossoming map will be
visited twice during this tour and since the algorithm does not change
the associated corner label when performing a tour along an edge it
makes sense to say that for the first visited ``side'' of an edge it will associate a
label $i$ and for the second visited ``side'' it will associate a
label $j$. We say that a blossoming map is \emph{well-blossoming} if
for each edge the label associated to its first visited ``side''
is bigger by one than the label associated to its second visited
``side''. We also say that a well-blossoming map is \emph{well-rooted}
if all the corner labels are nonnegative (in this case the root corner
is necessarily followed by a bud). Finally we say that a leaf or a bud is black/white if the minimum of
its adjacent labels is odd/even (with the convention that the bud
directly following the root corner is
black, if exists) and we say that a face is black/white if the minimum of
its labels is even/odd. We present below the picture of a
well-blossoming tree (that is a well-blossoming map of the sphere)
with the white root face (the root corner and its orientation is
indicated by a red arrow) and the color of buds/leaves indicated by the
color of their arrows.

\begin{figure}[h!!!]
	\adjincludegraphics[scale=0.8]{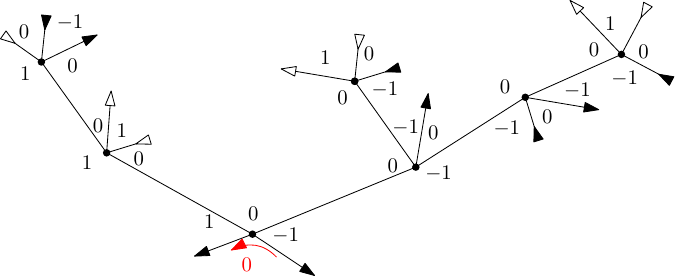}
\caption{A well-blossoming tree.}
\end{figure}

We are ready to state our main bijective theorem.

\begin{theorem}\label{thm:mainBijIntro}
Let $\mathbb{S}$ be a surface, and $n_\bullet, n_\circ, n_1, n_2, \cdots$ be integers with finite sum. 

There is an explicit bijection $\Phi_\ell$ between bipartite pointed
maps of $\mathbb{S}$ with
\begin{enumerate}
\item $n_\bullet$ black vertices,
\item $n_\circ$ white vertices,
  \item $n_k$
  faces of degree $2k$ (for any $k\in\N_{>0}$);
  \end{enumerate}
  and well-blossoming maps of $\mathbb{S}$ with
  \begin{enumerate}
\item the total number of black leaves and black faces equal to $n_\bullet$,
\item the total number of white leaves and white faces equal to $n_\circ$,
  \item $n_k$ vertices of degree $2k$ (for any $k\in\N_{>0}$).
  \end{enumerate}

Moreover, $\Phi_l(m)$ is well-rooted if and only if $m$ is a root-pointed map.
\end{theorem}

Note that a well-blossoming map has only one face so that the well-blossoming maps appearing in \cref{thm:mainBijIntro} have either $n_\bullet$ black
leaves and $n_\circ-1$ white
leaves, or $n_\bullet-1$ black
leaves and $n_\circ$ white
leaves depending on the color of the root face (and the latter case
corresponds to the black root face which is necessary for well-rooted
maps). If we now consider the case where $(n_k)_{k\in \N_{>0}}=(0,n,0,0\cdots)$, in conjunction with the classical \emph{radial construction}, we obtain the following corollary:

\begin{corollary}\label{cor:bij4Intro}
Let $\mathbb{S}$ be a surface, and $n_\bullet, n_\circ$ be integers.
There is an explicit constructive bijection between maps of $\mathbb{S}$ with $n_\bullet$ vertices and $n_\circ$ faces and well-rooted $4$-valent maps of $\mathbb{S}$ with $n_\bullet-1$ black leaves and $n_\circ$ white leaves.
\end{corollary}

Using strictly combinatorial methods developed in \cite{Lepoutre2019} and \cite{Lepoutre:thesis}, we use \cref{cor:bij4Intro} to give an interpretation of the rationality expressed in \cref{thm:AG00}.
Moreover, we provide an explanation of the aforementioned structural
difference between orientable and non-orientable maps. Finally, we
extend the methods of \cite{Lepoutre2019,Lepoutre:thesis} to the
non-orientable case to obtain additional deeper rationality results, refining
\cref{thm:AG00}, depending on certain properties of the maps, namely the form of the \emph{offset graph} of the \emph{scheme} of the opening of the maps.

Last, but not least, we hope that our bijection opens the doors to
study metric properties of some classes of random maps. Indeed, in
most cases random maps are studied by use of labeled objects, however
it was shown in \cite{Addario-BerryAlbenque2017,AlbenqueHoldenSun2020} that the blossoming approach turned out to be more natural in certain cases. We hope that our bijection might shed some light on possible extension of the aforementioned result.

\subsection{Organization of the paper}
\cref{sec:definitions} provides all the necessary definitions. In \cref{sec:bijection} we state and prove our main bijective
construction. \cref{sec:decomposition} is devoted to the
decomposition of the unicellular maps obtained previously into some smaller
pieces. This decomposition is studied in details in the case of
$4$-valent maps in \cref{sec:4valent}, which allows us to obtain enumerative results,
presented in \cref{sec:enumeration}. Following the suggestion of the
referee we added \cref{glossary}
which plays the role of a glossary. It contains the list of symbols and
numerous non-standard definitions used throughout the paper with
additional pictures. We suggest the
reader to print it out separately and use it as a handout while
reading the paper. This should
help in going through the technical parts of the paper.

\section{Definitions}
\label{sec:definitions}

\emph{A surface} $\mathbb{S}$ is a compact, connected, two dimensional real
manifold. By the classification theorem a surface $\mathbb{S}$ is
uniquely determined by an integer $\chi_{\mathbb{S}} \geq -2$ called
\emph{Euler characteristic} (or, equivalently, by nonnegative
half--integer $g_{\mathbb{S}} \in \frac{1}{2}\N$ called \emph{genus}
and given by $\chi_\mathbb{S} = 2-2 g_{\mathbb{S}}$) together with an
information whether $\mathbb{S}$ is orientable or not. Note that
a surface $\mathbb{S}$ with odd Euler characteristic cannot be orientable. An \emph{embedded graph} on a surface $\mathbb{S}$ is a proper (meaning that the edges are non-crossing) embedding of a graph on $\mathbb{S}$, considered up to homeomorphism. A \emph{face} of an embedded graph is a connected component of its complement in $\mathbb{S}$. An embedded graph is a \emph{map} if it is \emph{cellularly embedded}, meaning that all its faces are simply connected.
The set of maps of a surface $\mathbb{S}$ is denoted $\mathcal{M}_\mathbb{S}$.

Each edge of a map can be divided into two \emph{halfedges} by removing its middle point. 
A \emph{corner} is an adjacency between a vertex and a face, delimited by two halfedges. We denote $\mathbf{C}_m$ the set of corners of a map $m$. The vertex and face adjacent to a corner $\cc$ are denoted $v_m(\cc)$ and $f_m(\cc)$.

We arbitrarily associate to each vertex of $m$ an orientation that we call \emph{direct}. In particular, this defines a permutation called \emph{vertex rotation} and denoted $\sigma_m:\mathbf{C}_m\to \mathbf{C}_m$ that takes a corner $\cc$ and returns the next corner around $v_m(\cc)$ in direct order.
An \emph{oriented corner} $\vec{\cc}$ is a corner with an additional
orientation of $v_m(\cc)$. Alternatively, if a direct orientation is
fixed, it can be defined formally as a pair made of a corner $\cc$ and a \emph{spin} $s(\vec{\cc})$, which is $1$ if the orientation of $\vec\cc$ agrees with the direct orientation of the vertex $v_m(\cc)$, and $-1$ otherwise. 
The set of oriented corners of $m$ is denoted $\vec{\mathbf{C}}_m$.
The \emph{reversed corner} of $\vec{\cc}$, denoted $\cev{\cc}$ corresponds
to reversing the orientation of $v_m(\cc)$, and in terms of a direct
orientation it is defined by  $\cev{\cc}\coloneqq(\cc,-s(\vec{\cc}))$. 
The vertex rotation is naturally defined on the set of oriented
corners so that it is not necessary to use the notion of a direct orientation: $\vec \sigma_m(\vec\cc)$ is the next corner around $v_m(\cc)$ taken
with respect to the orientation of $\vec\cc$ and $\cev \sigma_m(\vec\cc)$
is a reversed corner of $\vec \sigma_m(\vec\cc)$. Alternatively, we can
use the notion of a direct
orientation and then the vertex rotation can be defined formally as: $\vec\sigma_m(\vec{\cc})\coloneqq(\sigma_m^{s(\vec{\cc})}(\cc),s(\vec{\cc}))$, and $\cev\sigma_m(\vec{\cc})\coloneqq(\sigma_m^{s(\vec{\cc})}(\cc),-s(\vec{\cc}))$. 
Note that $\cev\sigma_m$ is an involution.
The edge (halfedge, resp.) associated to an oriented corner
$\vec{\cc}$, denoted $e_m(\vec{\cc})$ ($h_m(\vec{\cc})$, resp.), is
the edge (halfedge, resp.) that separates $\vec{\cc}$ and $\vec\sigma_m(\vec{\cc})$.

The \emph{face rotation}, denoted $\vec\theta_m:\vec{\mathbf{C}}_m\to\vec{\mathbf{C}}_m$, 
is defined as follows: the next oriented corner after an oriented
corner $\vec{\cc}$ around $f_m(\cc)$ (i.e.~the first corner visited after $\vec{\cc}$ by
following the edge $e_m(\vec{\cc})$) is denoted $\vec\theta_m(\vec{\cc})$. We also denote $\cev\theta_m(\vec\cc)$ the reverse corner of $\vec\theta_m(\vec\cc)$.
Note that $\cev\theta_m$ is an involution.

\begin{remark}
Note that the initial choice of a local orientation around each vertex
(the direct orientation) is not part of the definition of a map, but
a gadget used to define $\sigma_m$, and $\theta_m$. 
\end{remark}

The maps that we consider are \emph{rooted}, that is they are equipped
with a distinguished oriented corner
$\vec{\rho}_m\in\vec{\mathbf{C}}_m$ called the \emph{root
  corner}. This choice of a distinguished corner allows to remove any
automorphisms that a map could have, thus allowing easier enumeration
and bijections.

Let $m$ be a map. 
We denote $\mathbf{V}_m$, $\mathbf{E}_m$ and $\mathbf{F}_m$ the set of vertices, edges and faces of $m$. We denote $n_m^v$, $n_m^e$ and $n_m^f$ the cardinal of these sets.
Any map $m$ of a surface $\mathbb{S}$ satisfies Euler formula:
\begin{equation}
    n_m^v-n_m^e+n_m^f=\chi_\mathbb{S}
\end{equation}
The univariate and bivariate generating series of maps of a surface $\mathbb{S}$ are defined by:
\begin{align}
    M_\mathbb{S}(z) &\coloneqq \sum_{m\in\mathcal{M}_\mathbb{S}} z^{n_m^e},\\
    M_\mathbb{S}(x,y) &\coloneqq \sum_{m\in\mathcal{M}_\mathbb{S}} x^{n_m^v}y^{n_m^f}.
\end{align}

A \emph{pointed map} is a rooted map with an additional distinguished vertex called $\emph{pointed vertex}$ and denoted $p_m$.
The families of pointed maps and their generating series are denoted by a $\bullet$. For instance, the generating series of pointed maps of a surface $\mathbb{S}$ is denoted $M_\mathbb{S}^\bullet$.
Note that $M_\mathbb{S}^\bullet(x,y)=\frac{x\cdot\partial M_\mathbb{S}(x,y)}{\partial x}$.
A \emph{root-pointed} map is a map $m$ such that $p_m=v_m(\rho_m)$.

\begin{figure}
\centering
\subfloat[]{
	\label{subfig:BipQuadrang}
	\includegraphics[width=0.47\linewidth]{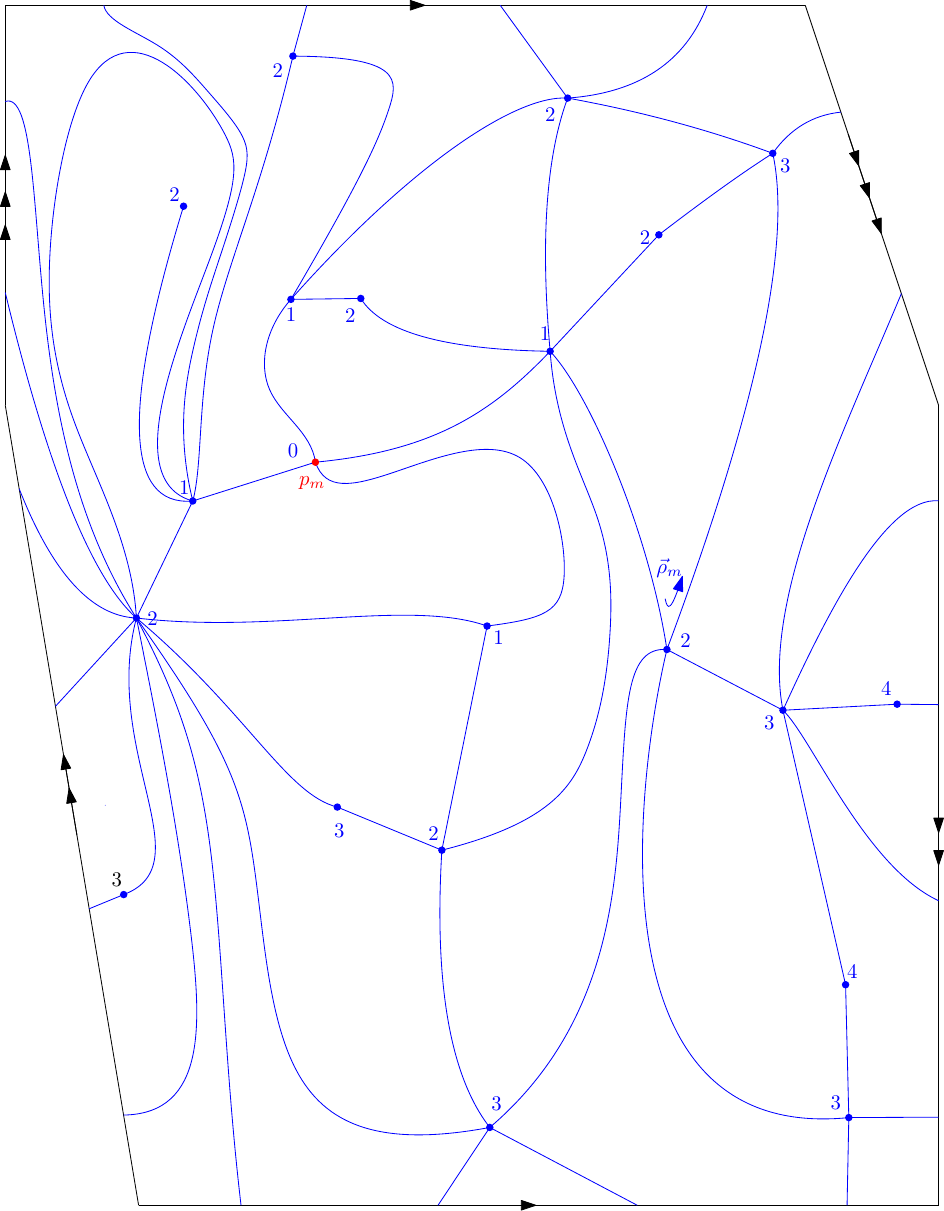}}
\quad
\subfloat[]{
	\label{subfig:Dual}
	\includegraphics[width=0.47\linewidth]{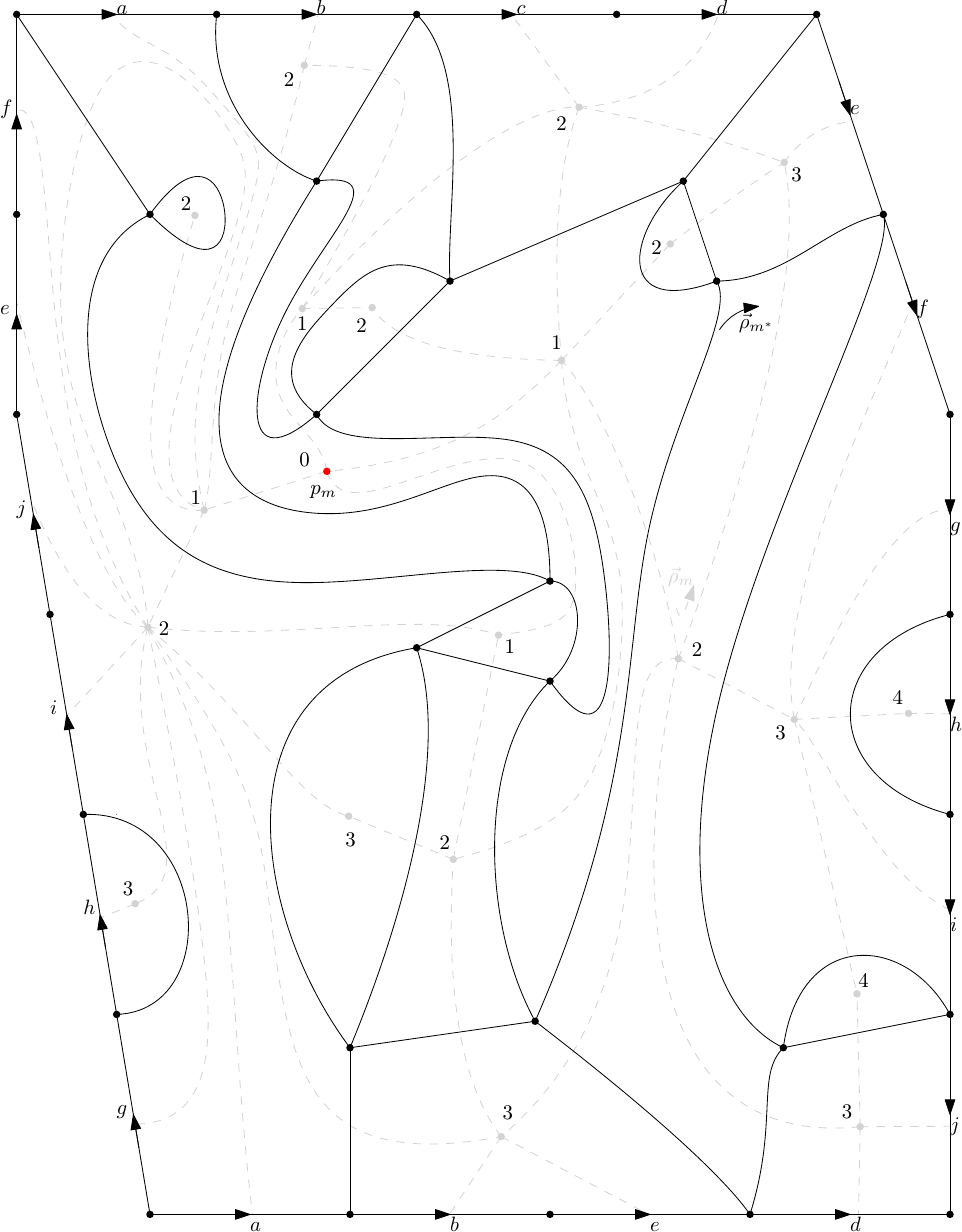}}
\caption{Figure \protect\subref{subfig:BipQuadrang} presents a
  bipartite map $m$ with a pointed vertex $p_m$ and root
  $\rho_m$. Labels of the vertices give the distance from the pointed
  vertex $p_m$. This map is embedded into the Klein bottle which is
  presented as the gluing of edges of the black hexagon in such a way that the
  direction of arrows is preserved by the gluing. Figure
  \protect\subref{subfig:Dual} figures the dual map $m^*$, which is bicolorable. Edges
  denoted by letters $a,\dots,j$ should be identified such that the
  direction of arrows is preserved, and with this identification $m^*$ is also embedded into the Klein bottle.}
\label{fig:exDual}
\end{figure}

The \emph{dual} of a map $m$ is the map $m^*$ defined on the same
surface with the same set of corners, whose vertices are the faces of
$m$, whose faces are the vertices of $m$, and such that the order of
corners around vertices and faces (with respect to an arbitrarily
chosen orientation around each vertex and around each face) corresponds to the order of corners around faces and vertices of $m$ (see \cref{fig:exDual}). 
The orientation of the root of $m$ and the root of $m^*$ are opposite, meaning that vertex $v_m(\vec\theta_m(\vec \rho_m))$ and face $f_{m^*}(\vec\sigma_{m^*}(\vec \rho_{m^*}))$ are dual one of each other.
Note that duality is an involution.

The degree of a vertex $v$ or a face $f$, denoted $\delta_m(v)$ or $\delta_m(f)$, is the number of corners adjacent to it.
A graph or a map is \emph{bipartite} if all its cycles have even length. 
In particular, all faces of a bipartite map have even degree.
Vertices of a bipartite map can be separated into two distinct sets: the \emph{black} (resp. \emph{white}) vertices are vertices at even (resp. odd) distance from the root. The \emph{vertex-color-weight} of a bipartite map $m$, denoted $\gamma_m^v$, is defined as the couple $(\gamma_m^{v\bullet},\gamma_m^{v\circ})$, where $\gamma_m^{v\bullet}$ (resp. $\gamma_m^{v\circ}$) is the number of black (resp. white) vertices of $m$.
A map is a \emph{quadrangulation} if all its faces have degree $4$.

A map is \emph{bicolorable} if its dual is bipartite. 
Note that a bicolorable map is also \emph{Eulerian}, which means that all its vertices have even degree. 
Although the two notions (bicolorable and Eulerian, or dually bipartite and even degree faces) are equivalent in the plane, this is not the case in other surfaces.
Similarly to vertices of a bipartite map, faces of a bicolorable map can be separated into black and white faces. The \emph{face-color-weight} of a bicolorable map $m$, denoted $\gamma_m^f$, is the vertex-color-weight of its dual: $\gamma_m^f \coloneqq \gamma_{m^*}^v$. 
A map is $4$-valent if all its vertices have degree $4$.
We denote $\mathcal{BP}_\mathbb{S}$ and $\mathcal{BC}_\mathbb{S}$ the sets of bipartite and bicolorable maps of a surface $\mathbb{S}$. 
We denote $\mathcal{BP}^\square_\mathbb{S}$ and $\mathcal{BC}^\times_\mathbb{S}$ the sets of bipartite quadrangulations and bicolorable $4$-valent maps of a surface $\mathbb{S}$. 

The \emph{face-weight} of a bipartite map $m$, denoted $\delta_m^f$, is the vector $(n_i)_{i\in\N_{>0}}$ such that for all $i$, $n_i$ is the number of faces of $m$ of degree $2i$.
Similarly, the \emph{vertex-weight} of a bicolorable map $m$, denoted $\delta_m^v$, is the vector $(n_i)_{i\in\N_{>0}}$ such that for all $i$, $n_i$ is the number of vertices of $m$ of degree $2i$.
The multivariate generating series of bipartite and bicolorable maps of a surface $\mathbb{S}$ are defined as follows:
\begin{align}
    BP_\mathbb{S}(\mathbf{z}) &\coloneqq \sum_{m\in\mathcal{BP}_\mathbb{S}} \mathbf{z}^{\delta_m^f}, \\
    BC_\mathbb{S}(\mathbf{z}) &\coloneqq \sum_{m\in\mathcal{BC}_\mathbb{S}} \mathbf{z}^{\delta_m^v},
\end{align}
with the conventions that $\mathbf{z}\coloneqq(z_i)_{i\in\N_{>0}}$ and ${(z_i)_{i\in\N_{>0}}}^{(n_i)_{i\in\N_{>0}}}\coloneqq \prod_{i \in \N_{>0}}z_i^{n_i}$.
The colored multivariate generating series of bipartite and bicolorable maps of a surface $\mathbb{S}$ are defined as follows:
\begin{align}
    BP_\mathbb{S}(\mathbf{z},x,y) &\coloneqq \sum_{m\in\mathcal{BP}_\mathbb{S}} \mathbf{z}^{\delta_m^f} x^{\gamma_m^{v\bullet}}y^{\gamma_m^{v\circ}}, \\
    BC_\mathbb{S}(\mathbf{z},x,y) &\coloneqq \sum_{m\in\mathcal{BC}_\mathbb{S}} \mathbf{z}^{\delta_m^v} x^{\gamma_m^{f\bullet}}y^{\gamma_m^{f\circ}}.
\end{align}
The bivariate generating series of bipartite quadrangulations and bicolorable $4$-valent maps are defined by:
\begin{align}
    BP^\square_\mathbb{S}(x,y) &\coloneqq
                                  BP_\mathbb{S}(\mathbf{z},x,y)\big|_{\mathbf{z}
                                  = (0,1,0,\dots)} = \sum_{m\in\mathcal{BP}^\square_\mathbb{S}} x^{\gamma_m^{v\bullet}}y^{\gamma_m^{v\circ}}, \\
    BC^\times_\mathbb{S}(x,y) &\coloneqq BC_\mathbb{S}(\mathbf{z},x,y)\big|_{\mathbf{z}
                                  = (0,1,0,\dots)} = \sum_{m\in\mathcal{BC}^\times_\mathbb{S}} x^{\gamma_m^{f\bullet}}y^{\gamma_m^{f\circ}}.
\end{align}

\section{Bijection}
\label{sec:bijection}
\subsection{Left-most geodesic tree}

Let $m$ be a map of a surface $\mathbb{S}$. A \emph{spanning tree} of
$m$ is the embedded graph $t$ on $\mathbb{S}$, obtained from $m$ by
keeping all its vertices but only a subset of edges of $m$ such that $t$ is connected and
acyclic. A spanning tree is naturally rooted at the oriented corner containing $\vec\rho_m$. 
Note that, except if $\mathbb{S}$ is a sphere, $t$ is not cellularly
embedded. However, $t$ can alternatively be seen as a planar map by
cutting the surface $\mathbb{S}$ around the edges of $t$, and gluing
the border of a part containing $t$ to a disk.

Let $t$ be a spanning tree of a pointed map $m$, and $v$ be a vertex of $m$.
To each vertex $v\in\mathbf{V}_m$, we associate a \emph{height}, denoted $\h_m(v)$, defined as the distance from $v$ to $p_m$.
By extension, the height of a corner $c$ is the height of $v_m(\cc)$.
Similarly, the \emph{height in $t$} of $v$ is the distance from $v$ to $p_m$ in $t$.
The spanning tree $t$ is called \emph{geodesic} if for any vertex $v$ of $m$, the height of $v$ and the height in $t$ of $v$ are the same. 

\begin{definition}[contour word]
Let $t$ be a spanning tree of a pointed map $m$.
The \emph{contour word} of $t$ is the sequence of heights of corners of $m$ in the tour of $t$ starting from $\vec \rho(m)$.
Because $t$ is spanning and acyclic, each corner is visited once by such a tour, so that the length of the contour word is equal to the number of corners of $m$.
\end{definition}

\begin{remark}\label{rem:unicityContourWord}
If $m$ is bipartite, the contour word of $t$ uniquely characterizes $t$. Indeed, starting a tour from $\vec{\rho}_m$, any time a yet-unvisited edge going to a yet-unvisited vertex is met, this vertex does not have the same height as the current vertex (because the map is bipartite), so that the contour word is enough to decide whether this edge should be in $t$ or not.
\end{remark}

\begin{definition}[leftmost geodesic tree]\label{def:LGT}
Let $m$ be a bipartite pointed map.
The \emph{leftmost geodesic tree} of $m$, denoted $t_\ell(m)$, is the
geodesic tree of $m$ which is maximal for the lexicographic order of the contour word.
\end{definition}

The uniqueness of the leftmost geodesic tree of a bipartite map is a
direct consequence of \cref{rem:unicityContourWord}. See \cref{fig:LT}

\begin{figure}
	\includegraphics[width=0.47\linewidth]{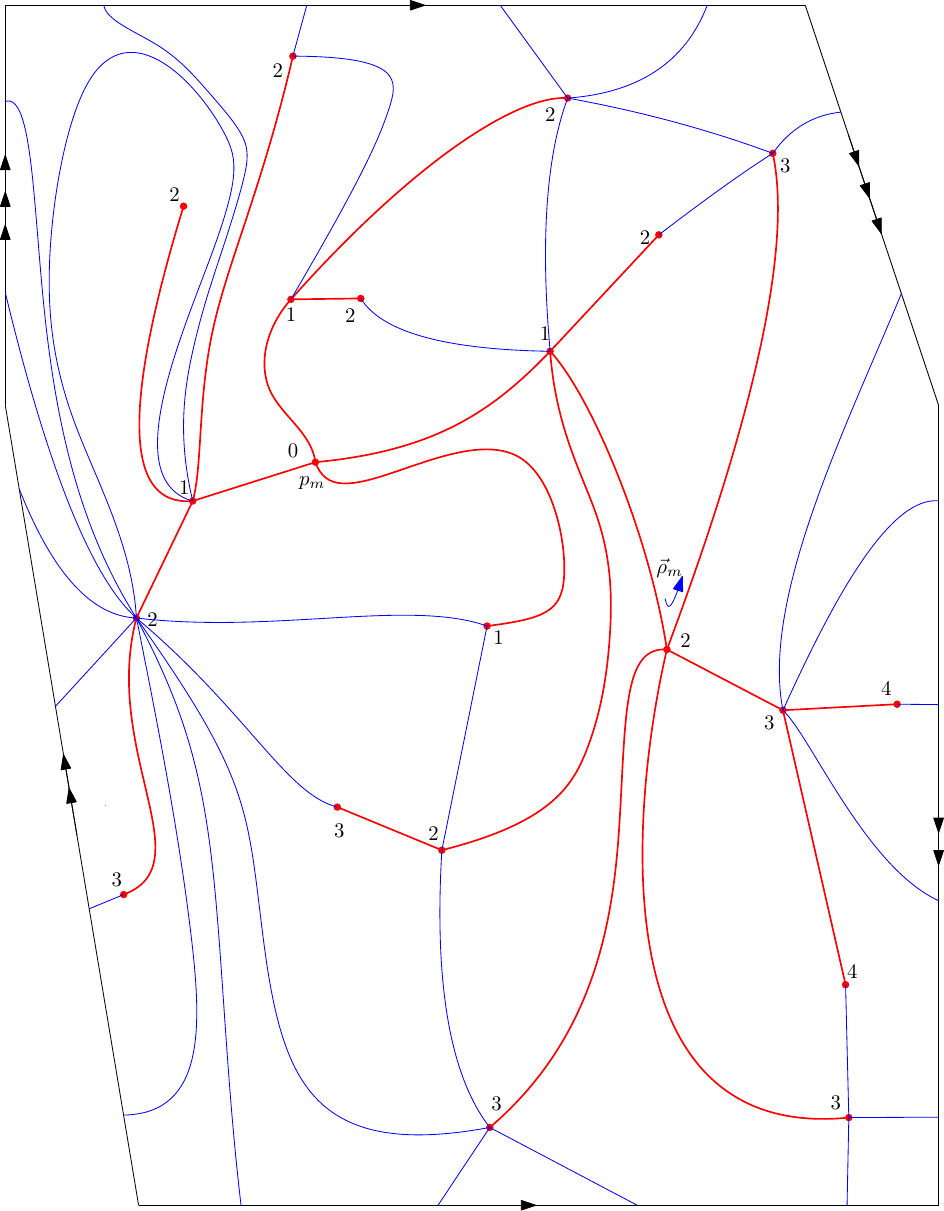}
\caption{Leftmost geodesic tree of the map $m^*$ from
  \cref{subfig:Dual} is depicted in red.}
\label{fig:LT}
\end{figure}

\begin{lemma}\label{lem:LGT}
The leftmost geodesic tree of a bipartite pointed map can be computed by \cref{alg:LGT}.
\end{lemma}

\begin{algorithm}[h]
\caption{the leftmost geodesic tree algorithm}
\label{alg:LGT}
\begin{algorithmic} 
\Require A bipartite pointed map $m$.
\Ensure The left-most geodesic tree of $m$, denoted $t$.
\State Set $t=\emptyset$, $\vec\cc=\vec\cc'=\vec \rho_m$, $m''=m'=m$, and $V=\emptyset$ ($V$ is the set of visited edges). 
\Repeat
	\State $e\leftarrow e_m(\vec\cc)$
	\If{$e\in V$}\Comment{$e$ has already been visited}
		\If{$e\in t$}
        	\State $\vec\cc\leftarrow \vec{\theta}_m(\vec\cc)$ 
        \Else 
        	\State $\vec\cc\leftarrow \vec{\sigma}_m(\vec\cc)$ 
        \EndIf
    \Else \Comment{$e$ is unvisited yet}
        \State $\vec {\cc'}\leftarrow \vec{\theta}_m(\vec\cc)$ 
        \State $m'\leftarrow m\setminus (V\setminus t)$
        \Comment{$m'$ is made of all vertices of $m$, as well as edges that can still be in $t$ at this point.}
        \State $m''\leftarrow m'\setminus \{e\}$
        \Comment{$m''$ is made of edges that can be in $t$ if the edge $e$ is discarded.}
    	\If{(1) $\h_m({\cc'})=\h_m(\cc)+1$ \Or (2) $m''$ is disconnected \Or (3) $\h_{m''}( \cc)\neq \h_m( \cc)$} 
        	\State add $e$ to $t$ 
        \EndIf
        \State add $e$ to $V$ 
	\EndIf
\Until {$\vec\cc=\vec \rho_m$}
\Return $t$
\end{algorithmic}
\end{algorithm}

\begin{proof}
Let $m$ be a bipartite pointed map, and $t$ be the embedded graph
computed by \cref{alg:LGT}. 
We are going to prove that $t$ is the leftmost geodesic tree of $m$ by
successively proving following steps:
\begin{enumerate}[label=(\Alph*)]
\item \label{A} \cref{alg:LGT} terminates and performs a tour of the root face of $t$, starting from $\vec \rho_m$.
\item \label{B} graph $t$ is a tree;
\item \label{C} graph $t$ is spanning and geodesic;
\item \label{D} graph $t$ is the leftmost geodesic tree of $m$.
\end{enumerate}

In order to prove \ref{A} we notice that at each step of the algorithm, we
are updating a corner $\vec\cc$. Let $\vec\cc_i$ denotes the corner $\vec\cc$
at $i$-th step of the algorithm. We start from the proof that
the algorithm terminates. 
In other terms we have to prove that $\vec\cc_{i+1} \notin \{\vec\cc_2,\dots,\vec\cc_i\}$ for each $i > 1$.
Suppose that this is not the case, and let $i$ be the minimal index for which $\vec\cc_{i+1} \in \{\vec\cc_2,\dots,\vec\cc_i\}$. This
means that there exists $1\leq j<i$ such that $\vec\cc_{i+1} = \vec
\cc_{j+1} =: \vec\cc$, but $\vec\cc_{i} \neq \vec
\cc_{j}$. Then $\{\vec\cc_i,\vec\cc_j\}=\{\vec{\sigma}^{-1}_m(\vec\cc),\vec{\theta}^{-1}_m(\vec\cc)\}$. In both cases, the two equalities imply respectively that $e_m(\cev{\cc})\notin t$ and $e_m(\cev{\cc})\in t$, which is a contradiction.
By consequence the algorithm terminates and performs a tour of the root face of $t$, starting from $\vec \rho_m$.

We are going to prove \ref{B}. Suppose that $t$ is not acyclic, and denote $C$ the set of edges that are part of a cycle of $t$, along with their incident vertices. Let $v$ be a vertex of $C$ with maximum height. Since $m$ is bipartite, all vertices sharing an edge of $C$ with $v$ have height strictly smaller than $v$. 
    
    Let $\vec\cc_0$ be the first corner incident to $v$ explored by the algorithm, and $(\vec\cc_i)_{i\in[1,k]}$ be the sequence of explored corners of $v$ such that $e_m(\vec\cc_i)\in t$, in the same order as they are encountered by the algorithm. 
    Suppose $e_m(\vec\cc_i)\notin C$. Then the part of $t$ explored from $\vec\cc_i$ is necessarily a tree, by definition of $C$. By consequence, at this step, the algorithm will just do the tour of this treelike part, so that no corner of $v$ is explored between $\vec\cc_i$ and $\vec{\sigma}_m(\vec\cc_i)$. Let $j$ be the smallest $i$ such that $e_m(\vec\cc_i)\in C$. 
    
    Consider the step of the algorithm when $\vec\cc=\vec\cc_j$. Since
    $v\in C$, we have $k\geq 2$, so that $\vec{\sigma}_m(\vec\cc_j)\neq \vec
    \cc_0$. By consequence, the edge $e\coloneqq e_m(\vec\cc)$ (as
    defined in the algorithm) is not explored yet. Thus Condition
    $(1)$ is not satisfied by maximality of the label of $v$ in
    $C$. Since $v$ has at least $2$ incident edges belonging to $C$ (and these edges do not belong to $V\setminus t$),
    and since all vertices adjacent to $v$ through edges from $C$ are labeled by $\h_m(v)-1$ by maximality of the label of $v$, then condition $(3)$ is not
    satisfied either. Additionally, these adjacent vertices are connected to
    the pointed vertex by paths of $t$ not going through $v$ (because $v$ has
    maximal height), so that condition $(2)$ is not satisfied either.
    This leads to a contradiction. 

Now we want to show \ref{C}. We recall that \cref{alg:LGT} performs
a simple cycle on
the subset $\mathbf{C}_t$ of 
corners of $m$, which follows edges of a tree $t$ (starting from $\vec
\rho_m$). This means that at the last step of the algorithm all the
edges incident to $\mathbf{V}_t$ are visited. Suppose that $t$ is not
spanning. That means that there exists a vertex $v \in
\mathbf{V}_m\setminus \mathbf{V}_t$. In particular, this means that
the map $m'' = m \setminus (V \setminus t)$ is disconnected, because
vertex $v$ is in a different connected component than
$\mathbf{V}_t$. Since a map $m''$ is connected in the first step of
the algorithm, this means that there exists the smallest step $i$,
when we added an edge to $V$ but not to $t$ and we made a map $m''$
disconnected. This is a contradiction with condition $(2)$ from the
algorithm. Similarly, condition $(3)$ enforces that at the end of the
algorithm the distance in $t$ from the vertex $p_m$ to an arbitrary
vertex $v$ coincides with its height $\h_m(v)$. Indeed, it is enough to
notice that at each step of the algorithm the height in the map $m'$
coincides with the distance from the pointed vertex $p$. Since at the
last step of the algorithm the map $m'$ coincides with  $t$, the claim follows.

Finally, we are ready to conclude \ref{D}. Suppose $t$ it is not the leftmost geodesic spanning tree:
      $t_\ell(m)\neq t$. Let $\vec\cc_0$ be the last common corner
      between $t_\ell(m)$ and $t$. Consider the step of the algorithm
      when $\cc=\cc_0$. By \cref{rem:unicityContourWord,def:LGT}, either 
\begin{enumerate}[label=(\alph*)]
\item
$\h_m(\vec\cc')=\h_m(\vec\cc)+1$ and $e\in t_\ell(m)$ but $e\notin t$, or
\label{Da}
\item
$\h_m(\vec\cc')=\h_m(\vec\cc)-1$ and $e\notin t_\ell(m)$ but $e\in t$. 
\label{Db}
\end{enumerate}
Case \ref{Da} is impossible because of condition $(1)$. Suppose we are in case \ref{Db}. Then either condition $(2)$ or $(3)$ is satisfied, which is a contradiction with the fact that $t_\ell(m)\subset m''$, while $t_\ell(m)$ is a geodesic spanning tree.

This concludes the proof.
\end{proof}

\subsection{Blossoming unicellular maps}
\label{subsec:Blossoming}

A map $m$ is \emph{unicellular} if it has a unique face. Note that a
tree (considered as a map) is unicellular.
We call \emph{tour corners} the oriented corners that can be written as $\vec{\theta}_m^i(\vec \rho_m)$ for some $i\in\mathbb{N}_{>0}$. 
We denote ${}^t\vec{\mathbf{C}}_m$ the set of tour corners of
$m$. Note that for any oriented corner $\vec\cc$, the set of tour
corners contains either $\vec\cc$ or $\cev{\cc}$, but not both. 
This is a way of choosing a unique orientation for each corner of a unicellular map. 
The \emph{tour index} of a tour corner $\vec\cc\in{}^t\vec{\mathbf{C}}_m$ is the minimum $i\in\mathbb{N}_{>0}$ such that $\vec\cc=\vec{\theta}_m^i(\vec \rho_m)$. 
The \emph{tour order starting from the root} (or root order), denoted $\preccurlyeq_m$, is the total order corresponding to the tour index. 
Note that the root corner is the maximum corner for the tour order.
The \emph{tour order} is the corresponding cyclic order.

A \emph{blossoming} map $m$ is a map with additional single (meaning that they are not matched to another halfedge to form an edge) halfedges, called \emph{stems}. We denote $\vec{\mathbf{C}}^s_m$ the set of oriented corners that are followed by a stem, and  $\vec{\mathbf{C}}^e_m$ the set of oriented corners that are followed by an edge. The function $e_m$ is not defined on $\vec{\mathbf{C}}^s_m$. The definition of the vertex rotation is not impacted. Since a stem is not matched to another halfedge, we consider that the face rotation of a corner of $\vec{\mathbf{C}}^s_m$ is equal to its vertex rotation.

The \emph{interior map} of a blossoming map $m$, denoted $m^\circ$, is the (non-blossoming) map obtained from $m$ by removing all its stems.
The \emph{interior degree} of a vertex $v$, denoted $\delta_m^\circ(v)$, is the degree of $v$ in $m^\circ$.
Note that the degree of a vertex is equal to the sum of its interior
degree and the number of stems adjacent to $v$.


In this paper, stems are oriented. Outgoing stems are called
\emph{buds}, while ingoing stems are called \emph{leaves}. We denote
$\vec{\mathbf{C}}^{\uparrow}_m$ (resp. $\vec{\mathbf{C}}^{\downarrow}_m$) the set of
oriented corners of $m$ that are followed by a bud (resp. by a
leaf). From this point on, we require that a blossoming map has as
many buds as leaves. However, sometimes we will let a blossoming map
to have different number of buds then leaves and in this situation we
will call it \emph{unbalanced blossoming map}.

The \emph{corner labeling} of a unicellular blossoming map $m$ is the unique function $\lambda_m:{\mathbf{C}}_m\to \mathbb{Z}$ that satisfies the following properties:
\begin{equation}\label{cndn:cornerLabeling}
\left\{
\begin{array}{lcll}
    \lambda_m(\vec\rho_m)&=&0,&\\
    \lambda_m(\vec\theta_m(\vec\cc))&=&\lambda_m(\vec\cc), &\text{ if }\vec\cc \in {}^t\vec{\mathbf{C}}^e_m,\\
    \lambda_m(\vec\theta_m(\vec\cc))&=&\lambda_m(\vec\cc)+1, &\text{ if }\vec\cc \in {}^t\vec{\mathbf{C}}^{\uparrow}_m,\\
    \lambda_m(\vec\theta_m(\vec\cc))&=&\lambda_m(\vec\cc)-1, &\text{ if }\vec\cc \in {}^t\vec{\mathbf{C}}^{\downarrow}_m.
\end{array}
\right.
\end{equation}
Note that the uniqueness of the previous definition comes from the fact that it can be computed by a tour of the unique face starting from the root.

A corner labeling is a \emph{well-labeling} if it satisfies the following additional properties:
\begin{equation}\label{cndn:wellLabeling}
\left\{
\begin{array}{lcll}
   \lambda_m(\sigma_m(\cc))&=&\lambda_m(\cc)-1, &\text{ if } \cc \in \mathbf{C}^e_m \text{ and } \cc \preccurlyeq_m \sigma_m(\cc),\\
   \lambda_m(\sigma_m(\cc))&=&\lambda_m(\cc)+1, &\text{ if } \cc \in \mathbf{C}^e_m \text{ and } \sigma_m(\cc) \preccurlyeq_m \cc.
\end{array}
\right.
\end{equation}

A \emph{well-blossoming} map is a unicellular blossoming map whose corner labeling is a well-labeling. 
We denote $\mathcal{B}_\mathbb{S}$ the set of well-blossoming maps of a surface $\mathbb{S}$.

A \emph{bud-rooted} map is a well-blossoming map $m$ such that $\vec\rho_m\in{}^t\vec{\mathbf{C}}_m^{\uparrow}$. 
We also say that such a map is \emph{rooted on a bud}.
A corner labeling is \emph{non-negative} if all labels are non-negative. 
A \emph{well-rooted} map is a well-blossoming map whose corner labeling is non-negative.
Note that a well-rooted map is necessarily bud-rooted. 
We denote $\mathcal{R}_\mathbb{S}$ the set of well-rooted maps
of a surface $\mathbb{S}$.

A stem is said to be \emph{black} (resp. \emph{white}) if the smaller
label adjacent to it is odd (resp. even) with the convention that the
root bud is always black.
The face is said to be \emph{black} (resp. \emph{white}) if its
minimum label is even (resp. odd). A stem of a bud-rooted map is \emph{rootable} if it is either the root
bud or a leaf. 

The \emph{face-color-weight} of a map $m\in\mathcal{B}_\mathbb{S}$,
denoted $\gamma_m^f$, is defined as the couple
$(\gamma_m^{f\bullet},\gamma_m^{f\circ})$, where $\gamma_m^{f\bullet}$
(resp. $\gamma_m^{f\circ}$) is the number of black (resp. white)
leaves plus the number of black (resp. white) faces of $m$. Note that
the latter is $0$ or $1$ since $m$ has only one face.

The \emph{rootable-stem-color-weight} of a map $m\in\mathcal{B}_\mathbb{S}$, denoted $\gamma_m^r$, is defined as the couple $(\gamma_m^{r\bullet},\gamma_m^{r\circ})$, where $\gamma_m^{r\bullet}$ (resp. $\gamma_m^{r\circ}$) is the number of black (resp. white) rootable stems of $m$.

The colored multivariate generating series of $\mathcal{B}_\mathbb{S}$, and $\mathcal{R}_\mathbb{S}$ are defined by:
\begin{align}
    B_\mathbb{S}(\mathbf{z},x,y) &\coloneqq \sum_{m\in\mathcal{B}_\mathbb{S}} \mathbf{z}^{\delta_m^v} x^{\gamma_m^{f\bullet}}y^{\gamma_m^{f\circ}},\\
    R_\mathbb{S}(\mathbf{z},x,y) &\coloneqq \sum_{m\in\mathcal{R}_\mathbb{S}} \mathbf{z}^{\delta_m^v} x^{\gamma_m^{r\bullet}}y^{\gamma_m^{r\circ}}.
\end{align}

The restriction of $\mathcal{B}_\mathbb{S}$, and $\mathcal{R}_\mathbb{S}$ to $4$-valent maps are denoted $\mathcal{B}^\times_\mathbb{S}$, and $\mathcal{R}^\times_\mathbb{S}$. Their bivariate generating function are defined by:
\begin{equation}
        F^\times_\mathbb{S}(x,y) \coloneqq F_\mathbb{S}(\mathbf{z},x,y)\big|_{\mathbf{z}
                                  = (0,1,0,\dots)},
\end{equation}
where $F$ is either $B$, or $R$.

\begin{remark}
  \label{rem:colors}
We use the face-color-weight for the generating series $B_\mathbb{S}(\mathbf{z},x,y)$ because the opening (which is part of
our main construction; see~\cref {subsec:opening}) is a weight-preserving bijection, where the
weight of the pointed vertex becomes the weight of the face. In
particular it guarantees that the generating function
$B_\mathbb{S}(\mathbf{z},x,y)$ is equal to the generating function of
pointed bipartite maps $BP^\bullet_\mathbb{S}(\mathbf{z},x,y)$.
Note that a
well-rooted map $m \in\mathcal{R}_\mathbb{S}$ always has a black root
face so that the rootable-stem-color-weight coincides with the
face-color-weight in this case (and also explains the convention that the root bud
is black). In the following sections we will prove certain rationality properties
of the generating series $R^\times_\mathbb{S}$ and one of the
important tool is a rerooting operation described
in~\cref{sec:reroot}. The reason we use the rootable-stem-color-weight for the
generating series $R_\mathbb{S}$ is that it is
easier to control it than the face-color-weight when analyzing the rerooting operation.
\end{remark}

\subsection{Opening of a map along a spanning tree}
\label{subsec:opening}

\begin{figure}
\centering
\subfloat[]{
	\label{subfig:DualAndLT}
	\includegraphics[width=0.47\linewidth]{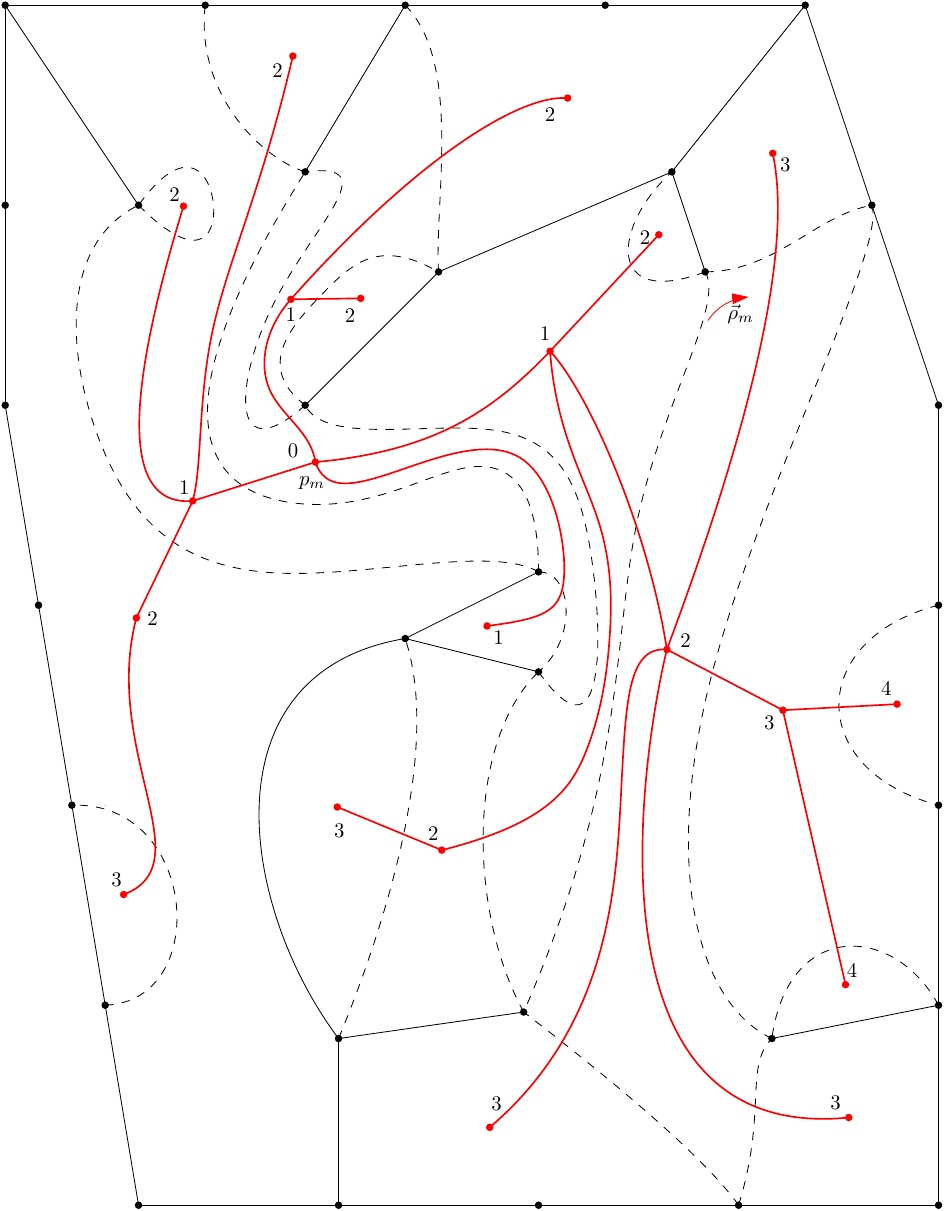}}
\quad
\subfloat[]{
	\label{subfig:Blossoming}
	\includegraphics[width=0.47\linewidth]{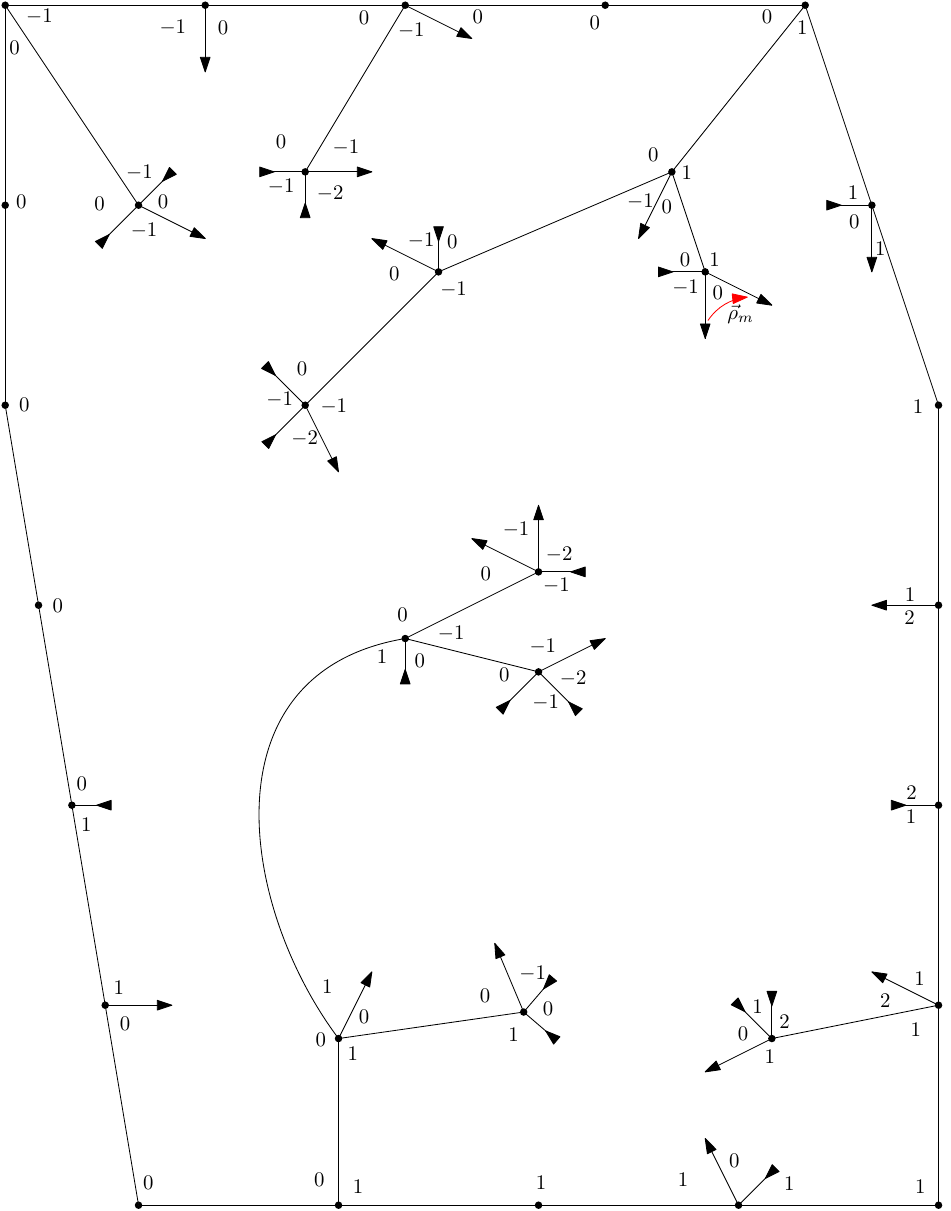}}
\caption{\protect\subref{subfig:DualAndLT}: the map represented in \cref{subfig:Dual}, along with the leftmost geodesic tree of its dual.
\protect\subref{subfig:Blossoming}: the opening of the dual of the previous map.}
\label{fig:Opening}
\end{figure}

Let $m$ be a bipartite pointed map, and $t$ a spanning tree of $m$. 
The \emph{opening of $m$ along $t$}, denoted $\Phi(m,t)$, is the blossoming map obtained from $m^*$, the dual of $m$, by \cref{alg:opening}. 
Furthermore, we call \emph{opening of $m$} and denote
$\Phi_\ell(m)\coloneqq\Phi(m,t_\ell(m))$, the opening of $m$ along its
leftmost geodesic tree. See \cref{fig:Opening} which shows how to
obtain $\Phi_\ell(m)$ from the map $m$ from
\cref{subfig:BipQuadrang}. 

\begin{algorithm}[h]
\caption{the opening algorithm.}
\label{alg:opening}
\begin{algorithmic} 
\Require A pointed map $m$ and a spanning tree $t$.
\Ensure The opening of $m$ along $t$, denoted $o\coloneqq\Phi(m,t)$.
\State Set $o=m^*$, the dual of $m$, and denote $t^*$ the dual of $t$
\State $\vec\cc\leftarrow \vec \rho_{m^*}$
\Repeat
	\If{$\vec\cc \in \vec{\mathbf{C}}^e_o$}\Comment{$\vec\cc$ is followed by an edge in $o$}
	\State set $e\leftarrow e_o(\vec\cc)$
		\If{$(1)$ $e\notin t^*$} 
        	\State $\vec\cc\leftarrow \vec{\theta}_o(\vec\cc)$ \Comment{we keep $e$ and continue the tour}
        \Else \Comment{we disconnect the edge so as to create $2$ stems}
            \State remove $e$ from $o$
            \If{$(2)$ $\h_m(v_m(\vec{\sigma}_{m^*} (\vec\cc)))=\h_m(v_m(\vec\cc))+1$} 
                \State add a bud in $o$ after $\vec\cc$ 
                \State add a leaf in $o$ before $\vec{\theta}_{m^*}(\vec\cc)$
            \Else 
                \State add a leaf in $o$ after $\vec\cc$
                \State add a bud in $o$ before $\vec{\theta}_{m^*}(\vec\cc)$
            \EndIf
        \EndIf
    \Else\Comment{$\vec\cc$ is followed by a stem in $o$}
        \State $\vec\cc \leftarrow \vec{\theta}_o(\vec\cc)$ \Comment{we continue the tour}
	\EndIf
\Until {$\vec\cc=\vec \rho_o$} 
\State \Return $o$
\end{algorithmic}
\end{algorithm}

\begin{lemma}\label{lem:opening}
  Let $m \in \mathcal{BP}^\bullet_\mathbb{S}$ be a bipartite pointed map of a surface $\mathbb{S}$. Then
  \begin{enumerate}
    \item the opening of $m$, $o\coloneqq\Phi_\ell(m)$, is a
      well-blossoming map of $\mathbb{S}$;
      \item the face-weight of $m$ is equal to the vertex-weight of
        $o$: $\delta_m^f=\delta_o^v$;
        \item the vertex-color-weight of $m$ is equal to the
          face-color-weight of $o$: $\gamma_m^v=(\gamma_o^{f\bullet},\gamma_o^{f\circ})$.
            \end{enumerate}

Moreover, if $m$ is root-pointed, then $o$ is well-rooted. 
\end{lemma}

\begin{proof}
Let $m$ be a bipartite pointed map of a surface $\mathbb{S}$, $t\coloneqq t_\ell(m)$ be the leftmost geodesic tree of $m$, and $u\coloneqq\Phi_\ell(m)$ be the opening of $m$.  

\begin{itemize}
    \item Similarly to \cref{lem:LGT}, we can prove that this algorithm terminates and performs a simple cycle on corners of $m^*$, which follows edges not in $t^*$. 
    Furthermore, condition $(1)$ ensures that this cycle is exactly the dual of the cycle performed in \cref{alg:LGT}.
    By consequence, all corners are visited and the resulting map, which is the complement of the dual of a tree, is unicellular and cellularly embedded in $\mathbb{S}$.
    
    \item To each corner $\vec\cc$ of $o$ we associate the label $\h_m(v_m(\vec\cc))-\h_m(v_m(\vec \rho_m))$. Condition $(2)$ ensures that this labeling satisfies \cref{cndn:cornerLabeling}, so that it coincides with the corner labeling of $o$. 

    \item Condition $(1)$ of \cref{alg:LGT} ensures that any edge $e_m(\vec\cc)$ not in $t$ satisfies  $\h_m(\theta_m(\vec\cc))=\h_m(\vec\cc)-1$, so that the corner labeling of $o$ satisfies \cref{cndn:wellLabeling} and is hence a well-labeling.
    
    \item If $m$ is root-pointed, then $\h_m(v_m(\vec \rho_m)) = 0$. Thus for any $v\in\mathbf{V}_m$ label $\h_m(v)-\h_m(v_m(\vec \rho_m)) = \h_m(v)$ is non-negative, so that $o$ is well-rooted.
    
    \item The degree of the vertices of $m^*$ and $o$ are the same. Every time we remove a black (resp. white) face out of $o$ during the algorithm, we add a black (resp. white) leaf into~$o$.
\end{itemize}
\end{proof}

\subsection{Closing a unicellular blossoming map}
\label{subsec:closing}

Let $u\in\mathcal{B}_\mathbb{S}$ be a well-blossoming map of a surface $\mathbb{S}$. The \emph{closure} of $u$, denoted $\Psi(u)$, is the bipartite pointed map of $\mathbb{S}$ obtained by the following procedure:
\begin{enumerate}
    \item match stems by pairs bud/leaf in the following recursive way: match a bud with a leaf if there is no unmatched stem between the bud and the leaf, for the cyclic tour order.
    \item merge each pair of stems into a single edge.
    \item all corners of minimum label are now in the same face; mark this face.
    \item take the dual of the map so obtained; the pointed vertex is the dual of the previously marked face.
\end{enumerate}

The matching procedure matches all stems. 
Furthermore, the order in which matchings are done does not affect subsequent matchings. 
This (in addition to step $(3)$) can be easily seen in the following alternative way of deciding the stem matching:
\begin{enumerate}
    \item represent the sequence of stems in tour order starting from the root by a unidimensional walk: a bud is represented by an upstep and a leaf by a downstep.
    \item a bud corresponding to an upstep $b$ going from height $i$ to height $i+1$ is matched with the leaf corresponding either to the first downstep after $b$ that goes from height $i+1$ to height $i$ (if it exists) or either to the first downstep of the walk that goes from height $i+1$ to height $i$, otherwise (the existence of such a step comes from the fact that there are as many buds as leaves).
\end{enumerate}

\begin{lemma}\label{lem:closure}
Let $u\in\mathcal{B}_\mathbb{S}$ be a well-blossoming map of a surface
$\mathbb{S}$. Then
\begin{enumerate}
  \item the closure of $u$, $m\coloneqq\Psi(u)$, is a bipartite
    pointed map of $\mathbb{S}$;
    \item the face-weight of $m$ is equal to the vertex-weight of $u$:
      $\delta_m^f=\delta_u^v$;
      \item the vertex-color-weight of $m$ is equal to the
        face-color-weight of $u$: $\gamma_m^v=(\gamma_u^{f\bullet},
          \gamma_u^{f\circ})$.
        \end{enumerate}
        
Moreover, if $u$ is well-rooted, then $m$ is root-pointed. 
\end{lemma}

\begin{proof}
Let $u\in\mathcal{B}_\mathbb{S}$ be a well-blossoming map of a surface $\mathbb{S}$, and $m\coloneqq\Psi(u)$ be the closure of $u$.
The vertex-weight of $u$ is not modified when merging stems, so that it is equal to the face-weight of $m$.

Note (recursively) that if we merge pairs of stems in the order in which they are decided in step $(1)$, the labels of the corners of the root face satisfy \cref{cndn:cornerLabeling}, while for any other face $f$, the labels of corners of $f$ are all the same. We label in $m$ the dual vertex of $f$ by this common label, minus the minimum of all corner labels of $u$.

Because $u$ is well-blossoming, the labels of any two adjacent vertices of $m$ differ by $1$, which implies that $m$ is bicolorable, and that the labeling of any vertex is smaller or equal to its height.
Additionally, let $v \in \mathbf{V}_m$ be a vertex of $m$ different from $p_m$. The corresponding face was created by merging together $2$ stems into an edge $e$. The neighbour of $v$ in $m$ along $e$ has label $1$ less than $v$. Iteratively, this ensures the existence of a path from $v$ to $\rho_m$ of length equal to the label of $v$. 
This suffices to prove that that the labeling of $m$ actually coincide with the height labeling.
    
As a consequence, merging black (resp. white) bud with a leaf creates
a black (resp. white) face in $m$. Furthermore, if $u$ is well-rooted,
$\vec \rho_u$ ends up in the face with minimum label, so that $m$ is root-pointed.
\end{proof}

\subsection{Opening and closure are inverse bijections}

\begin{lemma}\label{lem:inverse1}
Let $u$ be a well-blossoming map. Then the opening of the closure of $u$ is $u$ itself: $\Phi_\ell(\Psi(u))=u$.
\end{lemma}

\begin{proof}
Let $u\in\mathcal{B}_\mathbb{S}$ be a well-blossoming map of a surface $\mathbb{S}$, and $m\coloneqq\Psi(u)$ be the closure of $u$.

We denote $s$ the set of edges of $m$ that are dual to those created from $u$ by merging $2$ stems. Since $u$ is unicellular, $s$ is a spanning tree of $m$. Since the height and the height in $s$ of vertices of $m$ coincide (see the proof of \cref{lem:closure}), $s$ is geodesic. 
Furthermore, the sequence of corners in a tour of $s$ in $m$ corresponds exactly to a tour of the face of $u$.

Suppose that $s$ is not the leftmost geodesic tree of $m$. This means there exists a corner $\vec\cc$ such that the tours of $s$ and $t_\ell(m)$ coincide up to $\vec\cc$, but differ on the next edge (meaning in particular that the other extremity of the next edge comes later in the tour). Either of the following occurs:
\begin{itemize}
    \item $\h_m(\theta_m(\vec\cc))=\h_m(\vec\cc)+1$, and $e_m(\vec\cc)\in
      t_\ell(m)$, but $\vec\cc\in\vec{\mathbf{C}}_u^e$. This implies that $\h_u(\vec\sigma_u(\vec\cc))=\h_u(\vec\cc)+1$, which is a contradiction with the fact that the corner labeling of $u$ is a well-labeling.
    \item $\h_m(\theta_m(\vec\cc))=\h_m(\vec\cc)-1$, and $e_m(\vec\cc)\notin t_\ell(m)$, but $\vec\cc\in\vec{\mathbf{C}}_u^s$. Note that these conditions actually imply that $\vec\cc\in\vec{\mathbf{C}}_u^{\downarrow}$.
    The merging of the leaf $l$ following $\vec\cc$ with its matched bud $b$ divides $u$ into two faces, a face $f_\rho$ that contains the root, and a face $f_0$ that contains the lowest label.
    Since $t_\ell(m)$ is spanning, it contains an edge $e$ going from $f_\rho$ to $f_0$. 
    Let $\vec{\cc'}$ be the first corner adjacent to $e$ in the tour of $t_\ell(m)$; since $\vec{\cc'}\in f_\rho$ and \cref{lem:LGT} coincides with the tour of $u$ up to $\vec\cc$, $\vec{\cc'}$ comes after $b$ in the tour of $u$.
    Since $t_\ell(m)$ is geodesic, we have $\h_m(\theta(\vec{\cc'}))=\h_m(v_m(\cc'))-1$.
    By consequence, the higher-label side of the edge dual to $e$ is visited after its other side in the tour of $u$, which is a contradiction with $u$ being well-blossoming.
\end{itemize}

By consequence, $s$ is the leftmost geodesic tree of $m$ and the tour of the face of $u$ corresponds to the tour of corners of $m$ performed in \cref{alg:opening}. Condition $(1)$ of \cref{alg:opening} implies immediately that the interior of $u$ corresponds to the interior of $\Phi_\ell(m)$. Since the opening and closing algorithms also preserve the labels of corners, and the orientation of stems only depends of their adjacent labels, this is enough to conclude that $u=\Phi_\ell(m)$.
\end{proof}

\begin{lemma}\label{lem:inverse2}
Let $m$ be a bipartite pointed map. Then the closure of the opening of $m$ is $m$ itself: $\Psi(\Phi_\ell(u))=u$.
\end{lemma}
\begin{proof}
Let $m$ be a bipartite pointed map, and $u\coloneqq\Phi_\ell(m)$ its opening. Condition $(1)$ of \cref{alg:opening} ensures that the tour of corners in \cref{alg:opening} corresponds to the tour of the face of $u$. Additionally, the labels of these corners remain the same throughout the opening.

The matching of stems can be made in reverse order of creation in \cref{alg:opening}. Indeed, consider a pair of stems $(s_1,s_2)$ created by \cref{alg:opening} recursively ($s_1$ is created at the current position in the tour, and $s_2$ corresponds to the other extremity of the former edge) and suppose that all subsequently created pairs of stems can be matched together in the closure algorithm. If condition $(2)$ is satisfied at the creation of $(s_1,s_2)$, then $s_1$ is a bud and $s_2$ a leaf, and after the merging of already performed matchings, $s_2$ comes directly after $s_1$ in tour order, so that they are matched together by the closure.
Now suppose that condition $(2)$ is not satisfied. This means that in
the leftmost geodesic tree of $m$, the pointed vertex and the root are
not on the same side of the edge corresponding to $(s_1,s_2)$. In
particular, this means that all the corners between $s_2$ and $s_1$ in
tour order have label bigger than $i$ (where the labels surrounding $s_1$ are $i,i+1$). By consequence, $s_1$ and $s_2$ are matched in the closure. 
\end{proof}

\cref{lem:opening,lem:closure,lem:inverse1,lem:inverse2} imply the
following theorem, stated in the introduction as
\cref{thm:mainBijIntro}:

\begin{theorem}\label{thm:mainBij}
Let $\mathbb{S}$ be a surface, and $n_\bullet, n_\circ, n_1, n_2, \cdots$ be integers with finite sum. 

The opening $\Phi_\ell$ is a bijection between bipartite pointed
maps of $\mathbb{S}$ with
\begin{enumerate}
\item $n_\bullet$ black vertices,
\item $n_\circ$ white vertices,
  \item $n_k$
  faces of degree $2k$ (for any $k\in\N_{>0}$);
  \end{enumerate}
  and well-blossoming maps of $\mathbb{S}$ with
  \begin{enumerate}
\item the total number of black leaves and black faces equal to $n_\bullet$,
\item the total number of white leaves and white faces equal to $n_\circ$,
  \item $n_k$ vertices of degree $2k$ (for any $k\in\N_{>0}$).
  \end{enumerate}

Moreover, $\Phi_\ell(m) \in \mathcal{R}_\mathbb{S}$ if and only if $m$ is a root-pointed map.

This implies the following equalities:
\begin{align}
     BP^\bullet_\mathbb{S}(\mathbf{z},x,y)&=B_\mathbb{S}(\mathbf{z},x,y), \\
     BP_\mathbb{S}(\mathbf{z},x,y)&=R_\mathbb{S}(\mathbf{z},x,y).
\end{align}
\end{theorem}

\subsection{An opening bijection for general maps}

\cref{cor:bij4} is a direct corollary of \cref{thm:mainBij}, in the special case where $n_k=0$ for all $k\neq 2$.

\begin{corollary}\label{cor:bij4}
Let $\mathbb{S}$ be a surface, and $n_\bullet, n_\circ$ be integers. 

The opening $\Phi_\ell$ is a bijection between bipartite root-pointed quadrangulations of $\mathbb{S}$ with $n_\bullet$ black vertices and $n_\circ$ white vertices, 
and well-rooted $4$-valent maps of $\mathbb{S}$ with $n_\bullet-1$ black leaves and $n_\circ$ white leaves.

This implies the following equality:

\begin{equation}
     BP^\square_\mathbb{S}(x,y)=R^\times_\mathbb{S}(x,y).
\end{equation}
\end{corollary}

We now describe a well-known bijection on maps, attributed to Tutte, and sometimes called quadrangulated map.

Let $m$ be a map of a surface $\mathbb{S}$. Then the \emph{quadrangulated map} of $m$, denoted $q(m)$, is the map of $\mathbb{S}$ whose vertices are the faces and vertices of $m$, and whose edges are in a bijection with corners of $m$ in the following way: the edge of $q(m)$ corresponding to a corner $c$ of $m$ connects the vertices of $q(m)$ corresponding to $v_m(\cc)$ and $f_m(\cc)$. 

\begin{theorem}\label{thm:TutteBij}
Let $\mathbb{S}$ be a surface, and $n_\bullet, n_\circ$ be integers.
Tutte's operation is a bijection between maps of $\mathbb{S}$ with $n_\bullet$ vertices and $n_\circ$ faces and bipartite quadrangulations of $\mathbb{S}$ with $n_\bullet$ black vertices and $n_\circ$ white vertices, so that:
\begin{equation}
    M_\mathbb{S}(x,y)=BP_\mathbb{S}^\square(x,y).
\end{equation}
\end{theorem}

This bijection is a classical result; see \cite{Lepoutre:thesis} for a short proof discussing the colored weight.
Therefore \cref{thm:TutteBij} applied to \cref{cor:bij4} implies the following corollary, denoted \cref{cor:bij4Intro} in the introduction:

\begin{theorem}\label{thm:bij4}
Let $\mathbb{S}$ be a surface, and $n_\bullet, n_\circ$ be integers.
There is an explicit constructive bijection between maps of $\mathbb{S}$ with $n_\bullet$ vertices and $n_\circ$ faces and well-rooted $4$-valent maps of $\mathbb{S}$ with $n_\bullet-1$ black leaves and $n_\circ$ white leaves.

This implies: 
\begin{equation}
    M_\mathbb{S}(x,y)=R_\mathbb{S}^\times(x,y).
\end{equation}
\end{theorem}

\section{Decomposition of a unicellular map}
\label{sec:decomposition}

We now proceed to the study of maps obtained by the opening algorithm. In the present section, in a very similar manner to what was done in previous work by the second author \cite{Lepoutre2019,Lepoutre:thesis}, we successively decompose maps into \emph{cores} and \emph{schemes}, following the work of Chapuy, Marcus and Schaeffer \cite{ChapuyMarcusSchaeffer2009}, while proceeding to two successive \emph{rerootings}. 

Well-rooted maps of $\mathcal{R}_\mathbb{S}$ can be \emph{rerooted} so as to obtain bud-rooted maps. These can in turn be \emph{pruned}, so as to obtain bud-rooted cores, which can be rerooted into scheme-rooted cores of $\mathcal{C}_\mathbb{S}$. These $3$ operations can be merged into a single one, whose enumerative consequence is given in \cref{lem:shortcutEnum}. A scheme-rooted core of $\mathcal{C}_\mathbb{S}$ can be reduced to a labeled scheme of $\mathcal{L}_\mathbb{S}$, which in turn can be reduced to an unlabeled scheme of $\mathcal{U}_\mathbb{S}$. 

In \cref{sec:coreScheme}, we introduce the core and scheme of a unicellular blossoming map. In \cref{sec:reroot} we describe the rerooting operation. In \cref{sec:shortcut} we state the enumerative consequence of the complete pruning procedure.

\subsection{Core, scheme}\label{sec:coreScheme}

\subsubsection{Interior core, interior scheme}

Let $\mathbb{S}$ be a surface. A \emph{core} of $\mathbb{S}$ is a unicellular map of $\mathbb{S}$ with no vertex of interior degree $1$. A \emph{scheme} of $\mathbb{S}$ is a core of $\mathbb{S}$ with no vertex of interior degree $2$ (nor $1$). 

Let $m$ be a (non-blossoming) unicellular map of a surface $\mathbb{S}$.
The \emph{interior core of $m$} is the core obtained from $m$ by iteratively removing all vertices of degree $1$ (and their adjacent edge). 
Note that if $\mathbb{S}$ is a sphere, then the core of $m$ is empty.
A map can be seen as a core on which are attached trees.
By consequence, the operation of iteratively removing vertices of degree $1$ is also called \emph{pruning}: retrieving the core just consists in cutting these trees.

Similarly, if $c$ is a core, the \emph{interior scheme of $c$} is the scheme obtained from $c$ by iteratively removing vertices of interior degree $2$, and merging their $2$ formerly adjacent edges. 
A core can be seen as a scheme whose edges have been replaced by \emph{branches}: sequences of vertices of interior degree $2$.
A vertex of $c$ is a \emph{scheme vertex} if it is also a vertex of the scheme of $c$.
Note that a core in a projective plane is essentially a cycle, so that the scheme construction is not defined in genus less than $1$.

\subsubsection{Blossoming core}

\begin{figure}
\centering
\subfloat[]{
	\label{subfig:WRMap}
	\includegraphics[page=2,width=0.4\linewidth]{Blossoming}}
\qquad
\subfloat[]{
	\label{subfig:itsCore}
	\includegraphics[width=0.4\linewidth]{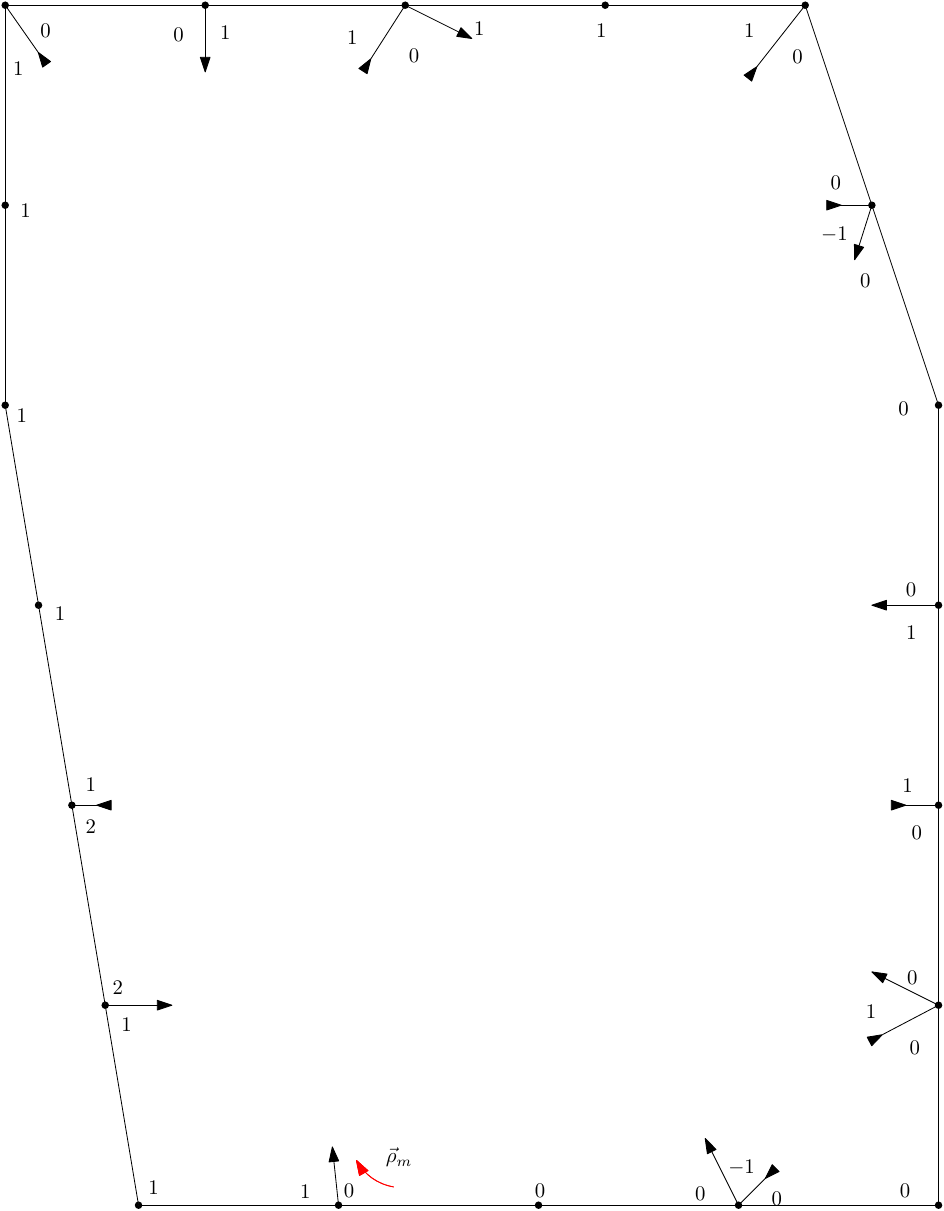}}
\caption{\protect\subref{subfig:WRMap}: A well-rooted map $m$. \protect\subref{subfig:itsCore}: The pruning of $m$.}
\label{fig:Core}
\end{figure}

In order to be able to deal with blossoming maps, the previous operations have to be modified accordingly. 
Let $m\in\mathcal{B}_\mathbb{S}$ be a well-blossoming map of $\mathbb{S}$.
The map $m$ is made of the core of its interior map, onto which are attached some stems and some trees. 
The \emph{core} of $m$, denoted $c(m)$, is obtained by the pruning process, which consists in replacing each of these trees by a leaf if the tree does not contain the root, and by a bud otherwise.
This amounts to cutting the edge between the core and the tree and
keeping the dangling halfedge, with adequate orientation. One can check that for every such
a tree which does not contain the root the number of leaves attached to it is larger by one than the number of
buds attached to it. Similarly, when this tree contains the root, the number of leaves attached to it is smaller by one than the number of
buds attached to it. This is a direct
consequence of the fact that $m$ is well-blossoming (the labels of two
corners attached to $c(m)$ and separated by such a tree
differ by one, and the corner visited first in the tour order has larger label). This implies that
the number of leaves is
equal to the number of buds in $c(m)$.
If $\rho_m$ is a corner of $c(m)$, then it is also set to be the root
of $c(m)$ (note that in this case the labels of the corners that belong both to $m$ and to $c(m)$ are the same in $m$ and in $c(m)$).
Otherwise, the root of $c(m)$ is set to be the corner directly
preceding the bud created by removing the tree of $m$ that contains
$\rho_m$. Note that in this case the label of any corner that belongs
both to $m$ and to $c(m)$ is shifted in $c(m)$ by the label of
the corner of $m$ that directly precedes the tree of $m$ that contains
$\rho_m$. This implies that $c(m)$ is a well-blossoming map and all its vertices have interior degree at least equal to $2$. 

We denote $\mathcal{T}$ the set of trees that can be cut from a core in the pruning process, so as to obtain a leaf. 
These trees can be empty, in case the original map already had a leaf on its core.
Otherwise, they are bud-rooted trees that can be described as follows: they are plane trees with $\deg(v)-2$ buds attached to each internal vertex $v$, and whose leaf vertices are replaced by leaf stems.

We denote $T_\bullet(\zz,x,y)$ the generating series of trees of $\mathcal{T}$ with the weight it would have if it had been removed from a map so as to create a black leaf in the core. Note for instance that the weight of the empty tree in $T_\bullet(\zz,x,y)$ is $x$.
The generating series $T_\circ(\zz,x,y)$ is defined by
$T_\circ(\zz,x,y)\coloneqq T_\bullet(\zz,y,x)$. We define
$T_\cdot^{1}:= T_\cdot(\zz,x,y), T_\cdot^{-1}:= T_\cdot(\zz,y,x),$
where $\cdot \in \{\bullet,\circ\}$.

\begin{lemma}
The series $T_\circ$ and $T_\bullet$ are related by the following equations: 

\begin{equation}
	\begin{cases}
    	T_{\bullet} &= x+\sum\limits_{k\geq 1}\left(z_{k}\cdot\sum\limits_{i_1+\cdots+i_k \leq k-1} \quad \prod\limits_{1 \leq j \leq k}T_{\bullet}^{(-1)^{j-1+\sum\limits_{1 \leq l \leq j}i_l}}\right),\\
    	T_{\circ} &= y+\sum\limits_{k\geq 1}\left(z_{k}\cdot\sum\limits_{i_1+\cdots+i_k \leq
          k-1}\quad \prod\limits_{1 \leq j \leq k}T_{\circ}^{(-1)^{j-1+\sum\limits_{1 \leq l \leq
          j}i_l}}\right).
    \end{cases}
\end{equation}
\label{lem:GFForTrees}
\end{lemma}

\begin{proof}
Indeed, let $t \in \T\setminus\{\emptyset\}$ be a non-empty tree. Let $2k$, with $k\in\mathbb{N}_{>0}$, denote the degree of its root vertex $v$; this means that the interior degree of $v$ is equal to $k$, and that $v$ has $k-1$ buds other than the root bud attached to it. Let $i_j$ denotes the number of buds between the $j-1$-th and the $j$-th edge attached to $v$. Then $i_1+\cdots+i_k \leq k-1$, and the value of this sum uniquely determines the value of $i_{k+1}$. Moreover, let $c_i$ be the label at the end of the tour of the $i$-th branch attached to $v$. Then the parities of $c_j$ and $c_{j+1}$ are the same if and only if $i_{j+1}$ is odd. This analysis explains the form of the generating series for $T_\bullet$ and $T_\circ$.
\end{proof}

Because of \cref{cor:bij4} we are especially interested in enumerating 4-valent blossoming trees, whose generating function is given by 
\[ T^{\times}_\bullet(x,y) \coloneqq
  T_\bullet(0,1,0,\dots;x,y).\]
The series $T^{\times}_\circ$ and $T^{\times}_\bullet$ are related by the following equations, obtained by specializing \cref{lem:GFForTrees}: 
\begin{equation}
\begin{cases}
    T^{\times}_\bullet &= x+(T^{\times}_\bullet)^2 +2T^{\times}_\circ T^{\times}_\bullet,\\
    T^{\times}_\circ   &= y+(T^{\times}_\circ)^2 +2T^{\times}_\circ T^{\times}_\bullet.
\end{cases}
\end{equation}

\subsubsection{Blossoming scheme} 
Let $l$ be an unlabeled blossoming map. A \emph{decent labeling} of $l$ is a labeling $\lambda$ on $\mathbf{C}_l$ which satisfies the following conditions: 
\begin{equation}
\left\{
\begin{array}{lcll}
    \lambda(\vec\rho_l)&=&0,&\\
    \lambda(\vec\sigma_l(\vec\cc))&=&\lambda(\vec\cc)+1, &\text{ if }\vec\cc \in {}^t\vec{\mathbf{C}}^{\uparrow}_l,\\
    \lambda(\vec\sigma_l(\vec\cc))&=&\lambda(\vec\cc)-1, &\text{ if }\vec\cc \in {}^t\vec{\mathbf{C}}^{\downarrow}_l,\\
    \lambda(\vec\sigma_l(\vec\cc))&=&\lambda(\vec\cc)-1, &\text{ if }\vec\cc \in {}^t\vec{\mathbf{C}}^e_l \text{ and } \cc \preccurlyeq_l {\vec\sigma}_l(\vec\cc),\\
    \lambda(\vec\sigma_l(\vec{\cc}))&=&\lambda(\vec\cc)+1, &\text{ if }\vec{\cc} \in {}^t\cev{\mathbf{C}}^e_l \text{ and } {\vec\sigma}_l(\vec{\cc}) \preccurlyeq_l \cc.
\end{array}
\right.
\end{equation}
A labeling that satisfies all these conditions but the first is called
\emph{almost decent}. Note that the corner labeling of a well-labeled
map is always a decent labeling, but the notion of decent labeling is
more general because the label $\lambda(\vec\cc)$ might be different from
$\lambda(\vec\theta_m(\vec\cc))$ for $\vec\cc \in
{}^t\vec{\mathbf{C}}^e_l$, which is not allowed for the corner labeling.

A \emph{labeled scheme} $l$ is a bud-rooted unbalanced scheme (that is
the interior map $l^\circ$ is a scheme and the numbers of leaves and
buds in $l$ are not required to be the same) decorated with a decent labeling. 
Note that a decent labeling may not coincide with the corner labeling
as we already remarked in the previous paragraph.
However, this decorative labeling overrides the usual canonical corner labeling, so that we simply denote it $\lambda_l$ and sometimes refer to it as the corner labeling.
The set of labeled schemes of $\mathbb{S}$ is denoted $\mathcal{L}_\mathbb{S}$.
 
A \emph{scheme-rooted} core is a bud-rooted core whose root vertex is a scheme vertex.
Let $c$ be a scheme-rooted core. 
The \emph{scheme of $c$} is the labeled scheme $s(c)\in\mathcal{L}_\mathbb{S}$ defined as the interior scheme of $c$, with the same root corner, and the same corner labels.
Since $c$ is scheme-rooted, in a tour of its face starting from the root, each side of a branch is visited all at once. 
By consequence, the root order of corners of $s(c)$ coincide in $s(c)$ and in $c$, and the corner labeling of $s(c)$ is decent. This confirms that $s(c)$ do belong to $\mathcal{L}_\mathbb{S}$.
See \cref{subfig:scheme} which illustrates the scheme of the scheme-rooted core $c$ from \cref{subfig:coreRerooted}.

\subsection{Rerooting}\label{sec:reroot}

In order to ensure that a labeled scheme always has a decent labeling,
we only attributed a scheme to scheme-rooted cores. In the process of
decomposing the opening of a map into a well-rooted map, we hence need a tool that transforms a well-rooted map into a scheme-rooted one.
We recall that a stem of a bud-rooted map is \emph{rootable} if it is either the root bud or a leaf.
In the next section, we will describe a way to reroot a bud-rooted map on any rootable stem. 
We would like to apply this procedure on a scheme leaf of the core, so as to obtain a scheme-rooted core.

\subsubsection{Rerooting on a rootable stem}

Let $m$ be a bud-rooted map and $s$ a rootable stem of $m$. The \emph{rerooting of $m$ on $s$}, denoted $\Omega(m,s)$, is the map obtained from $m$ by changing the root bud to a leaf, changing $s$ to a bud, and setting the root of $\Omega(m,s)$ to be the corner preceding $s$ in tour order. 
Note that (if $\rho^\uparrow_m$ denotes the root-bud of a bud-rooted map $m$) $\Omega(m,\rho^\uparrow_m)=m$, and that for any rootable stem $s$, $\Omega(\Omega(m,s),\rho^\uparrow_m)=m$.

\begin{lemma}\label{lem:rerootOnStem}
The rerooting of a bud-rooted map on a rootable stem is a bud-rooted map.
\end{lemma}

\begin{proof}
Let $m$ be a bud-rooted map, and $s$ a rootable stem of $m$ (different from the root).
It is clear that $\Omega(m,s)$ is rooted on a bud. Let us prove that it is well-blossoming.

There is a unique way to merge $s$ and $\rho^\uparrow_m$ into a single edge $e$
to separate $m$ into $2$ faces . We denote by $f_r$ the face which contains
corners strictly smaller than $s$ with respect to $\succcurlyeq_{m}$, and $f_l$ the other one (see \cref{fig:rerooting}). The corner labeling of $m$ and $\Omega(m,s)$ only depend on the orientation of the stems. Suppose that the labels around $s$ in $m$ are $i+1,i$. For any corner $c\in f_r$, we have $\lambda_{\Omega(m,s)}(\cc)=\lambda_m(\cc)-(i+1)$, whereas for any corner $c\in f_l$, we have $\lambda_{\Omega(m,s)}(\cc)=\lambda_m(\cc)-i+1$.

Let $\vec\cc$ be a corner of $\mathbf{C}^e_m$, such that $\vec\cc
\preccurlyeq_m\vec{\sigma}_m(\vec\cc)$, so that
$\lambda_m(\vec\sigma_m(\vec\cc))=\lambda_m(\vec\cc)-1$ because $m$ is
well-blossoming. Suppose that both $\vec\cc$ and $\vec{\sigma}_m(\vec
\cc)$ belong to the same face $f_r$ or $f_l$. Then we also have $\vec\cc
\preccurlyeq_{\Omega(m,s)}\vec{\sigma}_m(\vec\cc)$, and
$\lambda_{\Omega(m,s)}(\vec{\sigma}_m(\vec
\cc))=\lambda_{\Omega(m,s)}(\vec\cc)-1$. Suppose now that $\vec\cc\in f_r$, and $\vec{\sigma}_m(\vec\cc)\in f_l$. Then we have $\vec\cc \succcurlyeq_{\Omega(m,s)}\vec{\sigma}_m(\vec\cc)$, and $\lambda_{\Omega(m,s)}(\vec{\sigma}_m(\vec\cc))-\lambda_{\Omega(m,s)}(\vec\cc)=-1-(-i+1)+(i+1)=1$. The second equation of \cref{cndn:wellLabeling} can be proved the same way.
\end{proof}

A stem $s$ of a bud-rooted map $m$ is \emph{well-rootable} if the rerooting of $m$ on $s$ is well-rooted.

\begin{lemma}\label{lem:wellRootableStems}
Any bud-rooted map has exactly $2$ well-rootable stems.
\end{lemma}

\begin{proof}
Let $m$ be a bud-rooted, but not well-rooted map. Let $-k$ be
the minimum value of the corner labeling of $m$. We denote $\vec\cc_1$
($\vec\cc_2$ respectively) the first corner (for root order) with label $-k+1$
($-k$ respectively). Note that $\vec\cc_1\preccurlyeq_m \vec\cc_2$. Note also
that both $\vec\cc_1$ and $\vec\cc_2$ are preceded by a leaf, denoted
$l_1$ and $l_2$, respectively.

A careful look at the evolution of labels on the faces $f_r$ and $f_l$
(defined in the proof of \cref{lem:rerootOnStem}) allows to prove that the
rerooting of $m$ on $l_1$ or $l_2$ is well-rooted, while the rerooting
of $m$ on any other rootable stem is not well-rooted. The case when
$m$ is well-rooted can be proved in a very similar way.
\end{proof}

\subsubsection{Virtually-rooted maps}

In the orientable case, the second author showed in
\cite{Lepoutre2019} that a core always has at least one rootable
scheme stem, so that the previously described strategy of rerooting a bud-rooted
map on a rootable scheme stem is indeed possible.
However, this is not true anymore in the case of non-orientable
surfaces, and this leads us to define a more complicated rerooting operation
by introducing a notion of \emph{virtual} stems.

A \emph{virtually-rooted} map is a bud-rooted map $m$ with two marked buds: the one following $\vec\rho_m$, and the one preceding $\vec \rho_m$; and such that $\vec\sigma(\vec\rho_m)\notin{}^t\mathbf{C}_m^{\uparrow}$ and $\cev\sigma^{-1}(\vec\rho_m)\notin{}^t\mathbf{C}_m^{\uparrow}$. In other words, $\vec \rho_m$ is surrounded by exactly $2$ buds. 
Those $2$ marked buds are called \emph{virtual}.
Virtual buds have no color, and are not taken into account in the
weight of a map, nor in the degree of a vertex. Consequently, for a
virtually-rooted map $m$ the number of
its black (white, respectively) leaves is equal to
$\gamma_m^{r\bullet} $ ($\gamma_m^{r\circ}$, respectively) where
$\gamma_m^r = (\gamma_m^{r\bullet}, \gamma_m^{r\circ})$ is the
rootable-stem-color-weight of $m$.
The \emph{virtual degree} of a vertex is the degree of that vertex when taking into account its virtual stems.

From now on, $\mathcal{C}$, and $\mathcal{L}$ denote the sets of
(virtually-rooted or not) scheme-rooted cores and labeled schemes.

\subsubsection{Rerooting on a rootable corner}

Let $m$ be a (virtually-rooted or not) bud-rooted map.
A \emph{rootable corner} of $m$ is an oriented corner $\vec\cc\in{}^t\mathbf{C}_m$ which is not adjacent to a non-rootable stem.
Let $\vec\cc$ be a rootable corner of $m$. 
The \emph{rerooting of $m$ on $\vec\cc$}, denoted $\Omega(m,\vec\cc)$, is the map $m'$ obtained from $m$ by the following procedure: 
\begin{compactitem}
    \item if $\vec\cc$ is followed by a rootable stem $s$, reroot $m'$ on $s$.
    \item if $\vec\cc$ is followed by an edge, then add two virtual stems in place of $\vec\cc$: a bud $b$ followed by a leaf $l$; then reroot $m'$ on $l$.
    \item if $m'$ contains virtual stems not adjacent to the root corner, remove them.
\end{compactitem}
Note that the fact that the virtual stems of a virtually rooted map
$m$ are not followed nor preceded by a real bud ensures that, if
$\vec\cc$ is a rootable corner of $m$ different than $\vec\sigma_m(\vec \rho_m)$, the corner $\vec\cc_\rho$ of $\Omega(m,\vec\cc)$ that contains $\rho_m$ is rootable, and that $\Omega(\Omega(m,\vec\cc),\vec\cc_\rho)=m$.
Note also that $\vec\sigma_m(\vec \rho_m)$ is rootable and $\Omega(m,\vec\sigma_m(\vec \rho_m))=m$.

\begin{lemma}\label{lem:rerootOnCorner}
The rerooting of a (virtually-rooted or not) bud-rooted map $m$ on a rootable corner $\vec\cc$ is either a virtually-rooted bud-rooted map if $\vec\cc\in{}^t\mathbf{C}_m^e$, or a real bud-rooted map if $\vec\cc\in{}^t\mathbf{C}_m^s$.
\end{lemma}
\begin{proof}
This is clear if $m$ is not virtually-rooted. Suppose $m$ is virtually-rooted.
If $\vec\cc\in{}^t\mathbf{C}_m^e$, then after the second step, $m'$ is a bud-rooted map with $4$ virtual stems: $2$ virtual buds around the root corner of $m'$, and a bud followed by a leaf around $\vec\rho_m$.
Hence removing these $2$ virtual stems in the third step does not mess up with labels, and the resulting map $m'$ is well-blossoming, and hence bud-rooted.
The other case is similar.
\end{proof}

\begin{lemma}\label{lem:wellRootableStemsVirtual}
Any bud-rooted map has exactly $2$ well-rootable stems, and these are real.
\end{lemma}
\begin{proof}
Any well-blossoming map has as many buds as leaves. In particular, a virtually-rooted map $m$ has at least one real rootable stem. Applying \cref{lem:wellRootableStems} on the rerooting of $m$ on a real rootable stem implies that $m$ has exactly $2$ real well-rootable stem. Since $\rho_m$ is preceded by a bud, the root-bud is not well-rootable, which concludes the proof.
\end{proof}

\subsubsection{Unlabeled and unrooted maps}

Two bud-rooted maps are said to be \emph{root-equivalent} if they can be obtained one from another by a rerooting-on-a-rootable-corner operation. 
We call \emph{unrooted map of $m$} and denote $\overline{m}$ the equivalence class of $m$ for root-equivalence.
Equivalently, the unrooted map of $m$ is the map obtained from $m$ by
forgetting the root corner, removing the virtual stems, forgetting which stems are buds or leaves, and instead recalling which real stems are rootable or not.
Note that $m$ can be recovered from $\overline{m}$ just by the choice of a rootable corner of $\overline{m}$.

We call \emph{unlabeled scheme} the map obtained from a labeled scheme by forgetting its corner labeling.
The set of unlabeled schemes of a surface $\mathbb{S}$ is denoted $\mathcal{U}_\mathbb{S}$. The \emph{unrooted scheme} associated to a well-rooted map $m\in\mathcal{R}_\mathbb{S}$ is the unrooted map of the unlabeled scheme of the rerooting on a scheme rootable corner of the pruned map of $m$. 
Let $\overline{s}$ be an unrooted scheme. We recall that $\mathcal{C}$
and $\mathcal{L}$ denote the sets of (virtually-rooted or not)
scheme-rooted cores and labeled schemes, respectively.
We denote $\mathcal{R}_{\overline{s}}$ (resp. $\mathcal{C}_{\overline{s}}$) the set of maps of $\mathcal{R}$ (resp. $\mathcal{C}$) that have $\overline{s}$ as an unrooted scheme, and $\mathcal{M}_{\overline{s}}$ the set of maps whose opening have $\overline{s}$ as an unrooted scheme.

We also denote $\mathcal{C}_{l}$ (resp. $\mathcal{C}_{s}$) the set of maps of $\mathcal{C}$ that have $l$ as a labeled scheme (resp. $s$ as an unlabeled scheme), for $l\in\mathcal{L}$ and $s\in\mathcal{U}$.

\subsection{The shortcut algorithm}\label{sec:shortcut}

The pruning and rerooting operations can be merged into a single operation, referred to as the \emph{shortcut algorithm} in \cite{Lepoutre2019,Lepoutre:thesis}, although these previous works did not use virtual stems. 
A \emph{decorated core} is a core $c\in\mathcal{C}$ along with both a sequence of trees $(T_s)$ indexed by real rootable stems of $c$, and an integer $\varepsilon\in\{1,2\}$.

The \emph{shortcut algorithm} consists in the following operations:
\begin{compactitem}
    \item take a non virtually-rooted well-rooted map $m$ with a marked rootable scheme corner~$\vec\cc$.
    \item set $\varepsilon$ to be the integer such that $\rho_m$ is the $\varepsilon$-th well-rootable stem in the facial order starting from~$\vec\cc$.
    \item perform the pruning algorithm, and keep the removed trees as decoration indexed by the rootable stem they were cut from.
    \item reroot the map on $\vec\cc$.
\end{compactitem}

The \emph{inverse shortcut algorithm} consists in the following operations:
\begin{compactitem}
    \item take a decorated scheme core $(c,(T_s),\varepsilon)$.
    \item set $m$ to be the map obtained from $c$ by gluing each tree $T_s$ on the corresponding rootable stem $s$.
    \item reroot $m$ on the $\varepsilon$-th rootable stem in the facial order starting from the root.
    \item mark the corner of $m$ that contains $\rho_c$
\end{compactitem}

\begin{figure}
\centering
\subfloat[]{
	\label{subfig:corePruned}
	\includegraphics[page=1,width=0.45\linewidth]{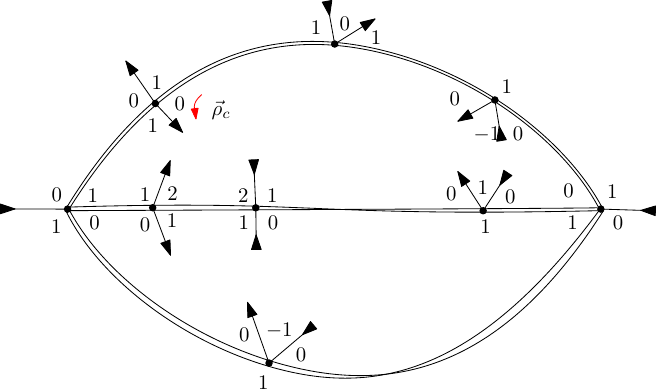}}
	\quad
	\subfloat[]{
	\label{subfig:coreRerooted}
	\includegraphics[page=3,width=0.45\linewidth]{core3d}}
    \caption{\protect\subref{subfig:corePruned}: another representation of the core represented in \cref{subfig:itsCore}. \\
    \protect\subref{subfig:coreRerooted}: the rerooting of the previous core on a scheme rootable corner. 
    In red, the merging of the two special stems, and the two special faces (see the proof of \cref{lem:rerootOnStem}).
    This is a virtually-rooted scheme-rooted core, whose virtual stems are represented in blue.}\label{fig:rerooting}
\end{figure}

\begin{theorem}\label{thm:shortcutBij}
The shortcut algorithm is a bijection between non virtually-rooted well-rooted maps with a marked rootable scheme corner having unrooted scheme $\overline{s}$, and decorated cores having unrooted scheme $\overline{s}$. 
\end{theorem}
\begin{proof}
\begin{itemize}
    \item \cref{lem:wellRootableStems} ensures that $\varepsilon\in\{1,2\}$, so that the shortcut algorithm produces a correct decorated core.
    \item \cref{lem:wellRootableStemsVirtual} ensures that the inverse shortcut algorithm is a non virtually-rooted well-rooted map with a marked rootable scheme.
    \item It is clear that the two algorithms are inverse one of another, and that they preserve the unrooted scheme.
\end{itemize}
\end{proof}

\begin{lemma}
\label{lem:shortcutEnum}
The shortcut algorithm yields, for any unrooted unlabeled scheme $\overline{s}$:
\begin{equation}
R_{\overline{s}}(\zz,x,y)=
\frac{C_{\overline{s}}^{\tiny \LEFTcircle}(T_\bullet,T_\circ)}
{n^r_{\overline{s}}},
\end{equation}
where $C_{\overline{s}}^{\tiny
  \LEFTcircle}(T_\bullet,T_\circ)\coloneqq
C_{\overline{s}}(T_\bullet,T_\circ)+C_{\overline{s}}(T_\circ,T_\bullet)$,
${n^r_{\overline{s}}}$ denotes the number of rootable corners of
$\overline{s}$ and
\[ C_{\overline{s}}(x,y) = \sum_{m\in\mathcal{C}_{\overline{s}}}
  x^{\gamma_m^{r\bullet}}y^{\gamma_m^{r\circ}}.\]
\end{lemma}

\begin{proof}
The proof of the real orientable case of \cref{lem:shortcutEnum} given in \cite{Lepoutre:thesis} can be directly extended to this setup.
\end{proof}

\section{The $4$-valent case}
\label{sec:4valent}

We now restrict our study to the case of $4$-valent maps, which are known by \cref{thm:TutteBij} to be in bijection with general maps. 
In this setup, the generating series can be studied more easily, as we are able to express the decomposition of cores into schemes by representing branches by decorated Motzkin paths, and to study in detail the combinatorial structure of schemes.

In \cref{sec:motzkin}, we introduce some series of decorated paths, and state a lemma that will be useful to express the rationality of a series written as a function of these decorated Motzkin paths series.
In \cref{sec:exprCl}, we describe the decomposition of a core into a scheme whose edges are decorated with Motzkin paths, and its enumerative consequences.
In \cref{sec:offset}, we study in more detail the structure of
schemes, and in particular the form of their \emph{offset graph}. This
part displays an important structural difference between orientable
and non-orientable maps and is crucial in achieving our ultimate goal,
which is a description of rationality results of the generating functions of maps with a given
underlying unlabeled scheme. This result is presented in
\cref{sec:enumeration} as \cref{thm:symOfLoop} and it refines
\cref{thm:AG00} and extends the recent result
of the second author \cite{Lepoutre2019}.

\subsection{Motzkin paths}\label{sec:motzkin}

A \emph{Motzkin walk} $w$ of length $\ell$ is a finite sequence of steps $w_1,\ldots,w_\ell$, where $w_i\in \{-1,0,1\}$, for $1\leq i \leq \ell$. 
A step $\omega_i$ is called \emph{horizontal}, \emph{up} or \emph{down} if it is equal to $0$, $+1$ or $-1$, respectively. For $1\leq k \leq \ell$, the \emph{height at time $k$} is equal to $\sum_{i=1}^{k} w_i$ and the \emph{height of the $k$-th step} is the height at time $k-1$ (the height of the $1$-st step is equal to $0$ by convention). The \emph{increment} of $w$ is the height at time $\ell$. 
A \emph{Motzkin bridge} is a Motzkin walk whose increment is equal to 0. A \emph{primitive Motzkin walk} is a Motzkin walk with increment $-1$ and such that the height of each step is non-negative. 

We introduce $D_\bullet$, $D_\circ$ and $B$ as the following generating series:
\begin{align}
D_{\bullet}(t_\bullet,t_\circ) &= \sum_{w \text{ primitive Motzkin}}(2(t_{\bullet}+t_\circ))^{h(w)}t_{\bullet}^{e(w)}t_{\circ}^{o(w)}\\
D_{\circ}(t_\bullet,t_\circ) &= \sum_{w \text{ primitive Motzkin}}(2(t_{\bullet}+t_\circ))^{h(w)}t_{\circ}^{e(w)}t_{\bullet}^{o(w)}\\
B(t_\bullet,t_\circ) &= \sum_{w \text{ Motzkin bridges}}(2(t_{\bullet}+t_\circ))^{h(w)}t_{\bullet}^{e(w)}t_{\circ}^{o(w)},
\end{align}
where $h(w)$, $o(w)$ and $e(w)$, denote the number of horizontal steps, non-horizontal steps with odd height (called odd steps) and non-horizontal steps with even height (called even steps), in $w$.

As proved in \cite{Lepoutre:thesis}, the series $B$, $D_\bullet$ and $D_\circ$ satisfy the following symmetry equations:
\begin{align}
B(t_\bullet,t_\circ)&=B(t_\circ,t_\bullet),\\
t_\bullet D_\circ &= t_\circ D_\bullet,
\end{align}
and the following decomposition equations:
\begin{align}
D_\bullet&=t_\bullet+2(t_\bullet+t_\circ)D_\bullet+t_\bullet\cdot D_\circ D_\bullet,\\
D_\circ&=t_\circ+2(t_\bullet+t_\circ)D_\circ+t_\circ\cdot D_\bullet D_\circ,\\
B&=1+2(t_\bullet+t_\circ)B+2t_\bullet\cdot D_\circ B.
\end{align}

We hence denote $D\coloneqq \frac{D_\bullet}{t_\bullet}= \frac{D_\circ}{t_\circ}$.
The previous equations imply the following properties:

\begin{align}
\label{prop:BDt}
    t_\circ &=\frac{1}{D_\bullet+2\left( \frac{D_\bullet}{D_\circ}+1\right) +\frac{1}{D_\circ}},\\
    t_\bullet &=\frac{1}{D_\circ+2\left( \frac{D_\circ}{D_\bullet}+1\right) +\frac{1}{D_\bullet}},\\
    D&=\frac{1-2(t_\bullet+t_\circ)-\sqrt{(1-2t_\bullet-2t_\circ)^2-4t_\bullet t_\circ}}{2t_\bullet t_\circ},\\
    B&=\frac{1}{\frac{1}{D}-t_\bullet t_\circ D}=\frac{1}{\sqrt{(1-2t_\bullet-2t_\circ)^2-4t_\bullet t_\circ}}.
\end{align}

A function $F(x,y)$ is \emph{$\parallel$-symmetric} (resp. \emph{$\parallel$-antisymmetric}) if for any $x,y$, $F(x,y)=F(\frac{1}{x},\frac{1}{y})$ (resp. $F(x,y)=-F(\frac{1}{x},\frac{1}{y})$).
Note that $B(D_\bullet,D_\circ)$ is $\parallel$-antisymmetric and rational in $t_\bullet$, $t_\circ$ and $D$, while $t_\bullet(\frac{1}{D_\bullet},\frac{1}{D_\circ})=t_\circ$.

The following theorem was stated and proved in \cite{Lepoutre:thesis}, and is the bivariate generalization of \cite[Lemma 9]{ChapuyMarcusSchaeffer2009}.

\begin{lemma}
\label{lem:criterion}
Let $f$ be a symmetric function and write $F$ for the function such that $F(D_\circ,D_\bullet)=f(t_\circ,t_\bullet)$. Then $F$ is also symmetric.

Moreover, the two following properties are equivalent: 
\begin{enumerate}
	\item $f$ is a rational function,
	\item $F$ is a rational function and is $\parallel$-symmetric.
\end{enumerate}
\end{lemma}

\subsection{Branches as Motzkin paths}\label{sec:exprCl}

\begin{figure}
    \centering
    \includegraphics[page=7]{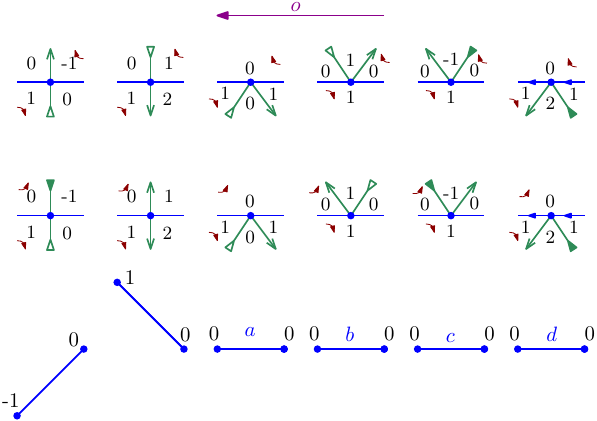}
    \caption{An illustration of the different types of branch vertices. The purple arrow represent the orientation of the edge, while the brown arrows represent the orientation of the corners in the facial tour. The first lines corresponds to a backward branch, while the second represents a forward branch. The last line is the Motzkin step encoding the vertex in the decomposition of a core into a scheme.}
    \label{fig:6types}
\end{figure}

We orient each edge of a scheme in such a way that it is first visited backward in the tour starting from the root.

Each of the $6$ types of vertex of interior degree $2$ in a branch can be matched to a step (see \cref{fig:6types}): an upstep (resp. downstep) for a vertex across which the labels increase (resp. decrease) by $1$ along the branch, and $4$ types of horizontal step for vertices across which the labels remain the same along the branch.

We denote $\Motz_i^j$ the set of decorated Motzkin paths (that is
Motzkin paths, where each horizontal
step is additionally decorated by one of the labels $a,b,c$ or $d$) starting from height $i$
and finishing at height $j$.

For $i\leq j$, we also denote:
\[ \Delta_i^j(x,y)\coloneqq x^{|[i,j[\cap (2\mathbb{Z})|}
  y^{|[i,j[\cap (2\mathbb{Z}+1)|}.\]

Let $l$ be a labeled scheme. We say that a halfedge $h$ which has labels $i+1,i$ has \emph{height} $i$, and denote $\lambda_l(h)=i$.
If the halfedge at the origin (resp. destination) of an edge $e$ of
$l$ has height $i$ (resp. $j$), we say that this edge has \emph{height at the origin} $i$ (resp. \emph{height at the destination} $j$) and write $\lambda_l^0(e)=i$ (resp. $\lambda_l^1(e)=j$).
The edge $e$ is said to be \emph{increasing} (resp. \emph{decreasing}) if $\lambda_l^0(e)\leq \lambda_l^1(e)$ (resp. $\lambda_l^0(e)>\lambda_l^1(e)$). 
An edge is said to be \emph{backward} (resp. \emph{forward}) if it is not (resp. it is) followed twice in the same direction in the tour of the face starting from the root.
The set of increasing backward, decreasing backward, increasing forward, and decreasing forward edges of $l$ are denoted $\mathbf{E}_l^{\tiny\nearrow}$, $\mathbf{E}_l^{\tiny\searrow}$, $\mathbf{E}_l^{\tiny\nwarrow}$, and $\mathbf{E}_l^{\tiny\swarrow}$.

A \emph{decorated labeled scheme} is a labeled scheme, along with a
sequence $(W_e)_{e\in\mathbf{E}_l}$ of Motzkin paths, indexed by
edges of $l$, such that each path $W_e \in
\Motz_{\lambda_l^0(e)}^{\lambda_l^1(e)}$ goes from height
$\lambda_l^0(e)$ to height $\lambda_l^1(e)$.

Recall that $\vec{\mathbf{C}}^{\downarrow}_l$ denotes the set of
corners of $l$ followed by a leaf and that
\[ C_{l}(x,y) = \sum_{m\in\mathcal{C}_{l}}
  x^{\gamma_m^{r\bullet}}y^{\gamma_m^{r\circ}}\]
is the bivariate generating function (with respect to
the rootable-stem-color-weight $\gamma_m^{r} = (\gamma_m^{r \bullet},
\gamma_m^{r \circ})$) of scheme-rooted cores
$\mathcal{C}_{l}$ that have $l$ as labeled scheme.

\begin{lemma}\label{lem:decompCore}
There is a bijection between scheme-rooted cores, and decorated labeled schemes.
This bijection yields, for any labeled scheme l:
\begin{align}
  C_l(t_\bullet, t_\circ) &= 
  t_\bullet^{\chi(l \text{ is real-rooted})}\cdot\prod_{\cc\in \vec{\mathbf{C}}^{\downarrow}_l}\Delta_{\lambda_l(\cc)}^{\lambda_l(\cc)+1}(t_\bullet,t_\circ)\cdot 
  \prod_{e\in\mathbf{E}_l^{\tiny\nearrow}}B\cdot\Delta_{\lambda_l^0(e)}^{\lambda_l^1(e)}(D_\bullet,D_\circ)\\
  &\cdot\prod_{e\in\mathbf{E}_l^{\tiny\searrow}}B\cdot\Delta_{\lambda_l^1(e)}^{\lambda_l^0(e)}(D_\circ,D_\bullet) \cdot
  \prod_{e\in\mathbf{E}_l^{\tiny\nwarrow}}B\cdot\Delta_{\lambda_l^0(e)}^{\lambda_l^1(e)}(t_\circ t_\bullet D,t_\circ t_\bullet D) \cdot  \prod_{e\in\mathbf{E}_l^{\tiny\swarrow}}B\cdot\Delta_{\lambda_l^1(e)}^{\lambda_l^0(e)}(D,D),
\end{align}
where $D\coloneqq \frac{D_\bullet}{t_\bullet}=\frac{D_\circ}{t_\circ}$
and $\chi(A) = \begin{cases} 1 &\text{ if } A \text{ is satisfied,}\\ 0 &\text{ otherwise. } \end{cases}$
\end{lemma}

An illustration of the aforementioned bijection is given in \cref{fig:Scheme}.

\begin{figure}
\centering
\subfloat[]{
	\label{subfig:core}
\includegraphics[page=2,width=0.35\linewidth]{core3d}}
\quad
\subfloat[]{
	\label{subfig:scheme}
	\includegraphics[width=0.35\linewidth]{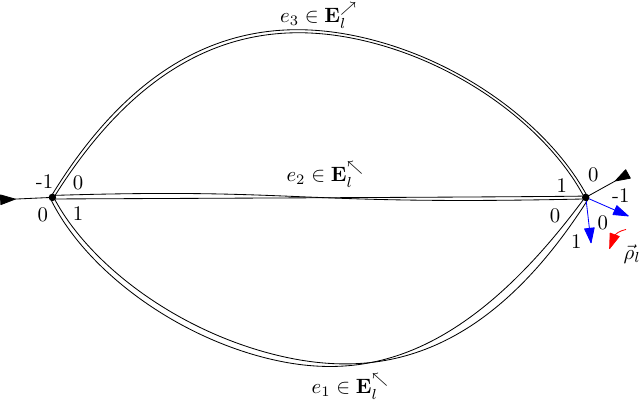}}
\quad
\subfloat[]{
	\label{subfig:Branches}
	\includegraphics[page=2,width=0.2\linewidth]{scheme3d.pdf}}
\caption{The bijection described in \cref{lem:decompCore} applied to the blossoming core represented in \cref{subfig:coreRerooted}. The core \protect\subref{subfig:core}, its labeled scheme \protect\subref{subfig:scheme}, and the Motzkin paths decorations \protect\subref{subfig:Branches} are successively represented.}
\label{fig:Scheme}
\end{figure}

\begin{proof}
The first part of the statement is rather clear. Indeed, let $c$ be a
scheme-rooted core. By performing a facial tour of $c$ starting from the root, we visit some
branches of $c$ and some scheme vertices. Note that for a given branch
of $c$ the edges belonging to this branch are either all forward or
all backward. Moreover, each vertex belonging to this branch has to be one of those represented in \cref{fig:6types}. We therefore encode the sequence of the vertices of each branch
of $c$ by the corresponding decorated Motzkin walk. 
Since the core $c$ can be recovered from solely its labeled scheme $l$ and the Motzkin path associated to each edge of $l$, this establishes the bijection of \cref{lem:decompCore}.

Let us compute $C_l(t_\bullet, t_\circ)$. Applying the bijection between scheme-rooted cores and decorated
labeled schemes to the definition of the generating series $C_l(t_\bullet, t_\circ)$ we
have:
\begin{align}
C_l(t_\bullet, t_\circ) = &\sum_{c \in
    \CCC_l}t_\bullet^{\gamma_c^{r\bullet}}t_\circ^{\gamma_c^{r\circ}}
  = t_\bullet^{\chi(l \text{ is real-rooted})}\cdot\prod_{\cc\in \vec{\mathbf{C}}^{\downarrow}_l}\Delta_{\lambda_l(\cc)}^{\lambda_l(\cc)+1}(t_\bullet,t_\circ)\\
  \cdot &\prod_{e\in\mathbf{E}_l^{\tiny\nearrow}}\sum_{W_e \in
    \Motz_{\lambda_l^0(e)}^{\lambda_l^1(e)}}W_e(t_\bullet,t_\circ)\cdot
  \prod_{e\in\mathbf{E}_l^{\tiny\searrow}}\sum_{W_e \in \Motz_{\lambda_l^0(e)}^{\lambda_l^1(e)}}W_e(t_\bullet,t_\circ) \\
  \cdot&\prod_{e\in\mathbf{E}_l^{\tiny\nwarrow}}\sum_{W_e \in \Motz_{\lambda_l^0(e)}^{\lambda_l^1(e)}}W_e(t_\bullet,t_\circ)\cdot  \prod_{e\in\mathbf{E}_l^{\tiny\swarrow}}\sum_{W_e \in \Motz_{\lambda_l^0(e)}^{\lambda_l^1(e)}}W_e(t_\bullet,t_\circ),
\end{align}
where $W_e(t_\bullet,t_\circ)$ is the rootable-stem-color-weight of a
branch associated with a decorated Motzkin
path $W_e$ through the bijection. In other terms
\[W_e(t_\bullet,t_\circ) = \begin{cases}
    t_{\bullet}^{e(W_e)+oh_{a,b}(W_e)+eh_{c,d}(W_e)}t_{\circ}^{o(W_e)+eh_{a,b}(W_e)+oh_{c,d}(W_e)}
    &\text{ if }
    e\in\mathbf{E}_l^{\tiny\nearrow},\mathbf{E}_l^{\tiny\searrow},\\
  t_{\bullet}^{\#\text{upsteps}(W_e)+oh_{a,b}(W_e)+eh_{c,d}(W_e)}t_{\circ}^{\#\text{upsteps}(W_e)+eh_{a,b}(W_e)+oh_{c,d}(W_e)}
    &\text{ if }
    e\in\mathbf{E}_l^{\tiny\nwarrow},\mathbf{E}_l^{\tiny\swarrow},\end{cases}\]
where $oh_L(W_e)$ ($eh_L(W_e)$, respectively) denotes the number of
horizontal steps in $W_e$ at odd (even, respectively) height decorated by $L$; see \cref{fig:6types}.
We recall that the factor
\[ t_\bullet^{\chi(l \text{ is real-rooted})}\cdot\prod_{\cc\in
    \vec{\mathbf{C}}^{\downarrow}_l}\Delta_{\lambda_l(\cc)}^{\lambda_l(\cc)+1}(t_\bullet,t_\circ)\]
comes from the convention that the bud-root is not counted by the
rootable-stem-color-weight in the case when $l$ is virtually rooted
and it is counted as a black rootable stem otherwise. We now proceed to analyze the expression $W_e(t_\bullet,t_\circ)$ in four different
cases: $e \in\mathbf{E}_l^{\tiny\nearrow},\mathbf{E}_l^{\tiny\searrow},\mathbf{E}_l^{\tiny\nwarrow}$,
and $e \in \mathbf{E}_l^{\tiny\swarrow}$. 

First of all, each $W_e \in \Motz_i^j$ can be uniquely decomposed as a
concatenation of one decorated Motzkin bridge and $|j-i|$ decorated primitive Motzkin walks. 
Indeed, if $i> j$ we decompose it by taking the longest possible walk
from height $i$ to height $i-1$, which stays strictly above height
$i-1$; then the longest possible walk from height $i-1$ to height
$i-2$, which stays strictly above height $i-2$; and so on, until we
reach level $j$ and the rest gives us a Motzkin bridge starting and
ending at height $j$. When $i \leq j$, we consider its reverse version in which we inverse the order of steps and we decompose it by taking the longest possible walk from height $j$ to height $j-1$, which stays strictly above height $j-1$; then the longest possible walk from height $j-1$ to height $j-2$, which stays strictly above height $j-2$; and so on, until we reach level $i$ and the rest gives us a Motzkin bridge starting and ending at height $i$. 

Let $e \in \mathbf{E}_l$ be a backward edge, and let $W_e \in\Motz_i^j$. 
By analyzing \cref{fig:6types} (see also~\cref{fig:Scheme} to compare with a
concerete example) we can see that a vertex corresponding to a non-horizontal step with even (resp. odd) height has weight $t_\bullet$ (resp. $t_\circ$).
Moreover, out of $4$ possible labels on the horizontal step, there are always exactly $2$ decorated horizontal steps with weight $t_\bullet$, and $2$ with weight $t_\circ$. 
If $e\in\mathbf{E}_l^{\tiny\searrow}$, the decomposition of the path
of $W_e$ into primitive Motzkin walks and a final Motzkin bridge
yields:
\[ \sum_{W_e \in
    \Motz_{\lambda_l^0(e)}^{\lambda_l^1(e)}}W_e(t_\bullet,t_\circ) =
  B\cdot\Delta_{\lambda_l^1(e)}^{\lambda_l^0(e)}(D_\circ,D_\bullet).\]
Indeed, for each $\lambda_l^1(e) \leq j < \lambda_l^0(e)$ the corresponding primitive Motzkin walk starting at height $j+1$ and
finishing at height $j$ is counted by
$\Delta_{j}^{j+1}(D_\circ,D_\bullet)$. Similarly, if $e \in
\mathbf{E}_l^{\tiny\nearrow}$ the decomposition of the reversed path
of $W_e$ into primitive Motzkin walks and a final Motzkin bridge
yields:
\[ \sum_{W_e \in
    \Motz_{\lambda_l^0(e)}^{\lambda_l^1(e)}}W_e(t_\bullet,t_\circ) =
  B\cdot\Delta_{\lambda_l^0(e)}^{\lambda_l^1(e)}(D_\bullet,D_\circ).\]
This follows from the fact that an even (odd, respectively) step in a
reversed Motzkin walk corresponds to an odd (even, respectively) step in
the non-reversed one, so that the generating series of a
reversed Motzkin walk is obtained by exchanging the role of
$t_\bullet$ and $t_\circ$.

Now, let $e \in \mathbf{E}_l$ be a forward edge and let $W_e \in \Motz_i^j$. 
Again, by analyzing \cref{fig:6types} (see also~\cref{fig:Scheme} to compare with a
concerete example), we can see that a vertex corresponding to an upstep (resp. to a downstep) has weight $t_\bullet t_\circ$ (resp. $1$). 
Moreover, again, out of $4$, there are always exactly $2$ decorated horizontal steps with weight $t_\bullet$, and $2$ with weight $t_\circ$. 
If  $e \in \mathbf{E}_l^{\tiny\nwarrow}$, then recall that we are
decomposing the inverse path of $W_e$, therefore the role of upsteps
and downsteps is exchanged. In this decomposition we get a sequence  of primitive Motzkin walks, followed by a final Motzkin bridge. 
Since, in a primitive Motzkin path, the difference between the numbers of downsteps and upsteps is always equal to $1$, and the last downstep should have weight $t_\bullet t_\circ$, these paths are enumerated by $t_\bullet t_\circ D$. 
Similarly, Motzkin bridges have as many upsteps as downsteps and are enumerated by $B$.
This yields:

\[ \sum_{W_e \in
    \Motz_{\lambda_l^0(e)}^{\lambda_l^1(e)}}W_e(t_\bullet,t_\circ) =
  B\cdot\Delta_{\lambda_l^0(e)}^{\lambda_l^1(e)}(t_\bullet t_\circ D,
  t_\bullet t_\circ D).\]

Again, a similar analysis for the case $e \in \mathbf{E}_l^{\tiny\swarrow}$ yields:

\[ \sum_{W_e \in
    \Motz_{\lambda_l^0(e)}^{\lambda_l^1(e)}}W_e(t_\bullet,t_\circ) =
  B\cdot\Delta_{\lambda_l^1(e)}^{\lambda_l^0(e)}(D,D),\]
since in this case the last downstep of each primitive Motzkin walk should have weight $1$.

This concludes the proof.
\end{proof}

\subsection{Structure of the offset submap}\label{sec:offset}

\begin{figure}[t]
\centering
\subfloat[]{
	\label{subfig:labScheme}
	\includegraphics[page=1,width=0.3\linewidth]{scheme3d.pdf}}
\quad
\subfloat[]{
	\label{subfig:unlabScheme}
	\includegraphics[page=3,width=0.3\linewidth]{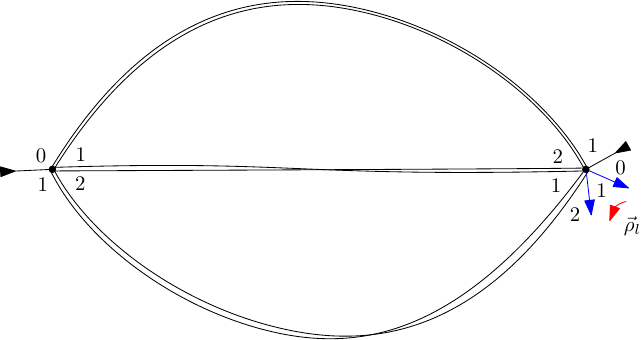}}
\quad
\subfloat[]{
	\label{subfig:OffGraph}
	\includegraphics[page=4,width=0.25\linewidth]{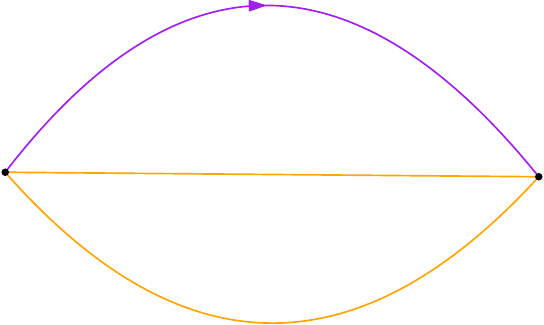}}
\caption{A labeled scheme \protect\subref{subfig:labScheme}, its
  unlabeled scheme \protect\subref{subfig:unlabScheme}, and its offset
  graph~\protect\subref{subfig:OffGraph} (the offset edge is the
  purple oriented edge, the two non-oriented orange edges are shifted
  edges and they do not belong to the offset graph).} 
\label{fig:unlab&Offset}
\end{figure}

We now study in more details the structure of a scheme. We extend the
labeling $\lambda_l$ of the corners of a labeled scheme $l$ to its
vertices: the \emph{height} of a vertex $v\in\mathbf{V}_l$, simply denoted $\lambda_l(v)$, is defined by:
\[\lambda_l(v):=\min_{\substack{ \cc\in\mathbf{C}_l \\ v_l(\cc)=v}}\lambda_l(\cc).\]
We also define a new corner labeling on $l$ called \emph{relative labeling}, denoted $\lambda^r_l$:
\[ \lambda^r_l(\cc) \coloneqq \lambda_l(\cc) - \lambda_l(v_l(\cc)). \]
Note that:
\begin{compactitem}
    \item the difference of relative labels around a vertex correspond to the difference of corner labels,
    \item the minimum relative label incident to a vertex is always $0$ by definition,
    \item the relative labeling of a labeled or unlabeled scheme is almost decent.
\end{compactitem}

We recall that an \emph{unlabeled scheme} is the map obtained from a labeled scheme by forgetting its height function, and its corner labeling (an example is given in \cref{fig:unlab&Offset}).
Note that all labeled schemes having the same unlabeled scheme have the same relative labeling, so that unlabeled schemes naturally have a relative labeling too.

The \emph{relative type} of a halfedge of a scheme (labeled or unlabeled) is the minimum relative label adjacent to that halfedge.
The \emph{relative type} of a vertex (sometimes called type for short) is the maximum relative type of a halfedge incident to this vertex.

\begin{theorem}\label{thm:virtualCoreNoVtx2}
An unlabeled scheme $s\in\mathcal{U}^\times$ (virtually rooted or not) has no vertex of relative type more than $1$.
\end{theorem}

\begin{proof}
Suppose that $s$ is not virtually-rooted. 
Then the cyclic sequence of relative labels around any vertex is
either $(0,1,0,1)$ or $(0,1,2,1)$, in which cases these vertices have
relative type $0$ or $1$.
Now suppose $s$ is virtually-rooted.
We obtain similarly that all non-root vertices of $s$ have relative type $0$ or $1$. 
The root bud of $s$ has virtual degree $6$.
Since the $2$ virtual stems of $s$ are not followed nor preceded by a bud, and the relative labeling of $s$ is almost decent, the cyclic sequence of relative labels around the root vertex, starting from the root, is necessarily of the form $(1,2,1,\varepsilon,1,0)$, where $\varepsilon$ is either $0$ or $1$, so that in any case, the root vertex has relative type~$1$.
\end{proof}

An edge of $l$ is said to be \emph{balanced} (resp. \emph{shifted}) if it is made of halfedges of relative type $0$ (resp. of type $1$). An edge $uv$ is said to be \emph{offset toward $v$} if the halfedge $u$ has relative type $0$ and the halfedge $v$ has relative type $1$. The \emph{offset graph of $l$} is the directed graph made of the offset edges of~$l$. An \emph{offset cycle} is an oriented cycle of the offset graph. An \emph{offset loop} is an offset cycle of length $1$ (see \cref{fig:unlab&Offset}).

It was proved in \cite{Lepoutre2019} that any (non-virtually-rooted) unlabeled scheme on an orientable surface has no offset cycle. Using this acyclicity, it was proved \cite{Lepoutre2019, Lepoutre:thesis} that the series of cores having a given scheme with no offset cycle is rational (both in the univariate and bivariate setups), which is a stronger property than the general case expressed in \cite{ArquesGiorgetti2000}. 
We generalize the aforementioned result to non-orientable schemes by studying the structure of the offset graph of a scheme in an arbitrary surface. 
Our result can be interpreted as a combinatorial explanation of the reason why the series $C_s$ is rational in the case of orientable maps, but not in the general case. 

\begin{theorem}\label{thm:structOff}
Let $s\in\mathcal{U}^\times_\mathbb{S}$ be an unlabeled 
scheme of a surface $\mathbb{S}$. 
\begin{compactitem}
\item Any offset cycle of $s$ has length at most $2$, and consists of consecutive forward edges in the tour of $s$. In particular, an orientable labeled scheme has no offset cycle.
\item Additionally, all these cycles are pairwise disjoint and the
  total length of cycles in the offset graph of any labeled scheme is
  at most equal to $2\cdot g_\mathbb{S}$, where $g_\mathbb{S}$ is the
  genus of $ \mathbb{S}$.
\end{compactitem}
\end{theorem}

We deduce \cref{thm:structOff} from a stronger result which works for
a larger class of maps called \emph{special $4$-valent maps} 
$\widetilde{\mathcal{M}}^\times_\mathbb{S}$. We say that $l$ is a
\emph{special labeled $4$-valent map} and write $l
\in\widetilde{\mathcal{L}}^\times_\mathbb{S}$ if $l$ is a bud-rooted
$4$-valent map (possibly virtually-rooted, with the usual convention that the virtual
buds are not counted in a vertex degree) equipped with a decent labeling such that
\begin{itemize}
\item every non-root vertex of type $1$ and internal degree at least
  $2$ is adjacent to at most one bud,
\item if $l$ is not scheme-rooted, then it is not virtually rooted either. 
\end{itemize}
We call \emph{special unlabeled $4$-valent map} the map obtained from a
special labeled $4$-valent map by forgetting its height function, and its
corner labeling. Note that clearly
$\mathcal{U}^\times_\mathbb{S} \subset
\widetilde{\mathcal{M}}^\times_\mathbb{S}$, since every vertex of a
scheme has internal degree at least $3$, so that these conditions are
clearly satisfied by unlabeled schemes. This inclusion is strict as
maps in $\widetilde{\mathcal{M}}^\times_\mathbb{S}$ may have many
vertices of internal degree less than $3$.

\begin{lemma}\label{lem:structOff}
Let $s\in \widetilde{\mathcal{M}}^\times_\mathbb{S}$ be a special
unlabeled $4$-valent map 
of a surface $\mathbb{S}$. The first visited offset cycle
 $C$ of $s$ can only be of the form presented in
 \cref{fig:Cycles}. In particular, $C$ consists of
 consecutive forward edges and has length $\ell(C)$ at
 most $2$.
\end{lemma}
\begin{figure}[h!!!]
	\adjincludegraphics[]{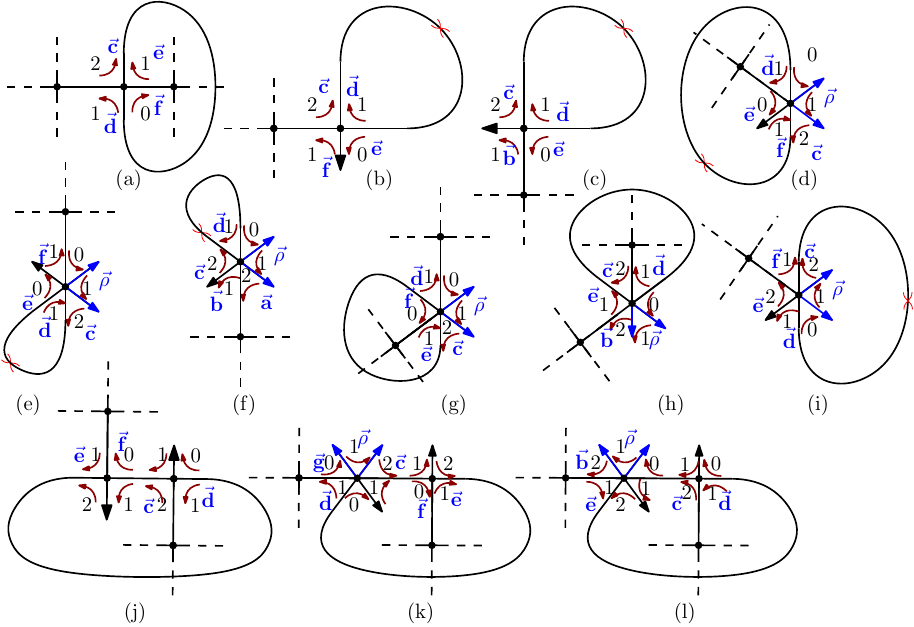}
\caption{All the possible types of the first visited offset cycle in
  special $4$-valent maps. $\vec \rho$ denotes the root corner,
  while the tour order of
  corners denoted by latin letters coincides with the latin alphabet
  order. Corner $\vec\cc$ is the same as in the proof of \cref{lem:structOff}.}
	\label{fig:Cycles}
\end{figure}

\begin{proof}
Assume that there exists a simple cycle in the offset
graph and let $C$ denote the first visited offset cycle of $s$. Notice that all the vertices of type $0$ have indegree $0$ in
the offset graph, therefore all the edges belonging to $C$ connect
only vertices of type $1$. We set terminology used in this proof:
\begin{itemize}
\item $e$ is the first edge of $C$ visited in the tour of $s$;
\item $\vec\cc_1$ is the first visited corner adjacent  to $e$;
\item $\vec\cc_2 \coloneqq \vec\theta_s(\vec\cc_1)$ is the corner following $\cc_1$ in the tour of $s$;
\item $v_1\coloneqq v_s(\vec\cc_1)$ is the vertex to which $\vec\cc_1$ is
  incident;
\item $v_2\coloneqq v_s(\vec\cc_2)$ is the vertex to which $\vec\cc_2$ is
  incident;;
    \item $h_s(\vec\cc)$ is the halfedge following a corner $\vec\cc$;
  \item in the case when $v \in \V_s$ is not the
    virtual root vertex we denote by $\vec\cc_f(v)$ the first visited
    corner in the tour of $s$ adjacent to $v$. Otherwise $\vec\cc_f(v)$
    is the first visited corner adjacent to
    the virtual root vertex $v$ after visiting $\vec \rho_s$ and $\vec
    \theta_s(\vec \rho_s)$ in the tour of $s$.
\end{itemize}
Since $\lambda_s^r$ is a decent-labeling then:
\begin{equation}\label{eq:DescentInOffset}
\lambda_s^r\left(\vec\cc_f(v)\right) = \lambda_s^r\left(\vec\sigma_s^{-1}(\vec\cc_f(v))\right) + 1 = 
\begin{cases} 
    \lambda_s^r\left(\vec\sigma_s(\vec\cc_f(v))\right) - 1 &\text{ if
      $\vec\cc_f(v) \in {}^t\vec{\mathbf{C}}^\uparrow_s$,}\\ 
    \lambda_s^r\left(\vec\sigma_s(\vec\cc_f(v))\right) + 1 &\text{ otherwise.}  
\end{cases}    
\end{equation}
Our proof strategy is to check various properties of the cycle $C$ by
searching the tree--style diagram presented in
\cref{fig:ProofDiagram} below. 
\begin{figure}[h!]
\centering
	\adjincludegraphics[width=\linewidth]{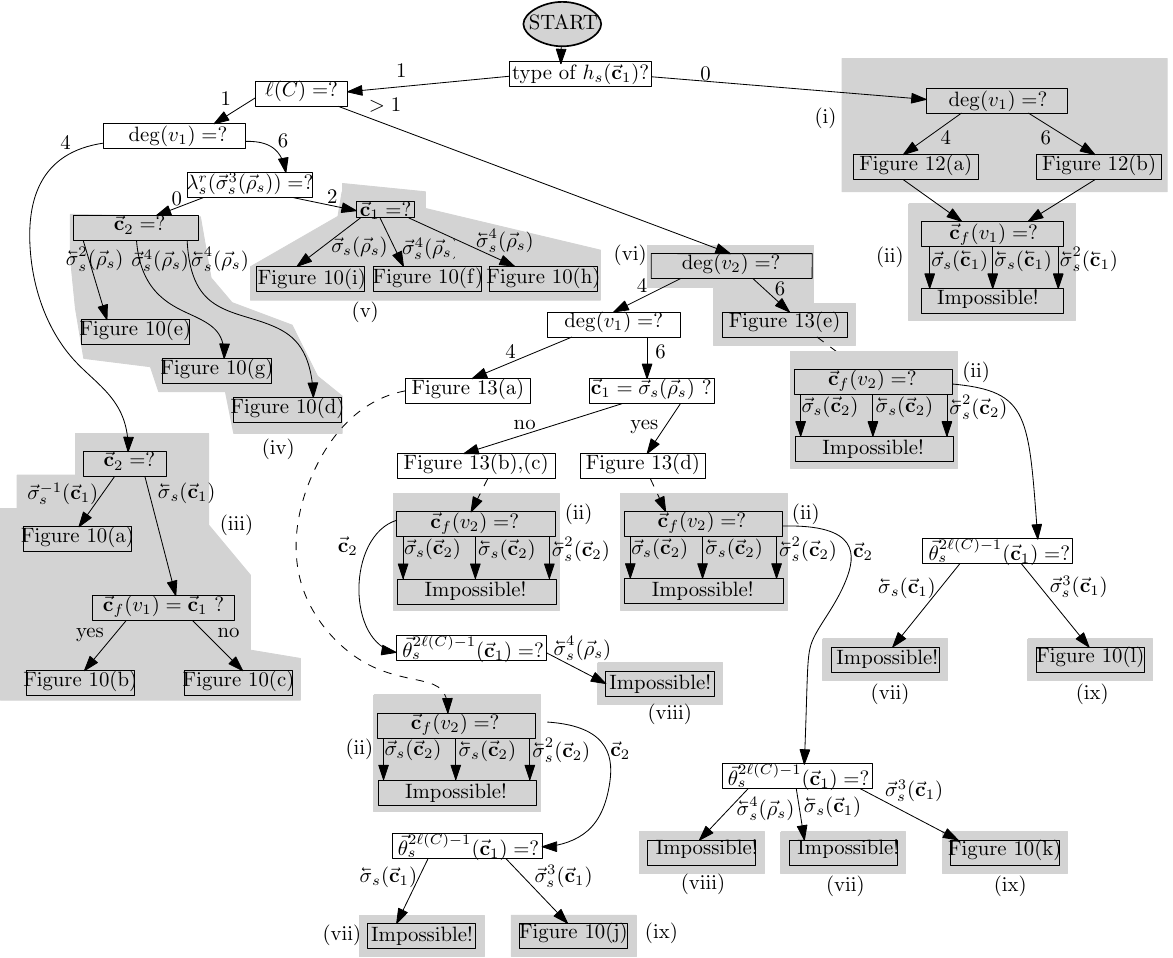}
      \caption{Successive steps in the proof of \cref{lem:structOff}.}
\label{fig:ProofDiagram}
\end{figure}
Exploring the whole diagram
ensures that we have checked all the possibilities. The proof is
rather involved so in order to help the reader understand
successive steps of the proof we enumerated them by
\ref{case1}--\ref{case9} and highlighted in
grey in \cref{fig:ProofDiagram}. Here are the steps we are going to prove:
\begin{enumerate}[label=(\roman*)]
\item
  \label{case1}
 We first consider the case when $h_s(\vec\cc_1)$ is of type $0$. Then 
 $\vec\cc_f(v_1)$ can be one of the following corners $\vec
 \sigma_s(\cev{\cc}_1), \cev
 \sigma_s(\cev{\cc}_1)$ or $\cev
 \sigma_s^2(\cev{\cc}_1)$. The case $\deg(v_1)=4$ is shown in \cref{subfig:case1a} and the
 case $\deg(v_1)=6$ is shown in \cref{subfig:case1b}.
\item
    \label{case2}
On the other hand $\vec\cc_f(v_1)$ cannot be equal to any of the
corners $\vec
 \sigma_s(\cev{\cc}_1), \cev
 \sigma_s(\cev{\cc}_1), \cev
 \sigma_s^2(\cev{\cc}_1)$.
\item
    \label{case3}
We now consider the case when $h_s(\vec\cc_1)$ is of type $1$, when
$\ell(C)=1$ and $\deg(v_1)=4$. We show that $\vec \cc_2$ can be either $\vec
 \sigma^{-1}_s(\vec\cc_1)$, or $\cev
 \sigma_s(\vec\cc_1)$. In the first case $C$ is presented in
 \cref{fig:Cycles}(a). In the second one either $\vec\cc_f(v_1) = \vec\cc_1$
 and necessarily $C$ has the form as displayed in \cref{fig:Cycles}(b) or
$\vec\cc_f(v_1) \neq \vec\cc_1$ and then $C$ has the form as displayed in
\cref{fig:Cycles}(c).
\item
  \label{case4}
When $h_s(\vec\cc_1)$ is of type $1$, $\ell(C)=1$, $\deg(v_1)=6$
and $\lambda_s^r(\sigma_s^3(\vec \rho_s)) = 0$
then necessarily $\vec\cc_1 = \vec \sigma_s(\vec \rho_s)$ and $\vec \cc_2$ is either $\cev
 \sigma^{2}_s(\vec \rho_s), \vec
 \sigma^{4}_s(\vec \rho_s)$, or $\cev
 \sigma^{4}_s(\vec \rho_s)$. In the first case $C$ is presented in
 \cref{fig:Cycles}(e), in the second one in
 \cref{fig:Cycles}(g) and in the last one in
 \cref{fig:Cycles}(d).
 \item
  \label{case5}
When $h_s(\vec\cc_1)$ is of type $1$, $\ell(C)=1$, $\deg(v_1)=6$
and $\lambda_s^r(\sigma_s^3(\vec \rho_s)) = 2$
then necessarily $\vec \cc_2 = \cev \sigma_s^4(\vec \rho_s)$ and $\vec\cc_1$ is either $\vec
 \sigma_s(\vec \rho_s), \vec
 \sigma^{4}_s(\vec \rho_s)$, or $\cev
 \sigma^{4}_s(\vec \rho_s)$. In the first case $C$ is presented in
 \cref{fig:Cycles}(i), in the second one in
 \cref{fig:Cycles}(f) and in the last one in
 \cref{fig:Cycles}(h).
  \item
   \label{case6}
Now we have to check the case when $h_s(\vec\cc_1)$ is of type $1$ and
$\ell(C)>1$. Various possible cases are shown in \cref{fig:casesOffset3}. In all these cases 
 $\vec\cc_f(v_2)$ can be one of the following corners $\vec
 \sigma_s(\vec\cc_2), \cev
 \sigma_s(\vec\cc_2), \cev
 \sigma_s^2(\vec\cc_2)$ or $\vec \cc_2$. 
  \item
    \label{case7}
We need to check various possibilities for $\vec\theta_s^{2\ell(C)-1}(\vec
\cc_1)$. We show that when $h_s(\vec\cc_1)$ is of type $1$ and $\ell(C)>1$ then the case $\vec\theta_s^{2\ell(C)-1}(\vec
\cc_1) = \cev
\sigma_s(\vec\cc_1)$ is impossible.
 \item
    \label{case8}
Also, when $h_s(\vec\cc_1)$ is of type $1$, $\ell(C)>1$ and $\deg(v_1) = 6$ then $\vec
\theta_s^{2\ell(C)-1}(\vec\cc_1)$ cannot be equal to $\cev
\sigma^{4}_s(\vec \rho)$.
 \item
    \label{case9}
The last possibility when $h_s(\vec\cc_1)$ is of type $1$, and $\ell(C)>1$
is that $\vec
\theta_s^{2\ell(C)-1}(\vec\cc_1)=\vec
 \sigma^{3}_s(\vec\cc_1)$. In this case either $\deg(v_2)=6$ and $C$ has the form as in \cref{fig:Cycles}(l), or $\deg(v_2)=4$. In the latter case either $\deg(v_1)=4$ and $C$ has the form as in \cref{fig:Cycles}(j), or $\deg(v_1)=6$ and $C$ has the form as in \cref{fig:Cycles}(k).
\end{enumerate}

\emph{Proof of \ref{case1}}:
The case when $\deg(v_1)=4$ is shown in \cref{subfig:case1a}, and we need to show that when $\deg(v_1)=6$ then $C$ has
necessarily the form as in \cref{subfig:case1b}.
\begin{figure}[h!]
\centering
\subfloat[]{
	\label{subfig:case1a}
	\includegraphics[width=30mm]{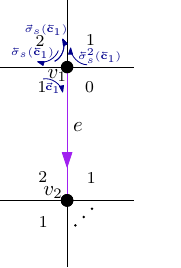}}
      \quad
      \subfloat[]{
	\label{subfig:case1b}
	\includegraphics[width=30mm]{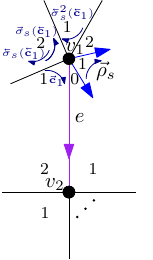}}
      \caption{Two possible cases in the proof of \cref{thm:structOff}\ref{case1}
        when $h_s(\vec\cc_1)$ is of relative type $0$:
        \protect\subref{subfig:case1a} when $\deg(v_1)=4$ and
        \protect\subref{subfig:case1b} when $\deg(v_1)=6$.}
\label{fig:casesOffset1}
      \end{figure}
Notice that
$\vec\sigma_s(\vec \rho_s)$ is the second visited corner performing the tour
and $h_s(\vec\sigma_s(\vec \rho_s))$ is of relative type $1$. Thus,
$e_s(\vec\sigma_s(\vec \rho_s))$ cannot belong to $C$ since it is visited before $e$. 
If $\lambda_s^r(\sigma^3_s(\vec \rho_s))=0$, then all the halfedges
adjacent to $v_1$ and different than $h_s(\vec\sigma_s(\vec \rho_s))$
are of relative type $0$ so there are no offset edges toward $v_1$,
which is a contradiction. By consequence
$\lambda_s^r(\sigma^3_s(\vec\rho_s)))=2$, and $C$ has the form as in
\cref{subfig:case1b}.

\cref{eq:DescentInOffset} implies that $\vec\cc_f(v_1)$ can be only one of the following corners: $\vec
 \sigma_s(\cev{\cc}_1), \cev
 \sigma_s(\cev{\cc}_1)$ or $\cev
 \sigma_s^2(\cev{\cc}_1)$.
 
\vspace{5pt}
\emph{Proof of \ref{case2}}:
\vspace{5pt}

\begin{itemize}
\item $C$ necessarily contains a halfedge of type $1$ adjacent to $v_1$, which can either be $e_s(\vec
 \sigma_s(\cev{\cc}_1))$ or $e_s(\cev
 \sigma_s(\cev{\cc}_1))$;
\item \cref{eq:DescentInOffset} implies that in the case $\vec
  \cc_f(v_1) = \cev\sigma_s^2(\cev{\cc}_1)$, the corresponding edge
  $e_s(\cev\sigma_s^2(\cev{\cc}_1))$ is a bud;
  \item therefore in all possible cases ($\vec
  \cc_f(v_1) = \cev \sigma_s^2(\cev{\cc}_1)$,  $\vec
  \cc_f(v_1) = \vec\sigma_s(\cev{\cc}_1)$ and $\vec
  \cc_f(v_1) = \cev\sigma_s(\cev{\cc}_1)$) both edges $e_s(\vec
 \sigma_s(\cev{\cc}_1))$ and $e_s(\cev
 \sigma_s(\cev{\cc}_1))$ are visited before $e$ which is a contradiction.
\end{itemize}

\vspace{5pt}
\emph{Proof of \ref{case3}}:
\vspace{5pt}

It is clear that when $h_s(\vec\cc_1)$ is of type $1$, $\ell(C)=1$ and
$\deg(v_1)=4$ then either $\vec \cc_2=\vec
 \sigma^{-1}_s(\vec\cc_1)$ or $\vec \cc_2=\cev
 \sigma_s(\vec\cc_1)$. In the first case $e_s(\cev{\cc}_1)$ cannot be a
 stem. Otherwise we will never visit $\sigma^2_s(\cev{\cc}_1)$. The same
 argument implies that also $e_s(\vec
 \sigma^2_s(\cev{\cc}_1))$ is an internal edge. Therefore $C$ has
 necessarily the form presented in
 \cref{fig:Cycles}(a). 

Notice that when $\vec \cc_2=\cev
 \sigma_s(\vec\cc_1)$ then $\vec \theta_s(\vec
\cc_2)=\cev
 \sigma^2_s(\vec\cc_1)$. Therefore either $\vec\cc_f(v_1) = \vec\cc_1$ or
$\vec\cc_f(v_1) = \sigma^3_s(\vec\cc_1)\neq \vec\cc_1$. The labels given by $\lambda_s^r$ imply that in the first case 
$e_s(\vec \sigma_s^2(\vec\cc_1))$ is a bud, which corresponds to
\cref{fig:Cycles}(b) and in the second case
$e_s(\vec \sigma_s^3(\vec\cc_1))$ is a bud, which corresponds to
\cref{fig:Cycles}(c).

\vspace{5pt}
\emph{Proof of \ref{case4}}:
\vspace{5pt}

It is clear that when $h_s(\vec\cc_1)$ is of type $1$, $\ell(C)=1$, $\deg(v_1)=6$
and $\lambda_s^r(\sigma_s^3(\vec \rho_s)) = 0$
then necessarily $\vec\cc_1 = \vec \sigma_s(\vec \rho_s)$ and $\vec \cc_2$ is either $\cev
 \sigma^{2}_s(\vec \rho_s), \vec
 \sigma^{4}_s(\vec \rho_s)$, or $\cev
 \sigma^{4}_s(\vec \rho_s)$.
 \begin{itemize}
    \item when $\vec \cc_2 = \cev
 \sigma^{2}_s(\vec \rho_s)$ the sequence of corners $\vec \rho_s,
 \vec \theta_s(\vec \rho_s), \vec \theta^2_s(\vec \rho_s), \vec
 \theta^3_s(\vec \rho_s)$ coincides with the sequence $\vec \rho_s,
 \vec \sigma_s(\vec \rho_s), \cev \sigma^2_s(\vec \rho_s), \vec
 \sigma^3_s(\vec \rho_s)$. Since the label of $\vec
 \sigma^3_s(\vec \rho_s)$ is smaller then that label of $\vec
 \sigma^4_s(\vec \rho_s)$, the edge $e_s(\vec
 \sigma^3_s(\vec \rho_s))$ is a bud. Finally, $v_1$ is a scheme vertex
 therefore $e_s(\vec\sigma^4_s(\vec
 \rho_s))$ is an internal edge. This corresponds to \cref{fig:Cycles}(e).
    \item when $\vec \cc_2 = \vec
 \sigma^{4}_s(\vec \rho_s)$ then $e_s(\vec \sigma^2_s (\vec\cc_1))$ is an
 internal edge. Otherwise, we will never visit the corner $\sigma^2_s
 (\vec\cc_1)$ performing the tour of $s$. Moreover $e_s(\cev \sigma^{-1}_s(\vec
 \rho_s))$ is an internal edge too, since $v_1$ is a scheme vertex. This corresponds to \cref{fig:Cycles}(g).
     \item when $\vec \cc_2 = \cev
\sigma^{4}_s(\vec \rho_s)$
 the sequence of corners $\vec \rho_s,
 \vec \theta^{-1}_s(\vec \rho_s), \vec \theta^{-2}_s(\vec \rho_s)$ coincides with the sequence $\vec \rho_s,
 \vec \sigma^5_s(\vec \rho_s), \cev \sigma^2_s(\vec \rho_s)$. Since the label of $\cev \sigma^2_s(\vec \rho_s)$ is bigger then that label of $\cev \sigma^3_s(\vec \rho_s)$, the edge $e_s(\cev \sigma^3_s(\vec \rho_s))$ is a bud. Finally, $v_1$ is a scheme vertex
 therefore $e_s(\vec \cc_2)$ is an
 internal edge. This corresponds to \cref{fig:Cycles}(d).
\end{itemize}

\vspace{5pt}
\emph{Proof of \ref{case5}}:
\vspace{5pt}

It is clear that when $h_s(\vec\cc_1)$ is of type $1$, $\ell(C)=1$, $\deg(v_1)=6$
and $\lambda_s^r(\sigma_s^3(\vec \rho_s)) = 2$
then necessarily $\vec \cc_2 = \cev \sigma^4(\vec \rho_s)$ and $\vec\cc_1$ is either $\vec
 \sigma_s(\vec \rho_s), \vec
 \sigma^{4}_s(\vec \rho_s)$, or $\cev
 \sigma^{4}_s(\vec \rho_s)$. 
 \begin{itemize}
    \item when $\vec
\cc_1 = \vec\sigma_s(\vec \rho_s)$ the sequence of corners $\vec \rho_s,
 \vec \theta_s(\vec \rho_s), \vec \theta^2_s(\vec \rho_s)$ coincides with the sequence $\vec \rho_s,
 \vec \sigma_s(\vec \rho_s), \cev \sigma^4_s(\vec \rho_s)$. Since the label of $\cev \sigma^4_s(\vec \rho_s)$ is smaller then that label of $\cev \sigma^3_s(\vec \rho_s)$, the edge $e_s(\cev \sigma^4_s(\vec \rho_s))$ is a bud. Finally, $v_1$ is a scheme vertex
 therefore $e_s(\vec \sigma^2_s(\vec
 \rho_s))$ is an internal edge. This corresponds to \cref{fig:Cycles}(i).
    \item when $\vec
\cc_1 = \vec
 \sigma^{4}_s(\vec \rho_s)$  the sequence of corners $\vec \rho_s,
 \vec \theta^{-1}_s(\vec \rho_s), \vec \theta^{-2}_s(\vec \rho_s), \vec
 \theta^{-3}_s(\vec \rho_s)$ coincides with the sequence $\vec \rho_s,
 \vec \sigma^5_s(\vec \rho_s), \cev \sigma^4_s(\vec \rho_s), \vec
 \sigma^3_s(\vec \rho_s)$. Since the label of $\vec
 \sigma^3_s(\vec \rho_s)$ is bigger then that label of $\vec
 \sigma^2_s(\vec \rho_s)$, the edge $e_s(\vec
 \sigma^2_s(\vec \rho_s))$ is a bud. Finally, $v_1$ is a scheme vertex
 therefore $e_s(\vec\sigma_s(\vec
 \rho_s))$ is an internal edge. This corresponds to \cref{fig:Cycles}(f).
     \item when $\vec
\cc_1 = \cev
 \sigma^{4}_s(\vec \rho_s)$ then $e_s(\vec \sigma^3_s (\vec\cc_1))$ is an
 internal edge. Otherwise, we will never visit the corner $\sigma^3_s
 (\vec\cc_1)$ performing the tour of $s$. Moreover $e_s(\vec \sigma_s(\vec
 \rho_s))$ is also an internal edge strictly from the definition of the
 virtual root. This corresponds to \cref{fig:Cycles}(h).
   \end{itemize}

   \vspace{5pt}
\emph{Proof of \ref{case6}}:
\vspace{5pt}

It is clear that all the possibilities when $\deg(v_2)=4$ are shown in \cref{subfig:case3a,subfig:case3b,subfig:case3c,subfig:case3d}. Therefore we need to show that when $\deg(v_2)=6$ then $C$ has
necessarily the form as in \cref{subfig:case3e}.
\begin{figure}[h!]
        \centering
\subfloat[]{
	\label{subfig:case3a}
	\includegraphics[width=0.18\linewidth]{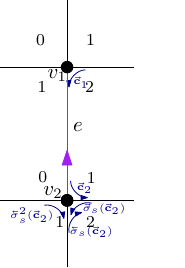}}
            \subfloat[]{
	\label{subfig:case3b}
	\includegraphics[width=0.18\linewidth]{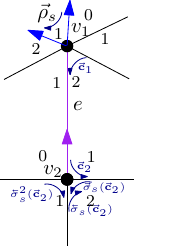}}
            \subfloat[]{
	\label{subfig:case3c}
	\includegraphics[width=0.18\linewidth]{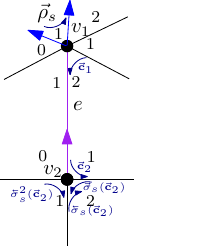}}
              \subfloat[]{
	\label{subfig:case3d}
	\includegraphics[width=0.18\linewidth]{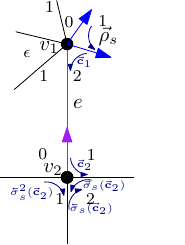}}
            \subfloat[]{
	\label{subfig:case3e}
	\includegraphics[width=0.18\linewidth]{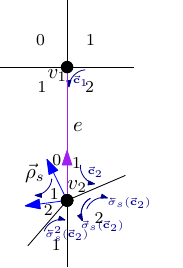}}
      \caption{Five possible cases in the proof of \cref{thm:structOff}\ref{case6}
        when $h_s(\vec\cc_1)$ is of relative type $1$ and $\ell(C) >1$:
        \protect\subref{subfig:case3a} when $\deg(v_2)=\deg(v_1)=4$,
        \protect\subref{subfig:case3b},\protect\subref{subfig:case3c},\protect\subref{subfig:case3d}
        when $\deg(v_1) = 6$ and
        \protect\subref{subfig:case3e} when
        $\deg(v_2)=6$.}y
\label{fig:casesOffset3}
\end{figure}
The halfedge of $C$ of type $1$
adjacent to $v_2$ cannot be $h_s(\vec\sigma_s(\vec\rho_s))$, since
$\vec\sigma_s(\vec\rho_s)$ is visited before $c$. By consequence
$\lambda_s^r(\sigma^3_s(\vec\rho_s)))=2$, and the halfedge of $C$ of type $1$ adjacent to $v_2$ is
either $e_s(\vec \sigma_s^3(\vec \rho_s))$ or $e_s(\cev \sigma_s^3(\vec
\rho_s))$, as shown on \cref{subfig:case3e}.

\cref{eq:DescentInOffset} implies that $\vec\cc_f(v_2)$ can be one of the following corners: $\vec
 \sigma_s(\vec\cc_2)$, $\cev
 \sigma_s(\vec\cc_2)$, $\cev
 \sigma_s^2(\vec\cc_2)$ or $\vec \cc_2$. Moreover,
 identical arguments as used in the case
 \ref{case2} show that in fact  $\vec\cc_f(v_2)$ cannot be equal to any of the
corners $\vec
 \sigma_s(\vec\cc_2)$, $\cev
 \sigma_s(\vec\cc_2)$, $\cev
 \sigma_s^2(\vec\cc_2)$. Thus the only possibility is that
 $\vec\cc_f(v_2) = \vec \cc_2$ and
 necessarily $e_s(\vec \cc_2)$ is a bud. By iterating
      this argument,
    we deduce that there exists a sequence of vertices
    $v_1,\dots,v_{\ell(C)}$ connected by edges $e=e_1,e_2,\dots,e_{\ell(C)}$, where $e_j$ is joining
    $v_{j+1}$ with $v_{j}$ and offset towards $v_j$ ($v_{\ell(C)+1} := v_1$
    by convention). Note that performing a tour of $s$ starting from
    $\vec\cc_1$ we visit an alternating sequence of offset edges and
    buds. In particular, the sequence of visited corners $\vec
    \cc_1,\vec \cc_2,\dots,\vec \theta_s^{2\ell(C)-1}(\vec\cc_1)$ coincides with
    the sequence $\vec\cc_1,$ $\vec\cc_2,$ $\vec\sigma_s \vec\theta_s(\vec
    \cc_1),$ $\vec\theta_s \vec\sigma_s \vec\theta_s(\vec
    \cc_1),$ $\dots,$ $\vec\theta_s(\vec\sigma_s \vec\theta_s)^{\ell(C)-1}(\vec
    \cc_1)$. Therefore $\vec \theta_s^{2\ell(C)-1}(\vec\cc_1)$ can be one of
    the following corners 
\begin{itemize}
\item $\cev
 \sigma_s(\vec\cc_1)$ in
 \cref{subfig:case3a,subfig:case3c,subfig:case3d,subfig:case3e},
\item $\cev
 \sigma^4_s(\vec \rho_s)$ in
 \cref{subfig:case3b,subfig:case3c,subfig:case3d},
\item $\vec
 \sigma^3_s(\vec\cc_1)$ in \cref{subfig:case3a,subfig:case3d,subfig:case3e}.
\end{itemize}

   \vspace{5pt}
\emph{Proof of \ref{case7}}:
\vspace{5pt}

Consider first the case when $\deg(v_2)=4$ (\cref{subfig:case3a,subfig:case3b,subfig:case3c,subfig:case3d}). When $\vec
\theta_s^{2\ell(C)-1}(\vec\cc_1) = \cev
 \sigma_s(\vec\cc_1)$ then $e_s(\vec\theta_s\cev
 \sigma_s(\vec\cc_1))$ is a bud.
        Indeed, it was not visited yet and since $\lambda_s^r\left(\vec\theta_s\cev
 \sigma_s(\vec\cc_1)\right) +1 = \lambda_s^r\left(\vec\sigma_s\vec\theta_s\cev
 \sigma_s(\vec\cc_1)\right)$ the claim follows by \cref{eq:DescentInOffset}. 
        This implies that there are two buds attached to $v_2$,
        thus $v_2$ has to be the root vertex (by the definition of a
        special $4$-valent map). Therefore $\vec\theta_s\cev
 \sigma_s(\vec\cc_1)$ is the
        root corner, which is clearly impossible (for instance it
        will imply that $\vec\cc_f(v_2) = \vec\sigma_s\vec\theta_s\cev
 \sigma_s(\vec\cc_1)$ which contradicts the fact that $\vec\cc_f(v_2) =
\vec\cc_2$). Suppose now that $\deg(v_2)=6$ (\cref{subfig:case3e}). Note that after visiting $\cev
 \sigma_s(\vec\cc_1)$ we are
        visiting the bud preceding the root, thus we performed the
        whole tour of $s$. Since $\vec \cc_2$ is visited
        before $\vec \sigma_s^2(\vec \rho_s)$ this means that after visiting $\vec \cc_2$ we visit successively $\vec
        \sigma_s (\vec\cc_2), \cev
 \sigma_s(\vec\cc_1)$ and the last corner in the whole tour. But
        this means that we never visited $\vec \sigma_s^2(\vec \rho_s)$ which is clearly a
        contradiction.

           \vspace{5pt}
\emph{Proof of \ref{case8}}:
\vspace{5pt}

When $h_s(\vec\cc_1)$ is of type $1$, $\ell(C)>1$, $\deg(v_1) = 6$ and $\vec
\theta_s^{2\ell(C)-1}(\vec\cc_1)=\cev
\sigma^{4}_s(\vec \rho_s)$ then $\vec\theta_s^{2}(\cev \rho_s)$ is
visited after $\vec\sigma_s\vec\theta_s^{2}(\cev \rho_s)$. The
label of $\vec\sigma_s\vec\theta_s^{2}(\cev \rho_s)$ is smaller than
the label of $\vec\theta_s^{2}(\cev \rho_s)$. As a consequence $e_s(\vec\sigma_s\vec\theta_s^{2}(\cev \rho_s))$
is a bud, hence $v_{\ell(C)}$ is adjacent to two buds. This
contradicts that $s$ is a special $4$-valent map, since $v_{\ell(C)}$ is not the root vertex.

           \vspace{5pt}
\emph{Proof of \ref{case9}}:
\vspace{5pt}

When $h_s(\vec\cc_1)$ is of type $1$, $\ell(C)>1$ and $\vec
\theta_s^{2\ell(C)-1}(\vec\cc_1)=\vec
\sigma^{3}_s(\vec\cc_1)$ then either $\deg(v_2)=6$
(\cref{subfig:case3e}), or $\deg(v_2) = 4$. In the latter case when
$\deg(v_1)=6$ then necessarily $\vec\sigma_s(\vec
\rho_s)= \vec\cc_1$ and $\lambda_s^r(\vec\sigma_s^3(\vec \rho_s))=0$ (see
\cref{fig:casesOffset3}). Therefore we have the following possibilities:
$\deg(v_1) = \deg(v_2) = 4$ (\cref{subfig:case3a}), $\deg(v_1) = 6$
(\cref{subfig:case3d}) or $\deg(v_2) = 6$ (\cref{subfig:case3e}). Notice that for each $j \in [1..\ell(C)]$ and for each $v_j$ there
        is exactly one edge $e(v_j)$ adjacent to it (except $v_1$ in \cref{subfig:case3d},
        where there is also the bud preceding the root bud) which was not visited yet in the tour of $s$ starting from its root and finishing in $\vec
\sigma^{3}_s(\vec\cc_1)$. 
        Let $e(v_j)$ be the first edge among $e(v_1),\dots,e(v_{\ell(C)})$ visited in a tour of $s$. 
        This edge is necessarily visited from a corner (non visited yet) with label $1$ (since $\lambda_s^r$ is a decent labeling). 
        We conclude that all edges $e(v_n)$, where $n \neq j$ have to be buds. 
        Indeed, all of them are visited first from a corner with label $0$ and since $\lambda_s^r$ is a decent labeling the claim follows. 
        That means that there exists $k \in [1..\ell(C)]$ such that there are
        two buds attached to $v_k$ unless $\ell(C) =j=2$.
        The existence of such a vertex $v_k$ contradicts the
        assumption that $s$ is a special $4$-valent map. Therefore
        necessarily $\ell(C) = j = 2$. The case when $\deg(v_1) =
        \deg(v_2) = 4$ corresponds to \cref{fig:Cycles}(j), the case
        when $\deg(v_1) = 6$ corresponds to \cref{fig:Cycles}(k), and the case
        when $\deg(v_2) = 6$ corresponds to \cref{fig:Cycles}(l).

\vspace{5pt}
We explored the whole tree--style diagram from \cref{fig:ProofDiagram},
checking all the possible properties of the cycle $C$. This concludes the proof.

\end{proof}

\begin{lemma}\label{lem:structOffRemove}
Let $s\in \widetilde{\mathcal{M}}^\times_\mathbb{S}$ be a special
unlabeled $4$-valent map 
of a surface $\mathbb{S}$, and let $C$ be the first visited offset cycle of
$s$. Then, we can transform $s$ into another special
unlabeled $4$-valent map $s' \in
\widetilde{\mathcal{M}}^\times_{\mathbb{S'}}$ such that:
\begin{itemize}
  \item $\mathbf{E}_{s'} \subset \mathbf{E}_s$ and $C'\neq C$ is an offset cycle in
    $s$
    if and only if it is an offset cycle in $s'$,
    \item $|C| \geq \chi_{\mathbb{S'}} - \chi_{\mathbb{S}}$.
  \end{itemize}
\end{lemma}

\begin{proof}
 We claim that if $C'\neq C$ is an offset simple cycle in
    $s$ then the set of vertices belonging to $C'$ is
    disjoint from the set of vertices belonging to $C$.  We recall that \cref{lem:structOff} implies that the first visited offset cycle
 $C$ of $s$ is of the form presented in
 \cref{fig:Cycles}. Let $v$ be a
 vertex belonging to $C$. Notice that if $v \in C'$ then it necessarily
 has adjacent halfedges of both types $0$ and $1$
which do not belong to $C$ but are contained in internal edges. This
case only appears in \cref{fig:Cycles}(a). In this case $e_s(\cev
    \cc)$ cannot belong to $C'$ since $C$ is the first visited offset
    cycle. As a consequence the set of vertices belonging to $C'$ is
    disjoint from the set of vertices belonging to $C$, as claimed.

 Let $s'$
          be a blossoming map obtained from $s$ by removing
all the vertices of
$C$ and all the halfedges adjacent to removed vertices. In \cref{fig:CyclesInOffset2} we colored halfedges
which were internal edges in $s$ but now become stems in $s'$ by red
or green.
          \begin{figure}[h!!!]
            \centering
	\includegraphics[width=\linewidth]{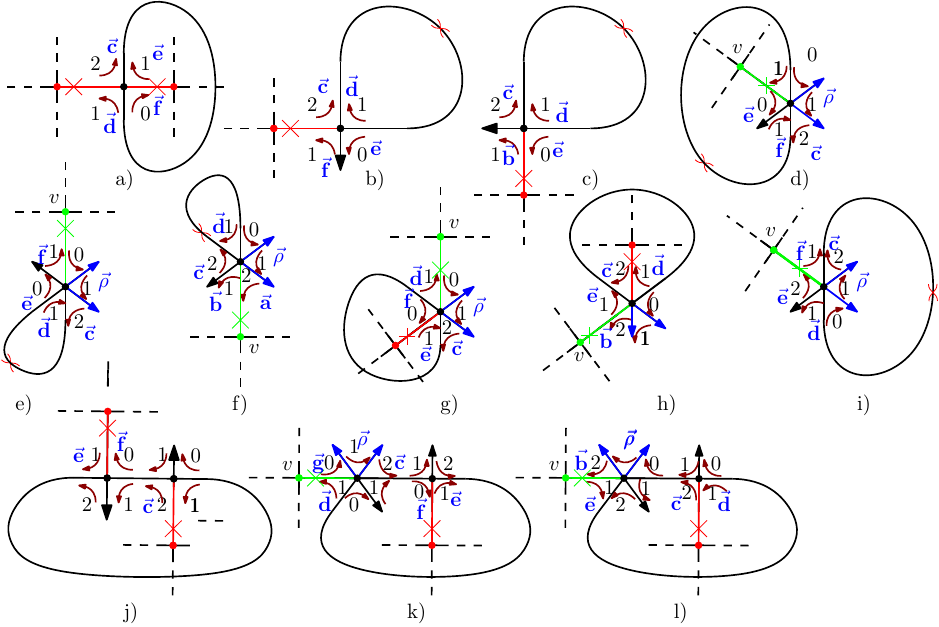}
\caption{Procedure of constructing $s'$ from $s$ described in \cref{lem:structOffRemove}.}
\label{fig:CyclesInOffset2}
      \end{figure}
We set red (green, respectively) stems to be
leaves (buds, respectively) in $s'$. Note that if the root vertex of
$s$ was contained in $C$ (these are precisely the cases when green
halfedges appear) the blossoming map $s'$ is not rooted. Let $v$ be the green vertex 
         adjacent to $v_s(\vec \rho_s)$ by the green edge $e$ (see \cref{fig:CyclesInOffset2}). We are going
to root $s'$ in the vertex $v$ in a way that $s'$ is a bud-rooted map. In order to do
          this we are going to show that when $v$ is of relative type $1$ then the edge sharing an
    adjacent corner to $e$ of label $1$ is necessarily a
    bud.

    It is clear from \cref{fig:CyclesInOffset2} that the first visited
    corner in the tour of $s$ adjacent to $v$ is the
    corner adjacent to $e$ with bigger label, while the
    last visited
    corner in the tour of $s$ adjacent to $v$ is the
    corner adjacent to $e$ with smaller label. Therefore,
    if this bigger label is equal to $2$ then the unique corner
    labeled by $0$ and adjacent to $v$ was visited before the corner labeled by $1$
          and adjacent to $e$ and $v$. This implies that the edge between
          these two corners has to be a bud, since $\lambda_s^r$ is a
          decent labeling. Similarly, if this bigger label is equal to $1$
        then the corresponding corner was visited before the
          unique corner adjacent to $v$, labeled by $2$.
          Therefore the edge between
          these two corners has to be a bud, since $\lambda_s^r$ is a
          decent labeling.
          
          Our analysis shows that when the root vertex of $s$
          was contained in $C$ we can root $s'$ at the vertex $v$ as
          shown in the following \cref{fig:transf}.
          \begin{figure}[h!!!]
            \centering
	\includegraphics[height=35mm]{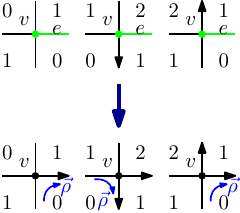}
        \caption{Procedure of rooting $s'$ from
          \cref{lem:structOffRemove}.}
                              	\label{fig:transf}
      \end{figure}
      Now $s'$ is a well-defined rooted blossoming map. Notice that
      our choice of the root is made in a way that the cyclic tour order of $s'$ corresponds to the tour order of $s$
restricted to the set of corners we did not remove (see
\cref{fig:CyclesInOffset2} and the order of visited corners indicated by latin
alphabet). In particular $s'$ is a unicellular map, and the labeling
$\lambda_{s'}^r$ coincides with the labeling $\lambda_{s}^r$ since
we set stems to be buds (leaves, respectively) if the adjacent corner
with the smaller (larger,
respectively) label was visited first in the
tour of $s'$. Moreover, the set of internal edges of $s'$ is a subset of
internal edges of $s$, therefore $\lambda_{s'}^r$ is a
well-labeling. Finally, we already proved that $C'\neq C$ is an offset cycle in
    $s$
    if and only if it is an offset cycle in $s'$. This means that
    $s' \in
\widetilde{\mathcal{M}}^\times_{\mathbb{S'}}$ is a special $4$-valent
map with the properties described in the hypothesis.

Finally, we need to prove that the condition
\[|C| \geq \chi_{\mathbb{S'}} - \chi_{\mathbb{S}}\]
holds true. This simply follows by comparing Euler characteristic
of $s$ and $s'$. Let us have a closer look at the evolution of parameters. 
When $C$ consists of a single edge the total length of offset cycles
decreases by~$1$, and we either remove $2$ edges or
$3$ edges and
$1$ vertex (see \cref{fig:Cycles}), so that Euler characteristic
increases by $1$ or $2$. When $C$ consists of two edges the total length of offset cycles
decreases by~$2$, and we remove $4$ edges and
$2$ vertices (see \cref{fig:Cycles}), so that Euler characteristic
increases by $2$. This finishes the proof.

\end{proof}

\cref{thm:structOff} is an easy corollary of
\cref{lem:structOff} and \cref{lem:structOffRemove}.

\begin{proof}[Proof of \cref{thm:structOff}]
  We will prove it by induction on genus $2g_{\mathbb{S}}$ for any
  special unlabeled $4$-valent map, thus in particular for any unlabeled scheme. For
  $2g_{\mathbb{S}} = 0$ our thesis follows from \cref{lem:structOff}. Assume that $s \in
  \widetilde{\mathcal{M}}^\times_\mathbb{S}$ is a special unlabeled
  $4$-valent map, where
  $2g_{\mathbb{S}} > 0$. Let $C$ be its first visited offset cycle,
  and let $s' \in \widetilde{\mathcal{M}}^\times_\mathbb{S'}$ be a
  special unlabeled
  $4$-valent map associated with $s$ by
  \cref{lem:structOffRemove}. We know by \cref{lem:structOffRemove}
  that $C$ is disjoint with all the other offset cycles of $s$. Moreover $|C| \geq
  \chi_{\mathbb{S'}} - \chi_{\mathbb{S}}$, thus
  \[ 2g_{\mathbb{S'}} \leq 2g_{\mathbb{S}} - |C| <
    2g_{\mathbb{S}}.\]
  Therefore the total length of cycles in the offset graph of $s'$ is
  at most equal to $2g_\mathbb{S'}$, which gives that the total length of cycles in the offset graph of $s$ is
  at most equal to $2g_\mathbb{S'}+|C|\leq 2g_\mathbb{S}$. This
  finishes the proof.
  \end{proof}

  \section{Enumerative results for general maps}
  \label{sec:enumeration}

We now use the results from the previous sections to obtain enumerative results. 
In \cref{sec:exprSurj}, we prove the rationality in $t_\bullet$,
$t_\circ$ and $D$ of the series $C_s^{\tiny\LEFTcircle}(t_\bullet, t_\circ)$, where we
recall that
\[C_s^{\tiny\LEFTcircle}(t_\bullet, t_\circ)  := C_s(t_\bullet, t_\circ)+C_s(t_\circ, t_\bullet),\]
and $C_s$ is the generating series of scheme-rooted cores whose
unlabeled scheme is equal to $s$:
\[ C_{s}(t_\bullet, t_\circ) = \sum_{m\in\mathcal{C}_{s}}
  t_\bullet^{\gamma_m^{r\bullet}}t_\circ^{\gamma_m^{r\circ}}.\]
Note that \cref{prop:BDt} implies that this series can be expressed as
a rational function
in $t_\bullet, t_\circ$ and $B^{-1} = a = \sqrt{(1-2(t_\bullet+ t_\circ))^2-4
  t_\bullet t_\circ}$ (see \cref{cor:rationalityInA}), giving a
combinatorial explanation of rationality expressed in \cref{thm:AG00}. It was proved that in
the case of an orientable scheme the corresponding bivariate
series is actually rational in $t_\bullet$ and $t_\circ$ only. In
\cref{sec:exprPerm} we use a construction from
\cite{Lepoutre:thesis} to prove \cref{thm:bivRatNorNofC} which generalises this result to any scheme with no offset cycle, depending on the parity of Euler characteristic of the underlying surface. This extends the previously
mentioned result from \cite{Lepoutre:thesis}.
Finally, in \cref{sec:ratOfLoop}, we refine \cref{thm:rationalityInA} by
expressing explicitly the generating series $C_s^{\tiny\LEFTcircle}$ of an arbitrary
unlabeled scheme $s$ in terms of the series of the form
$C_{s'}^{\tiny\LEFTcircle}$, where $s'$ is an unlabeled scheme with no
offset cycles (which can hence be treated by \cref{thm:bivRatNorNofC}). 

\subsection{Rationality in $D$}\label{sec:exprSurj}

Recall that Arquès and Giorgetti \cite{ArquesGiorgetti2000} proved that the series of maps on any surface is rational in $t_\bullet$, $t_\circ$, and $a$ (see \cref{thm:AG00}).

By \cref{prop:BDt}, any rational function of~$t_\bullet$, $t_\circ$, $D$, $D_\bullet$, $D_\circ$, and $B$, is also a rational function of~$t_\bullet$, $t_\circ$, and $a$.

The rest of this section will be dedicated to provide a proof of \cref{thm:rationalityInA}:

\begin{theorem}\label{thm:rationalityInA}
For any unlabeled scheme $s$, the series of scheme-rooted cores having unlabeled scheme $s$ is a rational function of $t_\bullet$, $t_\circ$, and $D$.
\end{theorem}

\cref{thm:rationalityInA}, along with \cref{lem:shortcutEnum,cor:bij4}, allows to deduce the following corollary, that gives a combinatorial interpretation of the rationality expressed in \cref{thm:AG00}:

\begin{corollary}\label{cor:rationalityInA}
The bivariate series of maps of a surface $\mathbb{S}$ is a rational function in $t_\bullet$, $t_\circ$, and $B$.
\end{corollary}

\begin{proof}[Proof of \cref{thm:rationalityInA}]
\cref{lem:decompCore} ensures for any labeled scheme $l$, the series $C_l(t_\bullet, t_\circ)$ is a polynomial in $t_\bullet$, $t_\circ$, $D$ and $B$, and is thus rational in $t_\bullet$, $t_\circ$, and $B$. However, since there is an infinite number of labeled schemes having $s$ as an unlabeled scheme, we cannot directly conclude with \cref{thm:rationalityInA}. 

Our strategy is to show that the contribution of all labeled schemes associated to a certain surjection can be written as the contribution of a certain canonical labeled scheme associated to this surjection, multiplied by a finite number of factors of the form $\frac{1}{1-P}$, where $P$ is a simple polynomial in $t_\bullet$, $t_\circ$, and $D$, and where each of these factors account for the difference of height between successive vertices.

However, in the bivariate setup, the weight associated to a branch depends on the parity of its height (see \cref{lem:decompCore}). By consequence, we need to keep track of the parity of these labels, hence leading to the definition of binary surjections.

We call \emph{binary surjection} a pair made of a surjection $p:[1,n]\rightarrow[1,k]$ for $k\in[1,n]$, and a binary sequence $\varepsilon\in\{0,1\}^{k-1}$. Such a surjection is said to have \emph{size} $n$ and \emph{height} $k$. The set of binary surjections of size $n$ is denoted $\mathbf{BS}_n$.

Let $s$ be an unlabeled scheme, and $l$ be its labeled version.
We arbitrarily denote the vertices of $s$ (and accordingly, those of
$l$) by: $v_1,\cdots, v_{n}$, where $n=n_s^v $.
The \emph{binary surjection associated to $l$} is the binary
surjection $(p,\varepsilon) \in \mathbf{BS}_n$ such that:
\begin{compactitem}
    \item $p(i)=p(j) \iff \lambda_l(v_i)=\lambda_l(v_j)$, (we
      consequently denote $\lambda_l(v_{p^{-1}(i)})$ the common height
      of vertices indexed by $p^{-1}(i)$)
    \item $p(i)<p(j) \iff \lambda_l(v_i)<\lambda_l(v_j)$,
    \item $\lambda_l(v_{p^{-1}(i+1)})-\lambda_l(v_{p^{-1}(i)}) \equiv \varepsilon_i \mod{2}$. 
\end{compactitem}

Note that the size $n$ of this binary surjection is equal to the
number $n^v_s$ of vertices in $s$, while its height $k$ is equal to
the number of distinct labels of vertices of $l$. The order on the
image of $p$ coincides with the order on the labels of vertices of
$l$. In other terms, $p$ tells us that among all the $n=n_s^v$ labels
of the vertices of $l$ there are $k$ distinct labels, and $p$ groups the
vertices with respect to their labels: it starts from the smallest label (the corresponding vertices
have indices in the preimage $p^{-1}(1)$) and finishing with the
largest label (the corresponding vertices
have indices in the preimage $p^{-1}(k)$).

The set of labeled schemes associated to the binary surjection $(p,\varepsilon)$ is denoted $\mathcal{L}_{(p,\varepsilon)}$.
The \emph{canonical labeled scheme of binary surjection $(p,\varepsilon)$}, denoted $l_{(p,\varepsilon)}$, is the labeled scheme $l$ such that $\lambda_l(v_{p^{-1}(i+1)})-\lambda_l(v_{p^{-1}(i)}) = 2-\epsilon_i$ for any $i\in[1,k-1]$. In other word, it is the labeled scheme of $\mathcal{L}_{(p,\varepsilon)}$ which minimizes the highest difference of height between two vertices.

There is a bijection between sequences of non-negative integers of
length $k-1$, where $k$ is the height of $p$, and labeled schemes in
$\mathcal{L}_{(p,\varepsilon)}$: to a sequence $(a_i)_{i\in[1,k-1]}$
we associate the labeled scheme $l$ such that
$\lambda_l(v_{p^{-1}(i+1)})-\lambda_l(v_{p^{-1}(i)}) =
2(a_i+1)-\epsilon_i$ for any $i\in[1,k-1]$. Note that the canonical
labeled scheme is associated with the sequence $(0,\dots,0)$.

Let $s$ be an unlabeled scheme, and $p$ a surjection. The set $\mathbf{E}_{s,p}^{\nearrow}$ (resp. $\mathbf{E}_{s,p}^{\searrow}$, $\mathbf{E}_{s,p}^{\nwarrow}$, $\mathbf{E}_{s,p}^{\swarrow}$) is defined as the set of edges of $s$ corresponding to edges of $\mathbf{E}_l^{\nearrow}$ (resp. $\mathbf{E}_l^{\searrow}$, $\mathbf{E}_l^{\nwarrow}$, $\mathbf{E}_l^{\swarrow}$), for some $l\in \mathcal{L}_{(p,\varepsilon)}$.
We denote $\mathtt{C}_{s,p}^{i\nearrow}$ (resp. $\mathtt{C}_{s,p}^{i\searrow}$, $\mathtt{C}_{s,p}^{i\nwarrow}$, $\mathtt{C}_{s,p}^{i\swarrow}$) the number of edges of $\mathbf{E}_{s,p}^{\nearrow}$ (resp. $\mathbf{E}_{s,p}^{\searrow}$, $\mathbf{E}_{s,p}^{\nwarrow}$, $\mathbf{E}_{s,p}^{\swarrow}$) that can be written $(u,v)$, where $p(u)\leq i$ and $p(v)>i$.
We also define $\mathtt{C}_{s,p}^{i}\coloneqq \mathtt{C}_{s,p}^{i\nearrow}+\mathtt{C}_{s,p}^{i\searrow}+\mathtt{C}_{s,p}^{i\nwarrow}+\mathtt{C}_{s,p}^{i\swarrow}$.

The bijection $f$, in conjunction with \cref{lem:decompCore}, leads to the following:
\begin{equation}
    \sum_{l\in\mathcal{L}_{(p,\varepsilon)}}C_l^{\tiny\LEFTcircle}=
    C^{\tiny\LEFTcircle}_{l_{(p,\varepsilon)}}\cdot \prod_{i\in[1,k-1]}\frac{1}{1-D^{2\mathtt{C}_{s,p}^{i}}\cdot(t_\bullet t_\circ)^{(\mathtt{C}_{s,p}^{i\nearrow}+\mathtt{C}_{s,p}^{i\searrow}+2\mathtt{C}_{s,p}^{i\nwarrow})}}.
\end{equation}

By summing over all possible binary surjection:
\begin{equation}
    C_{s}^{\tiny\LEFTcircle}=\sum_{(p,\varepsilon)\in\mathbf{BS}_{n_s^v}}C^{\tiny\LEFTcircle}_{l_{(p,\varepsilon)}}\cdot \prod_{i\in[1,k-1]}\frac{1}{1-D^{2\mathtt{C}_{s,p}^{i}}\cdot(t_\bullet t_\circ)^{(\mathtt{C}_{s,p}^{i\nearrow}+\mathtt{C}_{s,p}^{i\searrow}+2\mathtt{C}_{s,p}^{i\nwarrow})}}.
\end{equation}

Note that \cref{lem:decompCore} (along with \cref{prop:BDt}) ensures
the rationality in $t_\bullet$, $t_\circ$ and $D$ of the series
$C_{l_{(p,\varepsilon)}}$. Thus, \cref{lem:shortcutEnum} and the fact
that $\mathbf{BS}_{n_s^v}$ is finite allows to conclude the proof.
\end{proof}

\subsection{Rationality for schemes with no offset cycle}\label{sec:exprPerm}

We now adapt the construction of \cite{Lepoutre:thesis}, that groups labeled
schemes depending on their relative ordering, represented this time by
a permutation rather than a surjection. This allows to prove
\cref{thm:bivSymNorNofC}, which is a stronger result than
\cref{thm:rationalityInA}, but which can only be applied to some
schemes. 

The purpose of this section is to extend Theorem 2.5.1 of \cite{Lepoutre:thesis} to the following theorem:

\begin{theorem}\label{thm:bivSymNorNofC}
Let $s$ be an unlabeled scheme of genus $g$ with no offset cycle. Then the series $C_{s}^{\tiny \LEFTcircle}(t_\bullet, t_\circ)$ is:
\begin{itemize}
\item $\parallel$-symmetric in $D_\bullet$ and $D_\circ$, if $g$ is an integer,
\item $\parallel$-antisymmetric in $D_\bullet$ and $D_\circ$, if $g$ is not an integer.
\end{itemize}
\end{theorem}

\cref{thm:bivSymNorNofC}, in addition to \cref{prop:BDt,lem:criterion,lem:shortcutEnum}, leads to the following result:
\begin{corollary}\label{thm:bivRatNorNofC}
Let $s$ be an unlabeled scheme of genus $g\geq 1$ with no offset cycle.

The series $\displaystyle\frac{C_s^{\tiny\LEFTcircle}(t_\bullet, t_\circ)}{B^{2g}}$ is rational in $t_\bullet$ and $t_\circ$.
\end{corollary}

\begin{proof}[Proof of \cref{thm:bivSymNorNofC}]

The proof of \cref{thm:bivSymNorNofC} is very similar of that of \cite[Theorem 2.5.1]{Lepoutre:thesis}, which is developed in \cite[Section 2.5]{Lepoutre:thesis}. Since we only need small modifications to be able to prove \cref{thm:bivSymNorNofC}, we will simply reuse all notations and definitions given there, and refer the reader to \cite{Lepoutre:thesis} for more details. 
Note that the definition of a consistent labeling of a scheme is only valid if that scheme has no offset cycle.

The series ${}^c\!S_{s,\nu}^{(\pi,\zeta),k:\pm}$ is redefined as follows:
\begin{align}
    {}^c\!S_{s,\nu}^{(\pi,\zeta),k:\pm}=&
    \sum_{l\in{}^c\!\mathcal{L}_{s,\nu}^{(\pi,\zeta),k:+}}\left(
     \prod_{e\in \mathbf{E}^{\nearrow}_l}\Delta_{h_l^{0,k:\pm}(e)}^{h_l^{1,k:\pm}(e)}(D_{\bullet},D_\circ)
    \prod_{e\in \mathbf{E}^{\searrow}_l}\Delta_{h_l^{1,k:\pm}(e)}^{h_l^{0,k:\pm}(e)}(D_\circ,D_{\bullet})\right.\\
    &\qquad\cdot\left.
    \prod_{e\in \mathbf{E}^{\nwarrow}_l}\Delta_{h_l^{0,k:\pm}(e)}^{h_l^{1,k:\pm}(e)}(D,D)
    \prod_{e\in \mathbf{E}^{\swarrow}_l}\Delta_{h_l^{1,k:\pm}(e)}^{h_l^{0,k:\pm}(e)}(t_\bullet t_\circ D,t_\bullet t_\circ D)
    \right).
\end{align}

The need to distinguish the forward and backward edges (because of \cref{lem:decompCore}), leads to the following refinement of the definition of ${}^cW_{s,\nu,e}^{(\pi,\zeta),k}$:

\begin{equation}
    {}^cW_{s,\nu,e}^{(\pi,\zeta),k}\coloneqq 
    \begin{dcases}
    	W_{s,\nu,e}^{\pi,\zeta(k),k}(D_{-c},D_c),&\text{ if } e\in \mathbf{E}_{s,\nu}^{\pi,k,\nearrow},\\
    	W_{s,\nu,e}^{\pi,\zeta(k),k}(D_c,D_{-c}),&\text{ if } e\in \mathbf{E}_{s,\nu}^{\pi,k,\searrow},\\
    	W_{s,\nu,e}^{\pi,\zeta(k),k}(D,D),&\text{ if } e\in \mathbf{E}_{s,\nu}^{\pi,k,\nwarrow},\\
    	W_{s,\nu,e}^{\pi,\zeta(k),k}(t_\bullet t_\circ D,t_\bullet t_\circ D),&\text{ if } e\in \mathbf{E}_{s,\nu}^{\pi,k,\swarrow},\\
    	1,&\text{ otherwise.}
    \end{dcases}.
\end{equation}

Using \cref{lem:decompCore}, it is then easy to see that \cite[Lemma
2.5.16]{Lepoutre:thesis} still holds in this slightly more general
setup. Similarly, the following equality
\begin{equation}
    \overline{{}^cW_{s,\nu,e}^{(\pi,\zeta),k}}^\parallel=\dfrac{{}^{\left((-1)^{(\zeta(k)+1)}c\right)}W_{s,\nu,e}^{\overline{(\pi,\zeta)},\bar{k}-1}}{D_\bullet D_\circ}
\end{equation}
that was first stated in \cite[Eq. (2.891)]{Lepoutre:thesis}, still holds in this
setup, even for forward edges. By consequence, \cite[Lemma 2.5.17]{Lepoutre:thesis} also holds in this setup.

Following the end of the proof of \cite[Theorem 2.5.1]{Lepoutre:thesis}, we obtain:
\begin{equation}
    \overline{C^{\scalebox{0.7}\LEFTcircle}_s}=(-1)^{n_s^e+n_s^v+1}C^{\scalebox{0.7}\LEFTcircle}_s.
\end{equation}

Euler formula and \cref{lem:criterion} allow to conclude the proof of \cref{thm:bivSymNorNofC}.

\end{proof}

\subsection{Rationality for arbitrary schemes}\label{sec:ratOfLoop}

In the previous section we proved that for each unlabeled scheme $s$ of
genus $g$ with no offset cycle the corresponding series $\frac{C_{s}^{\tiny
  \LEFTcircle}(t_\bullet, t_\circ)}{B^{2g}}$ is rational in
$t_\bullet,t_\circ$. We know that these schemes are the ``building
blocks'' of the
bivariate generating series of \emph{orientable} maps of genus $g$. On the other hand, for any $g \geq 1$ the
bivariate generating series of \emph{general} maps of genus $g$ has the form
\[ f(t_\bullet,t_\circ )+B f'(t_\bullet,t_\circ ),\]
where both $f,f'$ are non-zero rational functions in
$t_\bullet,t_\circ$, and it is natural to ask what is the contribution
of the series $\frac{C_{s}^{\tiny
  \LEFTcircle}(t_\bullet, t_\circ)}{B^{2g}}$ to $f$ and $f'$ for an arbitrary unlabeled
scheme $s$ of genus $g$. The purpose of this section is to show that even
in the general case, when schemes might have offset cycles, the schemes with no offset cycle
are still the ``building blocks'' of the series $\frac{C_{s}^{\tiny
  \LEFTcircle}(t_\bullet, t_\circ)}{B^{2g}}$ for an arbitrary unlabeled
scheme $s$ of genus $g$. We are going to use our classification of
unlabeled schemes of genus $g$ given in \cref{sec:offset} to show how the series $\frac{C_{s}^{\tiny
  \LEFTcircle}(t_\bullet, t_\circ)}{B^{2g}}$ of an arbitrary unlabeled
scheme $s$ of genus $g$ can be expressed explicitly in terms of series $\frac{C_{s'}^{\tiny
  \LEFTcircle}(t_\bullet, t_\circ)}{B^{2g}}$, where $s'$ are unlabeled
schemes of
genus $g$ with no offset cycle (thus rational in $t_\bullet,t_\circ$).

\vspace{5pt}

Let $s$ be an unlabeled
scheme of
genus $g$. We know from \cref{sec:offset} that $s$ only has
offset cycles of length $1$ (denoted $C'_1,\dots,C_J'$ with $J\geq 0$)
or $2$ (denoted $C_1,\dots,C_K$ with $K \geq 0$), and that these cycles are pairwise disjoint. For any offset cycle $C$ of $s$ we
define the scheme $\xi_C(s)$ by adding a leaf on each side of the halfedges of $C$ that has type $1$, so as to make these halfedges have
type $0$ as illustrated in \cref{fig:transOfLoop}.

\begin{figure}
  \centering
  \subfloat[]{
    	\label{subfig:TransLoop}
        \includegraphics[scale=.8]{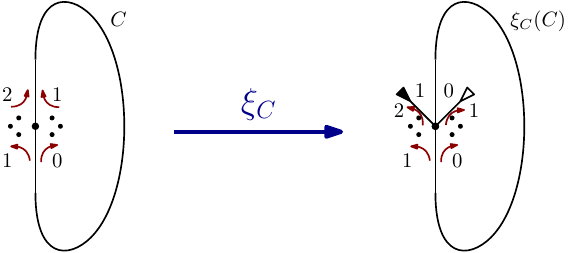}}
        \subfloat[]{
    	\label{subfig:TransLoop1}
    \includegraphics[scale=.8]{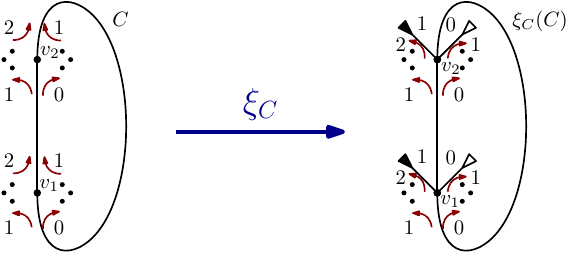}}
    \caption{The operation $\xi_C$ on an offset loop $C$.}
    \label{fig:transOfLoop}
  \end{figure}

  Moreover, if $C$ is an offset cycle of length $2$ then it has
  necessarily the form depicted by one of \cref{fig:Cycles}(j)--\cref{fig:Cycles}(l).
  Let $\vec\cc $ be the corner as in
  \cref{fig:Cycles}(j)--\cref{fig:Cycles}(l) and let $\xi_C'(s)$ denote the scheme obtained from $s$
  by identifying both vertices belonging to $C$ along the edge
  $e_s(\vec\cc)$ and by removing the buds $e_s(\vec\theta_s(\vec \cc))$
  and $e_s(\vec\sigma_s(\vec\cc))$. Geometrically, it corresponds to
  replace the offset cycle $C$ of length $2$ by an offset loop $\xi'_C(C)$, such
  that if $C$ has the form depicted by \cref{fig:Cycles}(j),
    \cref{fig:Cycles}(k), and \cref{fig:Cycles}(l) , then $\xi'_C(C)$ has the form depicted by \cref{fig:Cycles}(a),
          \cref{fig:Cycles}(g) and \cref{fig:Cycles}(h),
              respectively. Finally, we define the scheme
              $\xi''_C(s) := \xi_{\xi_C'(C)}(\xi'_C(s))$ by adding a
              leaf on each side of the halfedge of $\xi'_C(C)$ that
              has type $1$, so as to make this halfedge have type
              $0$. In other terms, we first replace the offset cycle
              $C$ of length $2$ by the appropriate offset loop, and
              then we add a leaf on each side of the halfedge of the
              obtained loop that
              has type $1$.
              For any subset $I \subset [1,K]$ we define the unlabeled
              scheme $s_I$ with no offset cycle by performing the
              following operations:
              \begin{itemize}
                \item first applying operators
              $\xi''_{C_i}$ to $s$, where $i$ runs over $I$, so that
              we obtain the scheme $\bigcirc_{i \in I}\xi''_{C_i}(s)$
              (the notation $\bigcirc$ denotes the composition of
              the operators; note that the order in which we
              apply the operators does not matter
              since the cycles $C_1,\dots,C_K$ are disjoint);
              \item then applying operators $\xi_{C_i}$ to
              the already obtained $\bigcirc_{i \in I}\xi''_{C_i}(s)$, where $i$ runs over
              the complement of $I$;
              \item finally by applying successively the
              operators $\xi_{C'_1}, \dots, \xi_{C'_J}$ 
              to the obtained scheme $\bigcirc_{i \in
                  I^c}\xi_{C_i}\bigcirc_{i \in I}\xi''_{C_i}(s)$.
                \end{itemize} In
                other words
              \begin{equation}
                s_I := \bigcirc_{i=1}^J\xi_{C'_i}\bigcirc_{i \in
                  I^c}\xi_{C_i}\bigcirc_{i \in I}\xi''_{C_i}(s).
                                \label{eq:associatedScheme}
              \end{equation}
              Geometric interpretation of $s_I$ is the following: we replace all the
              offset cycles in $s$, except $\{C_i : i \in I\}$, by balanced
              edges as in \cref{fig:transOfLoop} and we replace the
              cycles $\{C_i : i \in I\}$ first by offset loops, and
              then by balanced
              edges as in \cref{fig:transOfLoop}; see
              \cref{fig:ExampleOfTransformation}.

              \begin{figure}
  \centering
    \includegraphics[width=\linewidth]{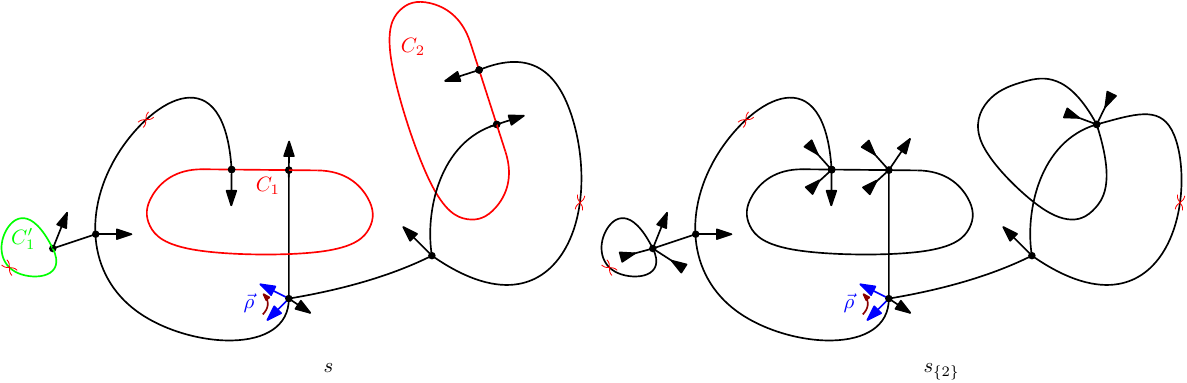}
    \caption{The unlabeled scheme $s$ of genus $\frac{5}{2}$ with
      $1$ offset loop and $2$ offset cycles of length $2$, and the
      associated unlabeled scheme $s_{\{2\}}$ of the same genus with no offset cycle.}
    \label{fig:ExampleOfTransformation}
  \end{figure}

\begin{theorem}\label{thm:symOfLoop}
For any unlabeled scheme $s$ of genus $g\geq 1$, the series of
scheme-rooted cores having unlabeled scheme $s$ is a rational function
of $t_\bullet$, $t_\circ$, and $D$ of the particular form:
\[\displaystyle\frac{C_s^{\tiny\LEFTcircle}(t_\bullet, t_\circ)}{B^{2g}} = \sum_{I \subset [1,K]} (-1)^{|I|} D^{|I|+J}(t_\bullet t_\circ)^{|I|-K}
    \frac{C_{s_I}^{\tiny\LEFTcircle}(t_\bullet, t_\circ)}{B^{2g}},\]
where $J$ is the number of offset loops in $s$, $K$ is the number of
offset cycles of length $2$ in $s$, and $s_I$ is the associated
unlabeled scheme of genus $g$ with no offset cycle. 
\end{theorem}

\begin{proof}

First, suppose that $s$ has an offset loop, denoted $C$. We define $\mathcal{C}_{\xi_C(s)}$ to be the family of scheme-rooted cores having $\xi_C(s)$ as unlabeled scheme, and whose non-scheme vertices all have degree $4$. 
When computing $C_s^{\tiny\LEFTcircle}(t_\bullet, t_\circ)$, the loop $C$ always has weight $B D t_\bullet t_\circ$ (by \cref{lem:decompCore}), whereas when computing $C_{\xi_C(s)}^{\tiny\LEFTcircle}(t_\bullet, t_\circ)$, the modified loop $\xi_C(C)$ has weight $B$. Since, apart from the two additional leaves, that have $2$ different colors, the rest of $s$ and $\xi_C(s)$ are exactly the same, we obtain the following equality:
\begin{equation}
    C_s^{\tiny\LEFTcircle}(t_\bullet, t_\circ)=D\cdot C_{\xi_C(s)}^{\tiny\LEFTcircle}(t_\bullet, t_\circ).
  \end{equation}

Suppose now that $s$ has an offset cycle $C$ of length $2$. For any integers $i,j$ we define $\mathcal{C}^{i,j}_{\xi_C(s)}$ to be the family of scheme-rooted
cores having $\xi_C(s)$ as unlabeled scheme with labels $i,j$ in
vertices $v_1$ and $v_2$ respectively, and whose non-scheme
vertices all have degree $4$. Computing both $C_s^{\tiny\LEFTcircle}(t_\bullet, t_\circ)$ and
$C_{\xi_C(s)}^{\tiny\LEFTcircle}(t_\bullet, t_\circ)$ we are summing the corresponding
generating functions of $\mathcal{C}^{i,j}_{s}$ and
$\mathcal{C}^{i,j}_{\xi_C(s)}$, respectively. Let us fix $i,j$ and let
us compare the corresponding generating functions.
In the case of $\mathcal{C}^{i,j}_{s}$ \cref{lem:decompCore} shows
that the total weight of $C$ is equal to
$B^2 D^{2|i-j|} (t_\bullet t_\circ)^{|i-j|+1}$ when $i\neq j$, and
$B^2 D^{2}(t_\bullet t_\circ)^2$ when $i=j$. In the case of
$\mathcal{C}^{i,j}_{\xi_C(s)}$ we have that the total weight of $\xi_C(C)$
is always equal to
$B^2 D^{2|i-j|} (t_\bullet t_\circ)^{|i-j|}$. Since, apart from the
two additional white and two black leaves, the rest of $s$ and $\xi_C(s)$ are exactly the same, we obtain the following equality:
\begin{equation}
  \label{eq:takietam''}
    C_s^{\tiny\LEFTcircle}(t_\bullet, t_\circ)=\frac{1}{t_\bullet t_\circ}\cdot
    C_{\xi_C(s)}^{\tiny\LEFTcircle}(t_\bullet, t_\circ) + \frac{t_\bullet t_\circ D^2 - 1}{t_\bullet t_\circ}\cdot \tilde{C}_{\xi_C(s)}^{\tiny\LEFTcircle}(t_\bullet, t_\circ),
  \end{equation}
  where $\tilde{C}_{\xi_C(s)}^{\tiny\LEFTcircle}(t_\bullet, t_\circ)$ is the generating
  function of the corresponding labeled cores with the same label of
  $v_1$ and $v_2$. We will show that
  $\tilde{C}_{\xi_C(s)}^{\tiny\LEFTcircle}(t_\bullet, t_\circ)$ can be easily expressed by
  $C_{\xi''_C(s)}^{\tiny\LEFTcircle}(t_\bullet, t_\circ)$. We already computed that the total weight of $\xi_C(C)$
              in $\tilde{C}_{\xi_C(s)}^{\tiny\LEFTcircle}(t_\bullet, t_\circ)$ is equal to
              $B^2$. On the other hand the total weight of $\xi''_C(C)$
              in $C_{\xi''_C(s)}^{\tiny\LEFTcircle}(t_\bullet, t_\circ)$ is equal to
              $B$ (this computation also already appeared at the
              beginning of this proof). Since there are two
              white and two black
              leaves attached to $\xi_C(C)$ in $\xi_C(s)$ and there is
              only one white and one black
              leaf attached to $\xi''_C(C)$ in $\xi''_C(s)$ we get the following equality
\[\tilde{C}_{\xi_C(s)}^{\tiny\LEFTcircle}(t_\bullet, t_\circ)=B t_\bullet t_\circ \cdot
  C_{\xi''_C(s)}^{\tiny\LEFTcircle}(t_\bullet, t_\circ).\]
Plugging it into \cref{eq:takietam''} we obtain
\begin{align}
    C_s^{\tiny\LEFTcircle}(t_\bullet, t_\circ)&=\frac{1}{t_\bullet t_\circ}\cdot
    C_{\xi_C(s)}^{\tiny\LEFTcircle}(t_\bullet, t_\circ) + B\left(t_\bullet t_\circ D^2 -
      1\right)\cdot C_{\xi''_C(s)}^{\tiny\LEFTcircle}(t_\bullet,
                                                t_\circ) \\
  &= \frac{1}{t_\bullet t_\circ}\cdot
    C_{\xi_C(s)}^{\tiny\LEFTcircle}(t_\bullet, t_\circ) - D\cdot C_{\xi''_C(s)}^{\tiny\LEFTcircle}(t_\bullet, t_\circ) ,
  \end{align}
  where the last equality comes from the relation between $B$ and $D$
  given by \cref{prop:BDt}.

Now, let $s$ have $J$ offset loops $C'_1,\dots,C_J'$ and $K$ offset cycles $C_1,\dots,C_K$ of length
$2$. We know by \cref{thm:structOff} that all these cycles are
pairwise disjoint, therefore we can apply the same operation
on each offset cycle of $s$, and obtain the desired equality
\begin{equation}
    C_s^{\tiny\LEFTcircle}(t_\bullet, t_\circ) = \sum_{I \subset [1,K]} (-1)^{|I|} D^{|I|+J}(t_\bullet t_\circ)^{|I|-K}
    C_{s_I}^{\tiny\LEFTcircle}(t_\bullet, t_\circ),
  \end{equation}
where $s_I $ is given by \cref{eq:associatedScheme}. This finishes the proof.
\end{proof}

\begin{remark}
  \label{rem:SpecialForm}
Note that the special case of an unlabeled scheme $s$ of genus $g$ with no offset
cycle of length $2$ gives a particularly nice form of the generating series:
\[\displaystyle\frac{C_s^{\tiny\LEFTcircle}(t_\bullet, t_\circ)}{B^{2g}} = D^{l}
  \frac{C_{s_\emptyset}^{\tiny\LEFTcircle}(t_\bullet, t_\circ)}{B^{2g}}.\]
It is tempting to look for an expression in the case of arbitrary unlabeled scheme $s$
of the similar form:
\[\displaystyle\frac{C_s^{\tiny\LEFTcircle}(t_\bullet, t_\circ)}{B^{2g}} = f(t_\bullet,t_\circ,D)
  \frac{C_{s'}^{\tiny\LEFTcircle}(t_\bullet, t_\circ)}{B^{2g}},\]
where $s'$ is an unlabeled scheme with no offset cycle, and
$f(t_\bullet,t_\circ,D)$ is an explicit rational function in
$t_\bullet,t_\circ,D$ which depends only on the number of offset
cycles on length $1$ and $2$. However, our analysis suggests that this
might be impossible. Most likely even in the simplest case of a scheme $s$ with zero
offset loops and $1$ offset cycle of length $2$ both ratios
\[ \frac{C_{s}^{\tiny\LEFTcircle}(t_\bullet, t_\circ)}{C_{s_{\emptyset}}^{\tiny\LEFTcircle}(t_\bullet, t_\circ)}, \frac{C_{s}^{\tiny\LEFTcircle}(t_\bullet, t_\circ)}{C_{s_{\{1\}}}^{\tiny\LEFTcircle}(t_\bullet, t_\circ)} \]
seem to depend on $s$, and there does not seem to be any other natural candidate for $s'$.
\end{remark}

\newpage

\appendix

\section{Glossary}
\label{glossary}

\begin{center}{\bf Basics (symbols and definitions)}\end{center}

  General convention for using fonts:

  \vspace{5pt}
  
    \begin{tabular}{ll}   
    $\mathcal{A}$ & A specific set of maps (e.g. bipartite maps,
                    well-rooted maps, etc.)\\
    $A$ & The (multivariate) generating function for $\mathcal{A}$\\
    $\mathbf{B}_a$ & A certain set associated with $a \in \mathcal{A}$
                     (e.g. a set of vertices of $a$, a set of faces
                     of $a$, etc.)
    \end{tabular}

\begin{minipage}{.55\textwidth}

  List of symbols:

    \vspace{5pt}
    
  \begin{tabularx}{\textwidth}{ l X } 
    $\mathbb{S}$ & A surface\\
    $\mathcal{M}_\mathbb{S}$ & The set of maps of $\mathbb{S}$\\
    $\mathcal{BP}_\mathbb{S}$ 
                 & The set of bipartite maps of $\mathbb{S}$\\
    $\mathcal{BP}^\square_\mathbb{S}$ & The set of bipartite quadrangulations of $\mathbb{S}$\\
    $\mathcal{BC}_\mathbb{S}$ & The set of bicolorable maps of
                                                                 $\mathbb{S}$\\
    $\mathcal{BC}^\times_\mathbb{S}$ & The set of bicolorable cubic maps of
                                       $\mathbb{S}$\\
        $\mathcal{B}_\mathbb{S}$ & The set of well-blossoming maps of
                      $\mathbb{S}$\\
    $\mathcal{R}_\mathbb{S}$ & The set of well-rooted maps of
                      $\mathbb{S}$\\
    $\mathcal{C}_\mathbb{S}$ & The set of scheme-rooted cores of
                               $\mathbb{S}$\\
    $\mathcal{L}_\mathbb{S}$ & The set of labeled schemes of
                               $\mathbb{S}$\\
    $\mathcal{U}_\mathbb{S}$ & The set of unlabeled schemes of
                               $\mathbb{S}$\\
    $\mathcal{R}_{\overline{s}}$ & The set of maps of $\mathcal{R}$
                                   that have $\overline{s}$ as an unrooted scheme\\
    $\mathcal{C}_{\overline{s}}$ & The set of maps of $\mathcal{C}$
    that have $\overline{s}$ as an unrooted scheme\\
    $\mathcal{M}_{\overline{s}}$ & The set of maps whose opening have
                                   $\overline{s}$ as an unrooted
                                   scheme\\
    $\mathcal{C}_{l}$ & The set of maps of $\mathcal{C}$
                                   that have $l \in \mathcal{L}$ as a labeled
                                                          scheme\\
    $\mathcal{C}_{s}$ & The set of maps of $\mathcal{C}$
                                   that have $s \in \mathcal{U}$ as an
                                                          unlabeled
                                                          scheme\\
    $\mathbf{E}_m$ & The set of edges of a map $m$\\
    $\mathbf{F}_m$ & The set of faces of a map $m$\\
    $\mathbf{V}_m$ & The set of vertices of a map $m$\\
    $\mathbf{C}_m$ & The set of corners of a map $m$\\
    $\vec{\mathbf{C}}_m$ & The set of oriented corners of a map $m$\\
    $\sigma_m$ & The vertex rotation\\
    $\vec\theta_m$ & The face rotation\\
    ${}^t\vec{\mathbf{C}}_m$ & The set of tour corners (the corners
    of the form $\vec{\theta}_m^i(\vec \rho_m)$) of a blossoming map $m$\\
    $\preccurlyeq_m$ & The tour order -- the natural cyclic order on
    ${}^t\vec{\mathbf{C}}_m$ starting from the root \\
    $\vec{\mathbf{C}}^s_m$ & The set of oriented corners that are
                             followed by a stem\\
        $\vec{\mathbf{C}}^{\uparrow}_m$ & The set of oriented corners that are
    followed by a bud
                                   \end{tabularx}
    \end{minipage}%
    \begin{minipage}{0.45\textwidth}
        \centering
        \includegraphics[]{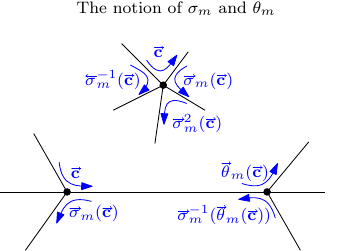}
                \centering
                \includegraphics[]{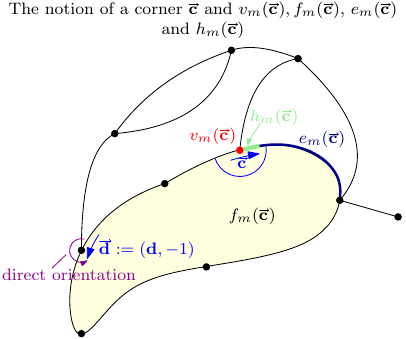}
                                        \centering
                        \includegraphics[]{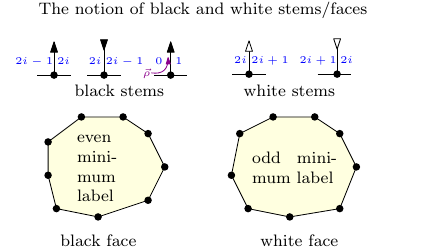}
                                \centering
        \includegraphics[]{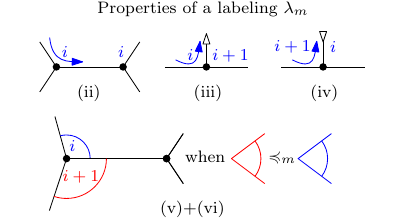}
      \end{minipage}%

      \begin{tabularx}{\textwidth}{ l X }
            $\vec{\mathbf{C}}^{\downarrow}_m$ & The set of oriented corners that are
    followed by a leaf\\
    $\vec{\mathbf{C}}^e_m$ & The set of oriented corners that are
    followed by an edge\\
    ${}^t\vec{\mathbf{C}}_m$ & The set of tour corners (the corners
    of the form $\theta_m^i(\vec\rho_m))$) of a map $m$\\
    $\lambda_m$ & The corner labeling\\
        $\gamma_m^{f\bullet}$ & The number of black leaves plus the number of black faces of $m$\\
    $\gamma_m^{f\circ}$ & The number of white leaves plus the number of white faces of $m$\\
    $\gamma_m^{r\bullet}$ & The number of black rootable stems (the
    root bud+leaves) of
    $m$\\
    $\gamma_m^{r\circ}$ & The number of white rootable stems (the root
    bud+leaves) of
    $m$
      \end{tabularx}

        \begin{center}{\bf Generating functions (from Sections
            3-6)}\end{center}

        \begin{align}
    B_\mathbb{S}(\mathbf{z},x,y) &\coloneqq \sum_{m\in\mathcal{B}_\mathbb{S}} \mathbf{z}^{\delta_m^v} x^{\gamma_m^{f\bullet}}y^{\gamma_m^{f\circ}},\\
    R_\mathbb{S}(\mathbf{z},x,y) &\coloneqq
                                   \sum_{m\in\mathcal{R}_\mathbb{S}}
                                   \mathbf{z}^{\delta_m^v}
                                   x^{\gamma_m^{r\bullet}}y^{\gamma_m^{r\circ}},\\
              B^\times_\mathbb{S}(x,y) &\coloneqq B_\mathbb{S}(\mathbf{z},x,y)\big|_{\mathbf{z}
                                  = (0,1,0,\dots)} =  \sum_{m\in\mathcal{B}^\times_\mathbb{S}}x^{\gamma_m^{f\bullet}}y^{\gamma_m^{f\circ}},\\
    R^\times_\mathbb{S}(x,y) &\coloneqq R_\mathbb{S}(\mathbf{z},x,y)\big|_{\mathbf{z}
                                  = (0,1,0,\dots)} = 
                               \sum_{m\in\mathcal{R}^\times_\mathbb{S}}x^{\gamma_m^{r\bullet}}y^{\gamma_m^{r\circ}},\\
          R_{\overline{s}}(\zz,x,y) &\coloneqq\sum_{m\in\mathcal{R}_{\overline{s}}}x^{\gamma_m^{r\bullet}}y^{\gamma_m^{r\circ}},\qquad
          s \in \mathcal{U}_\mathbb{S}\\
          C_{\overline{s}}(x,y) &\coloneqq \sum_{m\in\mathcal{C}_{\overline{s}}}
                                  x^{\gamma_m^{r\bullet}}y^{\gamma_m^{r\circ}},\qquad
          s \in \mathcal{U}_\mathbb{S}\\
                    C_{s}(x,y) &\coloneqq \sum_{m\in\mathcal{C}_{s}}
                                            x^{\gamma_m^{r\bullet}}y^{\gamma_m^{r\circ}},\qquad
                                 s \in \mathcal{U}_\mathbb{S}\\
                              C_{l}(x,y) &\coloneqq \sum_{m\in\mathcal{C}_{l}}
                                            x^{\gamma_m^{r\bullet}}y^{\gamma_m^{r\circ}},\qquad
          l \in \mathcal{L}_\mathbb{S} \\
C_{\overline{s}}^{\tiny
  \LEFTcircle}(x,y) &\coloneqq
C_{\overline{s}}(x,y)+C_{\overline{s}}(x,y),\qquad
          s \in \mathcal{U}_\mathbb{S}\\
          C_{s}^{\tiny
  \LEFTcircle}(x,y) &\coloneqq C_{s}(x,y)+C_{s}(x,y),\qquad
                      s \in \mathcal{U}_\mathbb{S}\\
                    C_{l}^{\tiny
  \LEFTcircle}(x,y) &\coloneqq C_{l}(x,y)+C_{l}(x,y),\qquad
          l \in \mathcal{L}_\mathbb{S}.
\end{align}

  \begin{center}{\bf Non-standard definitions}\end{center}
  \begin{itemize}
    \item A \textcolor{blue}{blossoming map} $m$ -- a map with additional single
      (meaning that they are not matched to another halfedge to form
      an edge) oriented halfedges, called \textcolor{blue}{stems}. Outgoing stems are called \textcolor{blue}{buds}, ingoing stems are
called \textcolor{blue}{leaves} and the number of buds is equal to the number of
leaves (exceptions to this rule are indicated by the term: \textcolor{blue}{unbalanced blossoming
  maps}). $\vec{\mathbf{C}}^{\uparrow}_m$, $\vec{\mathbf{C}}^{\downarrow}_m$,
$\vec{\mathbf{C}}^e_m$ -- the sets of oriented corners that are followed by a
bud, leaf and edge, respectively. The \textcolor{blue}{interior map} of a
blossoming map $m$, denoted $m^\circ$, is the map obtained from $m$ by
removing all its stems and
the \textcolor{blue}{interior degree} of a vertex $v$, denoted $\delta_m^\circ(v)$, is the degree of $v$ in $m^\circ$.
\item A map $m$ is \textcolor{blue}{unicellular} if it has a unique face. ${}^t\vec{\mathbf{C}}_m$ -- the set of \textcolor{blue}{tour corners} (of the
form $\vec{\theta}_m^i(\vec \rho_m)$) of a unicellular map $m$
equipped with a natural cyclic order $\preccurlyeq_m$, for which the
corner following the root corner $\vec\rho_m$ is the first one. When $m$ is blossoming then
${}^t\vec{\mathbf{C}}_m^{\uparrow}$, ${}^t\vec{\mathbf{C}}_m^{\downarrow}$, and
${}^t\vec{\mathbf{C}}^e_m$ are the intersection of
${}^t\vec{\mathbf{C}}_m$ with $\vec{\mathbf{C}}^{\uparrow}_m$, $\vec{\mathbf{C}}^{\downarrow}_m$,
$\vec{\mathbf{C}}^e_m$, respectively.
\item A \textcolor{blue}{labeling} of $m$ -- a function $\lambda_m : \mathbf{C}_m
  \to \mathbb{Z}$. 
\[\begin{array}{llcll}
    \mathrm{(i)}&\lambda_m(\vec\rho_m)&=&0,&\\
    \mathrm{(ii)}&\lambda_m(\vec\theta_m(\vec\cc))&=&\lambda_m(\vec\cc), &\text{ if }\vec\cc \in {}^t\vec{\mathbf{C}}^e_m,\\
    \mathrm{(iii)}&\lambda_m(\vec\theta_m(\vec\cc))&=&\lambda_m(\vec\cc)+1, &\text{ if }\vec\cc \in {}^t\vec{\mathbf{C}}^{\uparrow}_m,\\
    \mathrm{(iv)}&\lambda_m(\vec\theta_m(\vec\cc))&=&\lambda_m(\vec\cc)-1,
                                           &\text{ if }\vec\cc \in
                                             {}^t\vec{\mathbf{C}}^{\downarrow}_m,\\
     \mathrm{(v)}&\lambda_m(\sigma_m(\cc))&=&\lambda_m(\cc)-1, &\text{ if } c \in \mathbf{C}^e_m \text{ and } c \preccurlyeq_m \sigma_m(\cc),\\
   \mathrm{(vi)}&\lambda_m(\sigma_m(\cc))&=&\lambda_m(\cc)+1, &\text{ if } c \in \mathbf{C}^e_m \text{ and } \sigma_m(\cc) \preccurlyeq_m c.
  \end{array}\]

Various types of labelings:

    \begin{tabular}{l|l}   
\hline
      \textcolor{blue}{Corner labeling} & $\mathrm{(i)+(ii)+(iii)+(iv)}$\\
      \textcolor{blue}{Almost-decent labeling} & $\mathrm{(iii)+(iv)+(v)+(vi)}$\\
      \textcolor{blue}{Decent labeling} & $\mathrm{(i)+(iii)+(iv)+(v)+(vi)}$\\
    \textcolor{blue}{Well-labeling} & $\mathrm{(i)+(ii)+(iii)+(iv)+(v)+(vi)}$.
    \end{tabular}
\vspace{5pt}
\item
  A \textcolor{blue}{well-blossoming} map $m \in \mathcal{B}_\mathbb{S}$ -- a
unicellular blossoming map whose corner labeling is a well-labeling.
\item A \textcolor{blue}{bud-rooted} map $m$ --  a
  well-blossoming map which is rooted in a bud ($\vec\rho_m\in{}^t\vec{\mathbf{C}}_m^{\uparrow}$).
  \item A \textcolor{blue}{well-rooted} map $m \in \mathcal{R}_\mathbb{S}$ -- a
    well-blossoming map whose corner labeling is non-negative. In
    particular $\mathcal{R}_\mathbb{S} \subset \mathcal{B}_\mathbb{S}$.
    \item A \textcolor{blue}{core} of $\mathbb{S}$ -- a unicellular map of
      $\mathbb{S}$ with no vertex of interior degree $1$. The
      \textcolor{blue}{interior core} of a unicellular map $m$ is the core
      obtained from $m$ by iteratively removing all vertices of degree
      $1$ (and their adjacent edge).
      \item A \textcolor{blue}{scheme} of $\mathbb{S}$ -- a unicellular map of
      $\mathbb{S}$ with no vertex of interior degree $2$. The
      \textcolor{blue}{interior scheme} of a core $c$ is the scheme obtained from
      $c$ by iteratively removing vertices of interior degree $2$, and
      merging their $2$ formerly adjacent edges. A vertex of $c$ is a
      \textcolor{blue}{scheme vertex} if it is also a vertex of the scheme of
      $c$.
            \item A \textcolor{blue}{scheme-rooted} $c \in \mathcal{C}_\mathbb{S}$ -- a bud-rooted core (that is
$c^\circ$ is a core) whose root vertex is a scheme vertex.
      \item A \textcolor{blue}{labeled scheme} $l \in \mathcal{L}_\mathbb{S}$ -- a bud-rooted unbalanced scheme (that is
        $l^\circ$ is a scheme) decorated with a decent labeling $\lambda_l$.
        \item A \textcolor{blue}{virtually-rooted} map -- a bud-rooted map $m$ with two marked buds (called
          \textcolor{blue}{virtual}): the one following $\vec\rho_m$, and the one
          preceding $\vec \rho_m$; and such that
          $\vec\sigma(\vec\rho_m)\notin{}^t\mathbf{C}_m^{\uparrow}$ and
          $\cev\sigma^{-1}(\vec\rho_m)\notin{}^t\mathbf{C}_m^{\uparrow}$.
          \item The \textcolor{blue}{rerooting of $m$ on $s$}, denoted
            $\Omega(m,s)$ -- the map obtained from $m$ by changing the
            root bud to a leaf, changing $s$ to a bud, and setting the
            root of $\Omega(m,s)$ to be the corner preceding $s$ in
            tour order.
            \item A stem $s$ of a bud-rooted map $m$ is
              \textcolor{blue}{well-rootable} if the rerooting of $m$ on $s$ is
              well-rooted.
              \item The \textcolor{blue}{unrooted map of $m$} denoted
                $\overline{m}$ -- the equivalence class of a
                bud-rooted map $m \in $ for root-equivalence (maps are
                \textcolor{blue}{root-equivalent} if they can be obtained one
                from another by a rerooting-on-a-rootable-corner, that
                is a corner which is not adjacent to a non-rootable
                stem).
                \item An \textcolor{blue}{unlabeled scheme} $s \in
                  \mathcal{U}_\mathbb{S}$ -- the map obtained from a
                  labeled scheme by forgetting its corner labeling.
                  \item The \textcolor{blue}{unrooted scheme} $\overline{s}$ associated to a well-rooted map $m\in\mathcal{R}$ -- the unrooted map of the unlabeled scheme of the rerooting on a scheme rootable corner of the pruned map of $m$. We denote $\mathcal{R}_{\overline{s}}$ (resp. $\mathcal{C}_{\overline{s}}$) the set of maps of $\mathcal{R}$ (resp. $\mathcal{C}$) that have $\overline{s}$ as an unrooted scheme, and $\mathcal{M}_{\overline{s}}$ the set of maps whose opening have $\overline{s}$ as an unrooted scheme.
                \item A \textcolor{blue}{height of a vertex} $v \in \mathbf{V}_m$:
    \[\lambda_m(v)=\min_{\substack{ \cc\in\mathbf{C}_m \\
          v_m(\cc)=v}}\lambda_m(\cc).\]
    \item A \textcolor{blue}{relative labeling}:
\[ \lambda^r_l(\cc) \coloneqq \lambda_l(\cc) -
  \lambda_l(v_l(\cc)). \]
\item The \textcolor{blue}{relative type} of a halfedge -- the minimum relative label adjacent to that halfedge.
\item The \textcolor{blue}{relative type} of a vertex (sometimes called type for
  short) -- the maximum relative type of a halfedge incident to this vertex.
\item An edge of a scheme $l$ is said to be \textcolor{blue}{balanced}
  (resp. \textcolor{blue}{shifted}) if it is made of halfedges of relative type
  $0$ (resp. of type $1$). An edge $uv$ is said to be \textcolor{blue}{offset
    toward $v$} if the halfedge $u$ has relative type $0$ and the
  halfedge $v$ has relative type $1$. The \textcolor{blue}{offset graph of $l$}
  is the directed graph made of the offset edges of~$l$. An
  \textcolor{blue}{offset cycle} is an oriented cycle of the offset graph. An
  \textcolor{blue}{offset loop} is an offset cycle of length $1$.
  \item A \textcolor{blue}{special labeled $4$-valent map} $l
\in\widetilde{\mathcal{L}}^\times_\mathbb{S}$ -- a bud-rooted
$4$-valent map equipped with a decent labeling such that
\begin{itemize}
\item every non-root vertex of type $1$ and internal degree at least
  $2$ is adjacent to at most one bud,
\item if $l$ is not scheme-rooted, then it is not virtually rooted either. 
\end{itemize}
\item A \textcolor{blue}{special unlabeled $4$-valent map} -- a map obtained from a
special labeled $4$-valent map by forgetting its height function, and its
corner labeling.
\end{itemize}

\bibliographystyle{amsalpha}
\bibliography{biblio2015}

\end{document}